\documentclass[11pt]{article}
\usepackage[utf8]{inputenc}
\usepackage{epsfig}
\usepackage{graphicx}
\usepackage{amsthm}
\usepackage{amsmath}
\usepackage{algorithm}
\usepackage{amsfonts}
\usepackage{color}
\usepackage{authblk}
\usepackage[authoryear]{natbib}
\usepackage[colorlinks=true,linkcolor=blue, citecolor=blue]{hyperref}
\usepackage[left=1in,right=1in,top=1in,bottom=1.5in]{geometry}
\usepackage{bm}
\usepackage{mathtools, stmaryrd}
\usepackage{xparse} \DeclarePairedDelimiterX{\Iintv}[1]{\llbracket}{\rrbracket}{\iintvargs{#1}}
\NewDocumentCommand{\iintvargs}{>{\SplitArgument{1}{,}}m}
{\iintvargsaux#1} %
\NewDocumentCommand{\iintvargsaux}{mm} {#1\mkern1.5mu..\mkern1.5mu#2}
\usepackage{array}
\usepackage{booktabs}
\usepackage{multirow}
\usepackage{subcaption}
\usepackage{float}
\usepackage{epstopdf}

\usepackage[font=small,skip=0pt]{caption}
\newcommand{\field}[1]{\mathbb{#1}}
\newcommand{\dd}{\mathrm{d}}
\newcommand{\hd}{\mathrm{h}}
\newcommand{\zb}{\bm{z}}

\newcommand{\op}{\mathrm{op}}
\newcommand{\sfy}{\mathsf{y}}

\newcommand{\wb}{\bm{w}}
\newcommand{\gb}{\bm{g}}

\newcommand{\ab}{\mathbf{a}}
\newcommand{\Ub}{\mathbf{U}}
\newcommand{\Vb}{\mathbf{V}}

\newcommand{\OO}{\mathrm{O}}
\newcommand{\oo}{\mathrm{o}}
\newcommand{\Mt}{\mathtt{M}}

\newcommand{\p}{\field{P}}

\newcommand{\E}{\field{E}}
\newcommand{\FF}{\mathcal{F}}

\numberwithin{equation}{section}

\def\argmin{\mathop{\mbox{argmin}}}
\theoremstyle{Conjecture} \theoremstyle{example}
\theoremstyle{remark} \theoremstyle{lemma}
\theoremstyle{definition} \theoremstyle{corol}
\theoremstyle{proposition} \theoremstyle{condition}
\newtheorem{theorem}{Theorem}[section]
\newtheorem{assumption}{Assumption}[section]
\newtheorem{example}{Example}[section]
\newtheorem{remark}{Remark}[section]
\newtheorem{lemma}{Lemma}[section]
\newtheorem{definition}{Definition}[section]

\DeclarePairedDelimiter\floor{\lfloor}{\rfloor}

\begin{document}

\date{}

\title{ Simultaneous Sieve Inference for Time-Inhomogeneous Nonlinear Time Series Regression} 

{\large
\author[1]{Xiucai Ding \thanks{E-mail: xcading@ucdavis.edu. }}  
\author[2]{Zhou Zhou \thanks{E-mail: zhou@utstat.utoronto.ca}}
\affil[1]{Department of Statistics, University of California, Davis}
\affil[2]{Department of Statistical Sciences, University of Toronto}
}
\maketitle


\begin{abstract}
In this paper, we consider the time-inhomogeneous nonlinear time series regression for a general class of locally stationary time series. On one hand, we propose sieve nonparametric estimators for the time-varying regression functions which can achieve the min-max optimal rate as in \cite{MR673642}. On the other hand, we develop a unified simultaneous inferential theory which can be used to conduct both structural and exact form testings on the functions. Our proposed statistics are powerful even under locally weak alternatives.  We also propose a multiplier bootstrapping procedure for practical implementation. Our methodology and theory do not require any structural assumptions on the regression functions and we also allow the functions to be supported in an unbounded domain. We also establish sieve approximation theory for 2-D functions in unbounded domain and a Gaussian approximation result for affine and quadratic forms for high dimensional locally stationary time series, which can be of independent interest. Numerical simulations and a real financial data analysis are provided to support our results. 
\end{abstract}

\section{Introduction} \label{sec1}
\def\theequation{1.\arabic{equation}}
\setcounter{equation}{0}
\counterwithin{equation}{section}

In many scientific endeavors, researchers are interested in understanding the specific functional form of a relationship   between variables. To name but a few, the estimation of household Engle curves \cite{MR1085835}, the estimation of demand functions \cite{10.1162/REST_a_00636}, the estimation of an asset's price \cite{10.2307/30034997,10.2307/20061188,KOO2021562}, the understanding of the functional relation between  genetics  and diseases \cite{Wittkowski2010NonparametricMF} and the relationship between social and economic factors in social studies \cite{social}. There are usually two main challenges associated with these applications. First, the underlying theory does not narrow down a specific functional or parametric form for the relationship.  Second, the observations span long time horizons and constantly involving. Consequently, the relationship is unlikely to remain stable. To address these issues, a general time-inhomogeneous nonparametric regression model offers a more flexible framework.    

Consider the following time-inhomogeneous nonlinear time series regression
\begin{eqnarray}\label{eq:model}
Y_{i,n}=\sum_{j=1}^r m_j(t_i,X_{j,i,n})+\epsilon_{i,n}, \quad i=1,2,\cdots,n, 
\end{eqnarray}
where $r>0$ is some given fixed integer. Here $\{X_{j,i,n}\}$ are the covariates which can be locally stationary and depend on each other,  $t_i=i/n$ and 
$\epsilon_{i,n}$ is general locally stationary time series (c.f. Assumption \ref{assum_models}) whose covariance depend on both time and the covariates and satisfies that $\E(\epsilon_{i,n}|X_{j,i,n})=0.$ Moreover,  $m_j$'s are some smooth functions that map the time and covariates to the conditional mean.  (\ref{eq:model}) has been employed in the literature, for instance, see \cite{KARMAKAR2021,MR4244697, MV,MV1,ZWW}. It provides a more flexible and general framework and a wide range of models can fit into it. For example, (\ref{eq:model}) has been studied extensively in the literature by imposing various structural assumptions on $m_j$'s.  Assuming that $m_j$ has a multiplicative separability form in the time $t$ and the covariate, the time varying linear model has been studied in \cite{CAI2007163, MR1964455, MR1804172,10.2307/4140562, MR2168993, MR3310530,MR2758526} among others and the general multiplicative form has been recently studied in \cite{CSW, HHY,10.2307/43305634}. Moreover, when $X_{j,i,n}=Y_{i-j,n}$ in (\ref{eq:model}), it becomes  the nonparametric and nonlinear autoregressive process \cite[Section 3.1]{MV1}.       
Finally,  we point out that (\ref{eq:model}) can be regarded as a discretized version of the Langevin equation
\begin{equation*}
\mathrm{d} Y_t= \sum_{j=1}^r  m_j(t, X_{j,t}) \mathrm{d} t+\sigma(t, X_{1,t}, \cdots, X_{r,t})\mathrm{d} B(t),   
\end{equation*}
where $B(t)$ is some martingale.

In what follows, we summarize some related results in the literature in Section \ref{sec_subrelatedresults} and provide an overview of our results and state the novelties in Section \ref{sec_overview}.      

\subsection{Some related results}\label{sec_subrelatedresults}
In this subsection, we summarize some related results. In the literature, the general model setup (\ref{eq:model}) has been studied based on the nonparametric kernel estimation. In \cite{MV}, under the classical strong mixing conditions and using the locally stationary process as in \cite{DRW}, the author proposed a Nadaraya-Watson (NW) type estimator utilizing a product kernel. Moreover, the author also established the convergence rates for their estimation on any compact set and the point-wise normality for their estimators. Moreover,  the results and ideas have been used to study the structural changes problem in \cite{MV1}. Later on,  in \cite{ZWW}, the authors advanced the kernel estimation theory under a general setting using the physical  representation \cite{MR2172215}. They also applied their results to study the model selection problems. Moreover, recently, in \cite{KARMAKAR2021}, under specific model setups,  the authors considered the statistical inference for the time-varying parameters by analyzing the negative conditional Gaussian likelihoods with kernel estimations. They were able to establish the simultaneous inference results but only in the time domain without considering the covariates. 


We also point out that (\ref{eq:model}) has also been studied in the literature provided some structural assumptions are imposed. On one hand, when the smooth functions are not time-varying, (\ref{eq:model}) has been extensively studied. Especially, the kernel based estimators and pointwise  normality have been provided in the literature; see \cite{MR1964455, RePEc:pup:pbooks:8355,linton_wang_2016,10.1214/07-AOS533}  among others,  and the sieve least square estimators and asymptotic normality has also been studied; see  \cite{MR3343791,CXH, CHEN2015447, MR3306927} among others. On the other hand, when the functions are time-varying, the kernel estimators and simultaneous inference have been studied for the time-varying coefficient linear regression model when the smooth functions are of the form $X_{j,i,n} \theta_j(t);$ see \cite{CAI2007163,MR1804172, 10.2307/4140562, MR3160574, MR3310530}  among others. Furthermore, researchers have generalized the linear regression model to the setting when $m_j=\rho_j(t) g(X_{j,i,n})$ has a multiplicative separability structure. For instance, the kernel estimators and asymptotic normality have been established in \cite{CSW, PEI2018286} and the spline estimators and their asymptotic normality have been studied in \cite{10.2307/43305634} for functional data and in \cite{HHY} for locally stationary time series, among others.    


In summary, most of the relevant literature focus on providing a kernel based estimator for the smooth functions in the general setting (\ref{eq:model}). It is known that due to modeling the variation with respect to both time and covariates, the convergence rates for the estimations can be slow. Moreover, the existing works only established  convergence rates on a compact domain which may prohibit the applications when the covariates have unbounded supports. Finally, the literature has only focused on  proving the point-wise convergence rates of the general smooth functions and they may not be useful for conducting inference on the smooth functions simultaneously.

Motivated by the above challenges, in the current paper,  we study (\ref{eq:model}) in great generality and optimality. Inspired by the applications in the literature, we consider two scenarios of (\ref{eq:model}):  (1). The time-inhomogeneous nonlinear time series model, i.e.,  $X_{j,i,n}=Y_{i-j,n}$; (2). The general regression setting, i.e.,  $ X_{j,i,n}$ are some  locally stationary stationary predictors.    On one hand,  we provide  sup-norm rate optimal nonparamteric sieve estimators for the functions $m_j'$s. On the other hand, we develop simultaneously inferential theory for these functions. Our results and methodologies do not require any structural assumptions on $m_j$'s and we also allow the supports of these functions to be unbounded. We give an overview of our results in Section \ref{sec_overview}.  

\subsection{Overview of our results and novelties}\label{sec_overview}
In this subsection, we give a heuristic overview of our main results. Throughout the paper, for the locally stationary time series, we use the general form of physical representation following \cite{DZ, MR2172215, ZWW, ZW} (c.f. Assumption \ref{assum_models}) which covers a wide range of commonly used locally stationary time series (c.f. Examples \ref{exam_linear} and \ref{exam_nonlinear}). The aim of this paper is to provide a systematic and unified estimation and inferential theory for the smooth functions in (\ref{eq:model}) without imposing additional structural assumption.   

For the estimation part, we propose the sieve nonparamatric  estimators for the smooth functions. However, unlike in the standard applications of sieve expansions where the smooth functions are supported on a bounded domain \cite{MR3343791, CXH, CHEN2015447, DZ, DZ2,MR3306927}, the covariates in (\ref{eq:model}) are usually supported in $\mathbb{R}$ and hence the sieve methods and the classic approximation theory cannot be applied directly. To address this issue, we use the mapped sieve basis functions \cite{MR2486453,unboundeddomain} which can monotonically and smoothly map the unbounded domain to a compact domain. Consequently, we can construct a hierarchical sieve basis functions as in (\ref{eq_basisconstruction}) and establish the uniform approximation theory for the smooth functions on unbounded domain; see Theorem \ref{thm_approximation} for more details. Based on the approximation theory, it suffices to estimate the expansion coefficients using one ordinary least square to obtain the sieve estimators as in (\ref{eq_proposedestimator}). The estimation procedure is computationally cheap and  the proposed estimators are theoretically consistent. Especially, under mild regularity conditions, for short-range dependent locally stationary time series, our estimators achieve the min-max optimal rate as obtained by Stone in  \cite{MR673642}; see Theorem \ref{thm_consistency} and Remark \ref{rmk_minimaxoptimal} for more details. 

Once the estimation theory is established, we can develop the inferential theory. First, we provide the simultaneous confidence regions (SCR) based on our sieve estimators (c.f. (\ref{eq_scrdefinition})). Then we apply the SCR to conduct structural testings on the smooth functions. For example, we can test whether the functions are time-invariant (c.f. Example \ref{exam_stationarytest}) or have multiplicative separability structure (c.f. Example \ref{exam_seperabletest}). Second, we propose a $L_2$ statistic to test  whether the smooth functions are equal to some given functions (c.f. (\ref{eq_nullhypotheis})). The key technical input is to establish a Gaussian approximation result for high dimensional locally stationary time series for the affine forms and quadratic forms. We prove such a result in Theorem \ref{thm_gaussianapproximationcase} and it can be of independent interest. Especially, when the temporal relation of the locally stationary time series decays fast enough, the approximate rate matches that of \cite{MR3571252} which is claimed to be the best known rate. Moreover, to obtain the critical values for the SCR, we establish the maximum deviation of a Gaussian process using the device of  volume of tubes  \cite{MR1056331,KS, MR1207215}. For practical implementation, we propose a multiplier bootstrap procedure as in \cite{10.1214/13-AOS1161, MR3174655} which is both theoretically sound and empirically accurate and powerful; see Theorems \ref{thm_boostrapping} and \ref{thm_boostrapping2}. Numerical simulations and real data analysis are provided to support our results and methodologies.  

Finally, we point out that even though we focus on the object of the physical form of locally stationary time series, our methodologies and results can be applied to the locally stationary time series considered in \cite{DRW,MV} or even the general non-stationary time series in \cite{DZ2}, after some minor modifications. Before concluding this subsection, we summarize the contributions of the current paper as follows.
\begin{enumerate}
\item[(1).] We propose the sieve estimators for the time-inhomogeneous nonlinear time series regression (\ref{eq:model}) using mapped sieve basis functions. Our methodology does not require any structural assumptions on the smooth functions and we allow the functions to be supported in an unbounded domain. Moreover, our estimation procedure only needs one OLS regression and our estimator achieves the min-max optimal rates under mild assumptions. 
\item[(2).] We establish a general and unified simultaneous inferential theory for the smooth functions in  (\ref{eq:model}). Our theory can be used to conduct both structural and exact form testings on the smooth functions. Moreover, the proposed statistics are powerful even under weak alternatives.
\item[(3).] We provide a multiplier bootstrap procedure to implement our methodologies in practice. The bootstrapping statistics can asymptotically mimic the distributions of the statistics proposed in the inferential theory. Moreover,  our bootstrap  method is robust and adaptive to the data set.      
\item[(4).] We establish a Gaussian approximation result for both affine and quadratic forms of high dimensional locally stationary time series. The result matches the best known rate under mild regularity assumptions. Moreover, we also calculate the critical values of the  maxima of the Gaussian random field whose randomness arising from a locally stationary time series. 
\end{enumerate}

The paper is organized as follows. In Section \ref{sec_modelassumption}, we introduce the locally stationary time series, the concrete settings for (\ref{eq:model}) and some technical assumptions.  In Section \ref{sec_estimation}, we propose the sieve estimators for the smooth functions and prove their theoretical optimality.  In Section \ref{sec_inference}, we develop the simultaneous inferential theory based on our sieve estimators and consider several important hypothesis testing problems. For practical applications, we propose a multiplier  bootstrapping strategy. In Section \ref{sec_numerical}, we conduct extensive numerical simulations to support the usefulness of our estimators and proposed bootstrapping statistics for statistical inference. A real data analysis is provided in Section \ref{sec_realdata}.  Technical proofs are deferred to Appendix \ref{sec_techinicalproof}. Some auxiliary lemmas are collected in Appendix \ref{sec_auxililarylemma} and a list of commonly used sieve basis functions is provided in Appendix \ref{sec_sieves}. 


\section{The models and assumptions}\label{sec_modelassumption}
In this section, we introduce the models, some related notations and assumptions. In the current paper, we consider the family of locally stationary time series for ${Y_i, X_{j,i}, \epsilon_i}$ as considered in \cite{DRW, ZW}. As mentioned in Section \ref{sec1}, we are primarily interested in two scenarios. We now state the models for both cases in Assumption \ref{assum_models}. For any random variable $Z \in \mathbb{R},$ we denote its $q$-norm by 
\begin{equation*}
\| Z\|_q=\left( \mathbb{E}|X|^q \right)^{1/q}.
\end{equation*}
For simplicity, we  write $\| Z\| \equiv \|Z\|_2.$ For two sequences of real values $\{a_n\}$ and $\{b_n\},$ we write $a_n=\OO(b_n)$ if $|a_n| \leq C |b_n|$ for some constant $C$ and $a_n=\oo(b_n)$ if $|a_n| \leq \epsilon_n |b_n|$ for some positive sequence $\{\epsilon_n\}$ that $\epsilon_n \downarrow 0$ as $n \rightarrow \infty.$ Moreover, if $a_n=\OO(b_n)$ and $b_n=\OO(a_n),$ we write $a_n \asymp b_n.$ For a sequence of random variables $\{X_n\}$ and a positive sequence of $\{C_n\},$ we use the notation $X_n=\OO_{\mathbb{P}}(C_n)$ to say that $X_n/C_n$ is stochastically bounded, and  $X_n=\oo_{\mathbb{P}}(C_n)$ to say that $X_n/C_n$ converges to $0$ in probability. Throughout the paper, we omit the subscript $n$ without causing further confusion.


\begin{assumption} \label{assum_models} In this paper, we suppose the following assumptions hold for the regression (\ref{eq:model}): \\
\noindent{\bf (1). Time series regression setting: $X_{j,i}=Y_{i-j}$.} In this setting, we assume that both $\{X_i\}$ and $\{\epsilon_i\}$ are locally stationary time series such that 
\begin{equation}\label{eq_epsilons1form}
X_i=G_1\left( t_i, \mathcal{F}_i \right), \ \epsilon_i=G_2\left(t_i,\FF_i\right),
\end{equation}
where $\FF_i=(\cdots,\eta_{i-1},\eta_i)$ and $\{\eta_i\}_{i\in\mathbb{Z}}$ are i.i.d. random variables with $\E(\epsilon_{i}|\FF_{i-1})=0$. Here $G_1, G_2: \mathbb{R}^{\infty} \times \mathbb{R}^{\infty} \rightarrow \mathbb{R}$ are measurable functions such that for any fixed $t \in [0,1],$ $G_k(t, \mathcal{F}_i), k=1,2,$ are properly defined random variables. Moreover, we assume that $G_k(t, \cdot)$ are stochastic Lipschitz functions of $t$ such that for some $q>2$, some constant $C>0$ and any $s, t \in [0,1],$
\begin{equation}\label{eq_slc}
\sup_i \left\| G_k(s, \mathcal{F}_i)-G_k(t, \mathcal{F}_i) \right\|_q \leq C|s-t|, \ k=1,2. 
\end{equation}

\noindent {\bf (2). General non-stationary regression setting: $X_{j,i}, 1 \leq j \leq r,$ are  locally stationary predictors.} In this setting, we assume that all $\{Y_i\}, \{X_i\}, \{\epsilon_i\}$ are locally stationary time series such that 
\begin{equation}\label{eq_setting2}
Y_{i}=G(t_i, \mathcal{F}_i), \ X_{j,i}=H_j(t_i, \mathcal{G}_i), \ \epsilon_i=D(t_i, \mathcal{L}_i),
\end{equation}
where $\mathcal{F}_i=(\cdots, \eta_{i-1},\eta_i), \ \mathcal{G}_i=(\cdots, \gamma_{i-1}, \gamma_i)$ and $\mathcal{L}_i=(\cdots,(\eta_{i-1},\gamma_{i-1})^\top,(\eta_{i-1},\gamma_{i-1})^\top)$ with $\E(\epsilon_{i}|{\cal G}_i)=0.$ Here we assume that    $\{(\eta_i,\gamma_i)^\top\}_{i\in\mathbb{Z}}$ are i.i.d. random elements, $G, H_j, 1 \leq j \leq r, D$ are measurable functions satisfying the stochastic Lipschitz continuous condition as in (\ref{eq_slc}).

\end{assumption} 

We remark that the physical form (\ref{eq_epsilons1form}) is a very general representation in the sense that many locally stationary time series, linear or non-linear, can be written in this form. For illustration, we refer the readers to Examples \ref{exam_linear} and \ref{exam_nonlinear} below. In the current paper, we focus on the practical setting that only one time series is available, i.e., only one realization of  $\{X_i\}$ for Case (1) and $\{(Y_i, X_i)\}$ for Case (2) is observed. In this setting, it is natural to consider the time series with short-range temporal dependence. We employ the \emph{physical dependence measure} to quantify the temporal dependence of the time series.  For simplicity, we state the definition using $\{X_i\}$ as in (\ref{eq_epsilons1form}). 

\begin{definition}[Physical dependence measure] Let $\{\eta_i'\}$ be an i.i.d. copy of $\{\eta_i\}.$ Assume that for some $q>2,$ 
\begin{equation*}
\| X_i\|_q<\infty.
\end{equation*} 
Then for $j \geq 0,$ we define the physical dependence measure of $\{X_i\}$ as
\begin{equation*}
\delta_1(j,q)=\sup_t \| G_1(t, \mathcal{F}_0), G_1(t, \mathcal{F}_{0,j}) \|_q, 
\end{equation*}
where $\mathcal{F}_{0,j}=(\mathcal{F}_{-j-1}, \eta_{-j}', \eta_{-j+1}, \cdots, \eta_0).$ For convenience, we denote $\delta_1(j,q)=0$ when $j<0.$ 
\end{definition}

The physical dependence measure provides a convenient tool to study the temporal dependence of time series. On one hand, it quantifies the magnitude of change in the system's output when the input of the system $j$ steps ahead is replaced by an i.i.d. copy. On the other hand, the strong temporal dependence of the locally stationary time series can be controlled in terms of $\delta_1(j,q)$ and the concentration inequalities (c.f. Lemmas \ref{lem_concentration} and \ref{lem_mdependent}) can be established based on them. For more details on physical dependence measure, we refer the readers to \cite{MR2485027,MR3114713, MR2172215,WUAOP, MR2827528}. Moreover, for the purpose of illustration, in Examples \ref{exam_linear} and \ref{exam_nonlinear} below, we explain how the physical dependence measure can be calculated easily for the commonly used examples of locally stationary time series.  

Throughout the paper, we impose the following assumption to  ensure that the temporal dependence decays fast enough so that the locally stationary time series has short-range dependence. 
\begin{assumption}\label{assum_physical}
Suppose that there exists some constant $\tau>0$ and some constant $C>0$ such that for Case (1) of Assumption \ref{assum_models}
\begin{equation}\label{eq_physcialrequirementone}
\max\{\delta_1(j,q),\delta_2(j,q) \} \leq C j^{-\tau}, \ \text{for all} \ j \geq 1, 
\end{equation}
where $\delta_1, \delta_2$ are the physical dependence measures for $\{X_i\}$ and $\{\epsilon_i\},$ respectively. Moreover, for Case (2) of Assumption \ref{assum_models}, we assume that 
\begin{equation*}
\max_{0 \leq k \leq r+1} \delta_k(j,q) \leq C j^{-\tau}, \text{for all} \ j \geq 1,
\end{equation*}
where $\delta_0$ is the physical dependence measure of $\{Y_i\},$ $\delta_{r+1}$ is that of $\{\epsilon_i\}$ and $\delta_k, 1 \leq k \leq r,$ are those of $\{X_{k,i}\},  1 \leq k \leq r.$
\end{assumption} 

Before concluding this subsection, we provide two examples of locally stationary time series.  
\begin{example}[Locally stationary linear time series]\label{exam_linear} Let $\{\epsilon_i\}$ be i.i.d. random variables and $\{a_j(t)\}$ be continuously differentiable functions on $[0,1].$ Consider the locally stationary linear process such that 
\begin{equation*}
G(t, \mathcal{F}_i)=\sum_{k=0}^{\infty} a_k(t) \epsilon_{i-k}.
\end{equation*} 
It is easy to see that the physical dependence measure of the above linear process, denoted as $\delta(j,q)$ satisfies that
\begin{equation*}
\delta(j,q)=\OO(\sup_t |a_j(t)|). 
\end{equation*} 
Consequently, according to \cite[Example 2.4]{DZ}, (\ref{eq_slc}) and Assumption \ref{assum_physical} are satisfied if the following conditions hold 
\begin{equation*}
\sup_t|a_j(t)| \leq Cj^{-\tau}, \ \ \sum_{j=0}^{\infty} \sup_{t \in [0,1]} |a_j'(t)|^{\min\{2,q\}}<\infty. 
\end{equation*} 
\end{example}

\begin{example}[Locally stationary nonlinear time series]\label{exam_nonlinear} Let $\{\epsilon_i\}$ be i.i.d. random variables. Consider a locally stationary process as follows 
\begin{equation*}
G(t, \mathcal{F}_i)=R(t, G(t, \mathcal{F}_{i-1}), \epsilon_i),
\end{equation*}
where $R$ is some measurable function. The above expression is quite general such that many locally stationary time series can be written in the above form. For example, the threshold autoregressive model, exponential autoregressive model and bilinear autoregressive models. According to \cite[Theorem 6]{ZW}, suppose that for some $x_0,$ $\sup_t\|R(t,x_0,\epsilon_i) \|_q<\infty.$ Denote 
\begin{equation*}
\chi:=\sup_t L(t), \ \text{where} \ L(t):=\sup_{x \neq y} \frac{\|R(t,x, \epsilon_0)-R(t,y,\epsilon_0) \|_q}{|x-y|}.
\end{equation*}
Note that $\chi<1$  and $\delta(j,q)=\OO(\chi^j).$ Therefore, Assumption \ref{assum_physical} is satisfied. Moreover, by \cite[Proposition 4]{ZW} (\ref{eq_slc}) holds if 
\begin{equation*}
\sup_{t \in [0,1]} \| M(G(t, \mathcal{F}_0)) \|_q<\infty,  \ \text{where} \ M(x):=\sup_{0\leq t<s \leq 1}\frac{\|R(t,x,\epsilon_0)-R(s,x,\epsilon_0) \|_q}{|t-s|}. 
\end{equation*}


As a concrete example, we consider the time varying threshold autoregressive (TVTAR) model denoted as 
\begin{equation*}
G(t, \mathcal{F}_i)=a(t)[G(t, \mathcal{F}_{i-1})]^++b(t)[-G(t, \mathcal{F}_{i-1})]^+.
\end{equation*}
In this case, Assumption \ref{assum_physical} and (\ref{eq_slc}) will be satisfied if 
\begin{equation*}
a(t), b(t) \in \mathsf{C}^1([0,1]), \ \text{and} \ \sup_{t \in [0,1]}\left( |a(t)|+|b(t)| \right)<1. 
\end{equation*}
\end{example}

\section{Sieve estimation for time varying nonlinear regression}\label{sec_estimation}

In this section, we propose a nonparametric  approach to estimate the nonlinear functions $m_j(t,x), \ 1 \leq j \leq r. $

\subsection{Sieve least square estimation}
In this subsection, we use the sieve least square estimation approach to obtain the estimators for $m_j(t,x),  1 \leq j \leq r.$ Due to similarity, we mainly focus on explaining Case (1) of Assumption \ref{assum_models} and only briefly discuss Case (2) from time to time; see Remark \ref{rmk_case2ols} below for some discussion. Our methodology is nonparametric and utilizes    the sieve basis expansion; see \cite{CXH} for a comprehensive review.  We next provide a heuristic overview of the method and the details will be offered in Sections \ref{sec_2dsieves} and \ref{sec_olssieveestimation}.

 Since we can only observe one realization of the time series, it is natural to impose some smoothness condition on $m_j(t,x)$ (c.f. Assumption \ref{assum_smoothnessasumption}). First, under these conditions, it suffices to consider $m_{j,c,d}(t,x)$ (c.f. Theorem \ref{thm_approximation}) denoted as 
\begin{equation}\label{eq_firststeptruncation1}
m_{j, c,d}(t,x)=\sum_{\ell_1=1}^c \sum_{\ell_2=1}^d \beta_{j, \ell_1, \ell_2} b_{\ell_1, \ell_2}(t,x),
\end{equation} 
where $\{b_{\ell_1, \ell_2}(t,x)\}$ are the sieve basis functions and $\{\beta_{j, \ell_1, \ell_2}\}$ are the coefficients to be estimated. Here $c \equiv c(n), \ d \equiv d(n)$ depend on the smoothness of $m_j(t,x).$ For detailed discussion of the construction of sieve basis functions and the performance of (\ref{eq_firststeptruncation1}), we refer the readers to Section \ref{sec_2dsieves}. Second, in view of (\ref{eq_firststeptruncation1}), given a set of sieve basis functions, it suffices to estimate the coefficients. We will use the ordinary least square (OLS) to achieve this goal and obtain the estimators $\{\widehat{\beta}_{j, \ell_1, \ell_2}\}$. This aspect will be discussed in Section \ref{sec_olssieveestimation}. Based on the above calculation, we can obtain our sieve least square estimation as follows 
\begin{equation}\label{eq_firststeptruncation}
\widehat{m}_{j, c,d}(t,x)=\sum_{\ell_1=1}^c \sum_{\ell_2=1}^d \widehat{\beta}_{j, \ell_1, \ell_2} b_{\ell_1, \ell_2}(t,x). 
\end{equation}   
 
\subsubsection{Mapped hierarchical sieves for 2-D smooth function on unbounded domain}\label{sec_2dsieves}
In this subsection, we discuss how to use sieve basis functions to approximate a smooth 2-D function. Note that for $m_j(t,x), 1 \leq j \leq r,$ when $x$ is defined on a compact domain in $\mathbb{R}$, the results have been established, for example, see \cite[Section 2.3.1]{CXH} for a comprehensive review. However, in the real applications, the domain of $x$ is usually unbounded and mostly the whole real line. Therefore, we need to modify the commonly used sieves to accommodate for practical applications. 

In the literature, there exist three different approaches to deal with the unbounded domain. The first  method is to apply a finite interval method, such as Chebyshev polynomials, to an interval $x \in [-L,L]$ where $L$ is large but
finite. This method is also known as domain truncation. The second way is to use a basis that is intrinsic to the unbounded domain such as Hermite functions. The third approach is to map the infinite interval into a finite domain through a change of coordinate and then apply a finite interval method. We refer the readers to \cite{MR1874071,MR2486453,unboundeddomain} for a review.

In the current paper, we employ the third method that we will apply some suitable mappings to map the unbounded domain to a compact one. The main reason is because the first and second methods usually result in slow approximate rates so that we need to use more basis functions in order to achieve certain accuracy. The mapping idea is popular in numerical PDEs for handling the boundary value problems. For a review of the mapping strategy, we refer the readers to \cite{MR2486453,unboundeddomain}. In what follows, without loss of generality, we assume that the domain for $x$ is either $\mathbb{R}$ or $\mathbb{R}_+.$ The mappings will map the domain to a compact interval, say, $[-1,1].$ We state the definitions of the mappings as follows.  
\begin{definition}[Mappings]\label{defn_mappings}
For some positive scaling factor $s>0,$ consider a family of mappings of the form:
\begin{equation*}
x=g(y;s), \ s>0, \ y \in I:=(-1,1), \ x \in \Lambda:=\mathbb{R}_+ \ \text{or} \ \mathbb{R},
\end{equation*}
such that 
\begin{align*}
& \frac{\dd x}{\dd y}=g'(y;s)>0, \ y \in I; \\
& g(-1; s)=
\begin{cases}
0, & \Lambda=\mathbb{R}_+ \\
-\infty, & \Lambda=\mathbb{R}
\end{cases}, \ 
g(1; s)=\infty.
\end{align*}
Since the above mapping is invertible, we denote the inverse mapping by 
\begin{equation*}
y=g^{-1}(x;s)=u(x;s), \ x \in \Lambda, y \in I, s>0.
\end{equation*}
\end{definition}
Based on the above definition, it is easy to see that $\frac{1}{2}u(x;s)+\frac{1}{2}$ will map $\mathbb{R}$ or $\mathbb{R}_+$ to $[0,1].$ Before proceeding to construct the sieve basis functions, 
we pause to list a few examples of the mappings for illustration. 
\begin{example}\label{example_mappings} In this example, we provide some mappings when $\Lambda=\mathbb{R}.$ The first mapping is the following algebraic mapping 
\begin{equation}\label{eq_y}
x=\frac{sy}{\sqrt{1-y^2}}, \ y=\frac{x}{\sqrt{x^2+s^2}}.
\end{equation}
The second mapping is the following logarithmic mapping
\begin{equation}\label{eq_y1}
x=\frac{s}{2}\ln \frac{1+y}{1-y}, \  y=\tanh(s^{-1}x)=\frac{e^{s^{-1} x}-e^{-s^{-1}x}}{e^{s^{-1}x}+e^{-s^{-1}x}}.
\end{equation}
\end{example}
\begin{example} In this example, we provide some mappings when $\Lambda=\mathbb{R}_+.$ The first mapping is the following algebraic mapping 
\begin{equation*}
x=\frac{s(1+y)}{1-y}, \ y=\frac{x-s}{x+s}.
\end{equation*}
The second mapping is the following logarithmic mapping 
\begin{equation*}
x=\frac{s}{2} \ln \frac{3+y}{1-y}, \ y=1-2\tanh(s^{-1}x). 
\end{equation*}
\end{example}

Throughout the paper, without loss of generality, we assume that $m_j(t,x), 1 \leq j \leq r,$ takes value on the whole real line $\mathbb{R}$ for $x.$ Similar arguments apply when $x \in \mathbb{R}^+.$ Using the mappings in Example \ref{example_mappings}, due to  monotonicity,  we have related $m_j(t,x): [0,1] \times \mathbb{R} \rightarrow \mathbb{R} $ to 
\begin{equation}\label{eq_transformedmjty}
\widetilde{m}_j(t, y):=m_j(t, g(2y-1;s)): [0,1] \times [0,1] \rightarrow \mathbb{R}. 
\end{equation} 
Based on the above arguments, we can construct a sequence of mapped sieve basis functions following \cite{unboundeddomain}. Let $\{\phi_i(\cdot)\}$ be an orthonormal base of the smooth functions defined on $[0,1].$ Denote the mapped version of $\{\phi_i\}$ as $\{\widetilde{\phi}_i\}$ such that for $x \in \mathbb{R}$
\begin{equation}\label{eq_defnmappedbasis}
\widetilde{\phi}_i(x)=\phi_i \circ \sfy(x), \ \sfy(x):=\frac{u(x;s)+1}{2}. 
\end{equation} 
Note that $\{\widetilde{\phi_i}\}$ is a sequence of orthogonal polynomials of the functional space defined on $\mathbb{R}$ \cite{unboundeddomain}. Denote $\{\varphi_i\}$ as the orthonormal basis functions based on $\{\widetilde{\phi_i}\}.$ Armed with $\{\phi_i\}$ and $\{\varphi_i\},$ we can construct the hierarchical sieve basis functions as \cite{CXH}
\begin{equation}\label{eq_basisconstruction}
\{\phi_i(t)\} \otimes \{\varphi_j(x)\}. 
\end{equation} 
For examples of the commonly used sieve and mapped sieve basis functions and their properties, we refer the readers to  Appendix \ref{sec_sieves}. Let $\{b_{\ell_1, \ell_2}(t,x)\}$ be the collection of  sieve basis functions as in (\ref{eq_basisconstruction}). Recall (\ref{eq_firststeptruncation1}). 








\begin{assumption}\label{assum_smoothnessasumption}
For some constants $\mathsf{m}_k, k=1,2,$ denote $\mathtt{C}^{\mathsf{m}_k}([0,1])$ as the function space on $[0,1]$ of continuous functions that have continuous first $\mathsf{m}_k$ derivatives. We assume that for all $1 \leq j \leq r$ and $\widetilde{m}_j$ in (\ref{eq_transformedmjty})
\begin{equation*}
\widetilde{m}_j(t,\cdot) \in \mathtt{C}^{\mathsf{m}_1}([0,1]), \ \text{and} \  \widetilde{m}_j(\cdot, y) \in \mathtt{C}^{\mathsf{m}_2}([0,1]).  
\end{equation*}
\end{assumption}


\begin{theorem}\label{thm_approximation}
Suppose Assumption \ref{assum_smoothnessasumption} holds. Moreover, we choose $\{\phi_i(t)\}$ to be Fourier series, Jacobi polynomials and orthogonal wavelets, and $\{\varphi_j(x)\}$ to be mapped Fourier series, mapped Jacobi polynomials,  and mapped orthogonal wavelets as in Examples \ref{examplebasis_fourier}--\ref{examplebasis_wave}. For $m_{j,c,d}(t,x)$ defined in (\ref{eq_firststeptruncation1}),  we have that for all $ 1 \leq j \leq r,$
\begin{equation*}
\sup_{t \in [0,1]} \sup_{x \in \mathbb{R}}\left| m_j(t,x)-m_{j, c,d}(t,x) \right|=\OO\left( c^{-\mathsf{m}_1}+d^{-\mathsf{m}_2} \right). 
\end{equation*}
\end{theorem}

Theorem \ref{thm_approximation} establishes the asymptotic approximation results for a 2-D function on an unbounded domain using the mapped sieve basis functions. Similar results have been established for the compact domain; see \cite[Section 2.3.1]{CXH} for a comprehensive summary. In fact, when the domain is compact, we can directly use the basis functions $\{\phi_i(t)\} \otimes \{\phi_j(x)\}$ instead of (\ref{eq_basisconstruction}).  

%
%

\subsubsection{OLS estimation for the coefficients}\label{sec_olssieveestimation}
Armed with the basis expansion in Section \ref{sec_2dsieves} and the approximation result Theorem \ref{thm_approximation}, we see that it suffices to estimate the coefficients $\beta_{\ell_1, \ell_2}$ in (\ref{eq_firststeptruncation1}).  As mentioned earlier, we focus on the discussion of Case (1) of Assumption \ref{assum_models}.

 Since $r$ is bounded, in view of Theorem \ref{thm_approximation}, we can write    
\begin{equation*}
X_i=\sum_{j=1}^r m_{j,c,d}(t_i, X_{i-j})+\epsilon_i+\oo_{\mathbb{P}}(1), 
\end{equation*}
where $m_{j,c,d}(t,x)$ is defined in (\ref{eq_firststeptruncation1}). Consequently, the estimation of $m(t,x)$ reduces to estimating the coefficients $\{\beta_{j, \ell_1, \ell_2}\}.$ We apply OLS to estimate them. Denote the vector $\bm{\beta}=(\beta_1, \cdots, \beta_{rcd})^\top \in \mathbb{R}^{rcd}$ which  collects all these coefficients $\{\beta_{j, \ell_1, \ell_2}\}$ in the order of the indices $\Iintv{1,r} \times \Iintv{1,c} \times \Iintv{1,d};$ that is to say, for $1 \leq k \leq rcd,$
\begin{equation}\label{eq_indices}
\beta_k=\beta_{l+1, \ell_1, \ell_2+1}, \  l=\floor*{\frac{k}{cd}}, \ \ell_2=\floor*{\frac{k-lcd}{c}}, \ \ell_1=k-lcd-\ell_2d, 
\end{equation} 
where  $\floor*{\cdot}$ is the floor of a given real value. Denote $W \in \mathbb{R}^{(n-r) \times rcd}$ as the design matrix whose entry satisfies that
\begin{equation}\label{eq_designmatrix}
W_{ik}=\phi_{\ell_1}(t_i) \varphi_{\ell_2+1}(X_{i-(l+1)}),
\end{equation}
where $\ell_1$ and $\ell_2$ are defined in $(\ref{eq_indices}).$ Furthermore, we denote $\bm{Y}=(Y_i)_{r+1 \leq i \leq n} \in \mathbb{R}^{n-r}$ and $\bm{\epsilon}=(\epsilon_i)_{r+1 \leq i \leq n} \in \mathbb{R}^{n-r}.$ Recall that for Case (1), $Y_i=X_i.$  Armed with these notations, the OLS estimator for $\bm{\beta}$ is denoted as
\begin{equation}\label{eq_betaolsform}
\widehat{\bm{\beta}}=(W^\top W)^{-1} W^\top \bm{Y}.
\end{equation}  

For $1 \leq j \leq r,$ denote the diagonal matrix $\mathsf{I}_j \in \mathbb{R}^{(rcd) \times (rcd)}$ such that $\mathsf{I}_j=\oplus_{k=1}^ r \delta_k(j) \mathbf{I}_{cd},$ where $\oplus$ is the direct sum, $\delta_k(j)=1$ when $k=j$ and $0$ otherwise, and $\mathbf{I}_{cd}$ is the $cd \times cd$ identity matrix. Furthermore,  let $\bm{b} \in \mathbb{R}^{cd}$ be the collection of the basis functions $\{b_{i,j}(t,x)\}_{\{1 \leq i \leq c, 1 \leq j \leq d \}}.$ Then our proposed estimator for $m_{j}(t,x)$ is
\begin{equation}\label{eq_proposedestimator}
\widehat{m}_{j,c,d}(t,x)=(\widehat{\bm{\beta}} \mathsf{I}_j)^\top (\bm{b} \otimes \mathbf{I}_r), 
\end{equation}
where $\otimes$ is the Kronecker product and $\mathbf{I}_r$ is the $r \times r$ identity matrix. 

As will be seen in Section \ref{sec_theorecticalpropertyestimator}, $\widehat{m}_{j,c,d}(t,x)$ is a consistent estimator for $m_j(t,x).$


\subsection{Theoretical properties of the proposed estimators}\label{sec_theorecticalpropertyestimator}

In this subsection, we prove the properties of the proposed estimator (\ref{eq_proposedestimator}). We first prepare some notations and assumptions. Let $\bm{w}_i=(w_{ik})_{\{1 \leq k \leq rd\}} \in \mathbb{R}^{rd}, r+1 \leq i \leq n,$ such that
\begin{equation}\label{eq_desigmatrixform}
w_{ik}=\varphi_{\ell_2(k)} (X_{i-\ell_1(k)-1}), 
\end{equation}
where $\ell_1(k)$ and $\ell_2(k)$ are defined as 
\begin{equation}\label{eq_notationindex}
\ell_1(k):=\floor*{\frac{k}{d}}, \ \ell_2(k)=k-\ell_1(k)d.  
\end{equation}
Moreover, we denote $\bm{x}_i=(x_{ik})_{\{1 \leq k \leq rd\}} \in \mathbb{R}^{rd}, r+1 \leq i \leq n,$  as 
\begin{equation}\label{eq_ddd}
x_{ik}=w_{ik} \epsilon_i.
\end{equation}
In the following lemma, we show that both the high dimensional time series $\{\bm{w}_i\}$ and $\{\bm{x}_i\}$ can be regarded as short-range dependent locally stationary stationary time series. 
\begin{lemma}\label{lem_locallystationaryform}
Suppose (1) of Assumptions \ref{assum_models} and Assumption \ref{assum_physical} hold. Then there exist measurable functions $\mathbf{V}(t, \cdot)=(V_1(t,\cdot), \cdots, V_{rd}(t, \cdot)), \ \mathbf{U}(t, \cdot)=(U_1(t,\cdot), \cdots, U_{rd}(t,\cdot)) \in \mathbb{R}^{rd}$ satisfying the stochastic Lipschitz continuity as in (\ref{eq_slc}) such that $\{\bm{x}_i\}$ is mean-zero and
\begin{equation}\label{eq_locallystationaryform}
\bm{w}_i=\Vb(t_i, \mathcal{F}_i), \ \bm{x}_i=\Ub(t_i, \mathcal{F}_i), \ t_i=\frac{i}{n}.
\end{equation} 
Moreover, denote their physical dependence measures as
\begin{equation}\label{eq_deltaxdefinition}
\delta_{w}(j,q)=\sup_{1 \leq k \leq rd} \sup_t\| V_{k}(t, \mathcal{F}_0)- V_{k}(t, \mathcal{F}_{0,j})\|_q, \ \delta_{x}(j,q)=\sup_{1 \leq k \leq rd} \sup_t\| U_{k}(t, \mathcal{F}_0)- U_{k}(t, \mathcal{F}_{0,j})\|_q.
\end{equation} 
Then we have that for some constant $C>0$
\begin{equation}\label{eq_transferbound}
\max\{\delta_w(j,q),\delta_x(j,q) \} \leq C j^{-\tau}, \ \text{for all} \ j \geq 1, 
\end{equation}
Similar results hold when (2) of Assumption \ref{assum_models} holds. 
\end{lemma}

Next, we define the long-run covariance matrices of $\{\bm{w}_i\}$ and $\{\bm{x}_i\}$ as $\Pi(t) \in \mathbb{R}^{rd \times rd}$ and $\Omega(t) \in \mathbb{R}^{rd \times rd},$ respectively, as
\begin{equation}\label{eq_longrunwitht}
\Pi(t):=\sum_{j=-\infty}^{+\infty} \operatorname{Cov} \left( \Vb(t, \mathcal{F}_0), \Vb(t, \mathcal{F}_j) \right), \ \Omega(t):=\sum_{j=-\infty}^{+\infty} \operatorname{Cov} \left( \Ub(t, \mathcal{F}_0), \Ub(t, \mathcal{F}_j) \right). 
\end{equation}
Throughout the paper, we will use the following short-hand notation 
\begin{equation}\label{eq_defnp}
p:=rcd. 
\end{equation}
Using the above notations, denote the integrated long run covariance matrices $\Pi \in \mathbb{R}^{p \times p}$ and $\Omega \in \mathbb{R}^{p \times p}$ as
\begin{equation}\label{eq_longruncovariancematrix}
\Pi:=\int_0^1 \Pi(t) \otimes (\ab(t) \ab^\top(t)) \dd t, \ \Omega:=\int_0^1 \Omega(t) \otimes (\ab(t) \ab^\top(t)) \dd t,
\end{equation}
where $\ab(t)=(\phi_1(t), \cdots, \phi_c(t))^\top.$

In the current paper, we will need the following regularity assumption on $\Pi$ and $\Omega.$ It is frequently used in the statistics literature, for instance, see \cite{MR3476606,MR3161455,DZ,MR2719856}. 
\begin{assumption}\label{assum_updc} For $\Pi$ and $\Omega$ in (\ref{eq_longruncovariancematrix}), we assume that there exists a universal constant $\kappa>0$ such that 
\begin{equation*}
\min\{\lambda_p(\Pi), \lambda_p(\Omega) \} \geq \kappa,
\end{equation*}
where $\lambda_p(\cdot)$ is the smallest eigenvalue of the given matrix. 
\end{assumption}



Armed with the above notations and assumptions, we state the results regarding the consistency of our proposed estimator (\ref{eq_proposedestimator}). 
Denote 
\begin{equation}\label{eq_defnxic}
\xi_c:=\sup_{1 \leq i \leq c} \sup_{t \in [0,1]} |\phi_i(t)|, \ \zeta=\sup_{t,x} \| \bm{b} \|,
\end{equation}
 where $\| \bm{b} \| \equiv \| \bm{b} \|_2$ is the $L_2$ norm of $\bm{b}.$


\begin{theorem}\label{thm_consistency}
Suppose Assumptions \ref{assum_models}-- \ref{assum_updc} hold. Moreover, we assume 
\begin{equation}\label{eq_parameterassumption}
p\left( \frac{\xi^2_c}{\sqrt{n}}+\frac{\xi^2_c n^{\frac{2}{\tau+1}}}{n}\right)=\oo(1).
\end{equation}
Then we have that 
\begin{equation}\label{eq_rate}
\max_j \sup_{t \in [0,1],x \in \mathbb{R}}\left\| m_j(t,x)-\widehat{m}_{j,c,d}(t,x) \right\|=\OO\left(\xi_c \zeta \sqrt{\frac{p}{n}}+c^{-\mathsf{m}_1}+d^{-\mathsf{m}_2} \right). 
\end{equation}
\end{theorem}

\begin{remark}\label{rmk_minimaxoptimal}
Theorem \ref{thm_consistency} implies that our proposed sieve least square estimator (\ref{eq_proposedestimator}) is consistent under mild  conditions.  First, $\xi_c$ and $\zeta$ can be calculated for specific sieve basis functions and so does the convergence rate in (\ref{eq_rate}). For example, when $\{\phi_i(t)\}$ are chosen as the Fourier basis functions and $\{\varphi_j(t)\}$ as the mapped Fourier basis functions, then  $\xi_c=\OO(1)$ and $\zeta=\OO(\sqrt{p}). $ Additionally, if we assume that $\mathsf{m}_1=\mathsf{m}_2=\infty$ such that $m(t,x)$ is infinitely differentiable, by choosing $c=d=\OO(\log n),$ the rate on  the right-hand side of (\ref{eq_rate}) reads as $\OO(n^{-1/2} \log^2 n )$ which matches the optimal uniform (i.e. sup-norm) convergence rate as obtained in \cite{CHEN2015447,DZ, MR673642}. For the magnitudes of $\xi_c$ and $\zeta$ for more general sieves, we refer the readers to Appendix \ref{sec_sieves}.

Second, the condition (\ref{eq_parameterassumption}) ensures that $n^{-1} W^\top W$ in the OLS estimator  $\widehat{\bm{\beta}}$ in (\ref{eq_betaolsform}) will converge to the integrated long run covaraince matrix $\Pi,$ which guarantees the regularity behavior of $\widehat{\bm{\beta}}$; see (\ref{eq_consistencyconvergency}) for more details. In fact (\ref{eq_parameterassumption}) can be easily   satisfied. For example, when $\xi_c=\OO(1)$ and $p=\log^2n,$ we only require $\tau>1$ for the physical dependence measure as in (\ref{eq_physcialrequirementone}).   

\begin{remark}\label{rmk_case2ols}
We point out that the general non-stationary setting, i.e., (2) of Assumption \ref{assum_models} can be handled similarly with some minor modifications. For example, in (\ref{eq_designmatrix}), $X_{i-(l+1)}$ should be replaced by $X_{l+1, i},$ and in (\ref{eq_desigmatrixform}), $X_{i-\ell_1(k)-1}$ should be replaced by $X_{\ell_1(k)+1,i}.$ We omit the details since the main differences are just notational. 
\end{remark}

\end{remark}

\section{Statistical inference for the nonlinear regression}\label{sec_inference}

In this section, we infer the nonlinear functions $m_j(t,x), 1 \leq j \leq r,$ based on our proposed sieve least square estimator (\ref{eq_proposedestimator}). As in Section \ref{sec_estimation}, we focus our discussion on Case (1) of Assumption \ref{assum_models} and only briefly explain how to handle Case (2) without providing extra details in Remark \ref{rmk_case2testdistribution}. Throughout this section, we assume that for the error in (\ref{eq_rate})
\begin{equation}\label{eq_assumptionerrorreduce}
c^{-\mathsf{m}_1}+d^{-\mathsf{m}_2}=\OO(n^{-\gamma}), \ \gamma>\frac{1}{2}.  
\end{equation}
Note that (\ref{eq_assumptionerrorreduce}) is a mild assumption by choosing reasonably large values of $c$ and $d$. For example, if $\mathsf{m}_1=\mathsf{m}_2=\infty,$ it can be easily satisfied by choosing $c, d \asymp \log n.$ In general,  it will be satisfied if we choose 
\begin{equation}\label{eq_cdchoice}
c \asymp n^{C/\mathsf{m}_1}, \ d \asymp n^{C/\mathsf{m}_2},  
\end{equation}
for some large constant $C>\frac{1}{2}.$ This shows that we only need a relatively small amount of basis functions once $\mathsf{m}_k ,k=1,2,$ are reasonably large.

\subsection{Problems setup and proposed statistics}\label{sec_problemsetup}
In this subsection, we introduce several commonly concerned problems and provide our proposed statistics utilizing (\ref{eq_proposedestimator}). 

For $1 \leq j \leq r,$ our first goal is to construct a $(1-\alpha)$ simultaneous confidence region (SCR) of $m_j(t,x)$ based on  our estimator (\ref{eq_proposedestimator}), denoted as $\{\Upsilon_{\alpha}(t,x), 0 \leq t \leq 1, x \in \mathbb{R}\}.$ The SCR is defined as follows   
\begin{equation}\label{eq_scrdefinition}
\lim_{n \rightarrow \infty}\mathbb{P}\left\{ m_j(t,x) \in \Upsilon_\alpha(t,x), \ \forall \ 0 \leq t \leq 1, \ \forall \ x \in \mathbb{R} \right\}=1-\alpha,
\end{equation}
which generalizes the definition of simultaneous confidence band as in \cite{MR1804172, MR2758526}. For $0 \leq t \leq 1$ and $x \in \mathbb{R},$ denote 
\begin{equation}\label{eq_defhtx}
h_j(t,x)=\sqrt{\bm{l}_j^\top \Omega \bm{l}_j}, \ 1 \leq j \leq r,
\end{equation}
where $\bm{l}_j\equiv \bm{l}_j(t,x):=\Pi^{-1} (\bm{b} \otimes \mathbf{I}_r) \mathsf{I}_j$ and $\Pi, \Omega$ are defined in (\ref{eq_longruncovariancematrix}). Recall (\ref{eq_proposedestimator}).  Denote
\begin{equation}\label{eq_definitiont1}
\mathsf{T}_{1j}(t,x)=\widehat{m}_{j,c,d}(t,x)-m_j(t,x).
\end{equation}
As we will see in Theorem \ref{thm_asymptoticdistribution}, $h_j^2(t,x)$ is the asymptotic variance of $\sqrt{n}\mathsf{T}_{1j}.$ Therefore, given the nominal level $\alpha,$ the SCR can be constructed as follows 
\begin{equation}\label{eq_scrdefinition}
\hat m_{j,c,d}(t,x)\pm \frac{c_{\alpha}}{\sqrt{n}}h_j(t,x),\quad t\in[0,1], x\in\mathbb{R},
\end{equation}
where $c_{\alpha}$ is the critical value which follows that 
\begin{equation*}
\lim_{n \rightarrow \infty} \p(\sup_{t,x}\left| \frac{\sqrt{n}\mathsf{T}_{1j}}{h_j(t,x)}\right|\le c_{\alpha})= 1-\alpha. 
\end{equation*}

As discussed in \cite[Section 5.2]{YZ}, the SCR can serve as a foundation for inferring the structure of $m_j(t,x).$  The general idea can be described as follows. Under some specific structural assumptions (see Examples \ref{exam_stationarytest} and \ref{exam_seperabletest} below), the functions $m_{j}(t,x)$ can be estimated with faster convergence rates than those estimated using the general approach (\ref{eq_proposedestimator}). For example, in Example \ref{exam_stationarytest}, under the assumption that   $m_j$ is stationary (i.e., independent of time), instead of using $\OO(cd)$ hierarchical basis functions as in (\ref{eq_basisconstruction}), we only need $\OO(d)$ mapped sieve basis functions so that we will have a better rate compared to use all the basis functions as illustrated in (\ref{eq_rate}). Similarly, assuming $m_j(t,x)$ is separable in $t, x$ as  in Example \ref{exam_seperabletest}, we can estimate them separately so that we only need $\OO(c+d)$ basis functions.

Based on the above observations, if the functions $m_{j}(t,x)$ under the null hypothesis  can be estimated with faster convergence rates than those estimated using the general approach (\ref{eq_proposedestimator}), to test the structure of $m_j(t,x)$ under some specific null hypothesis, it suffices to check whether the estimated  functions under the null hypothesis (with faster convergence rates) can be embedded into the SCR constructed  generally using (\ref{eq_proposedestimator}). This general procedure achieves asymptotic power one as long as the alternatives deviate from the null hypothesis with a rate large than the order of the width of the SCR. In what follows, we present two important examples. 
 
\begin{example}[Testing time-homogeneity]\label{exam_stationarytest} We consider the structural assumption  that 
\begin{equation}\label{eq_hotestinghomogeneity}
\mathbf{H}_0: \ m_j(t,x) \equiv m_j(x), \ \ \forall  t \in [0,1],
\end{equation}
for some function $m_j(\cdot).$  Under the null assumption of (\ref{eq_hotestinghomogeneity}), we can estimate the function only using the basis $\{\varphi_j(x)\}$ instead of using (\ref{eq_basisconstruction}). Under the assumption of (\ref{eq_assumptionerrorreduce}), by an argument similar to Theorem \ref{thm_consistency}, let $\zeta_d:=\sup_x \| (\varphi_1(x), \cdots, \varphi_d(x))^* \|,$ we can obtain a convergence rate $\zeta_d (d/n)^{1/2}$ which is faster than the rate in (\ref{eq_rate}). 
\end{example}

\begin{example}[Testing separability]\label{exam_seperabletest} We consider the structural assumption that 
\begin{equation}\label{eq_hotestseparability}
\mathbf{H}_0: m_j(t,x)=f_j(t)g_j(x), \ \forall \ t \in [0,1] \ \text{and} \ x \in \mathbb{R}, 
\end{equation}
for some functions $f_j(\cdot)$ and $g_j(\cdot).$ Under the null assumption of (\ref{eq_hotestseparability}), we can estimate $f_j(\cdot)$ and $g_j(\cdot)$ separately using the basis functions $\{\phi_j(t)\}$ and $\{\varphi_j(x)\},$ respectively. Under the assumption of (\ref{eq_assumptionerrorreduce}), we can obtain a convergence rate $\zeta_c(c/n)^{1/2}+\zeta_d (d/n)^{1/2}$ which is faster than the rate in (\ref{eq_rate}).   
\end{example}

Our second task is to conduct hypothesis on the functions $m_j(t,x), 1 \leq j \leq r.$ In this problem, we want to test the exact forms of the functions instead of understanding their structure as discussed earlier. In this setting, generally, we cannot obtain an estimator with a faster convergence rate and SCR may not be applicable. Instead, we propose a $L_2$ test.  Specifically, we are interested in testing the following null hypothesis
\begin{equation}\label{eq_nullhypotheis}
\mathbf{H}_0: \ m_j(t,x) \equiv m_{j,0}(t,x),
\end{equation} 
where $m_{j,0}(t,x)$ is some pre-given function. To test (\ref{eq_nullhypotheis}), we propose the following  $L_2$ statistic
\begin{equation}\label{eq_statistic}
\mathsf{T}_{2j}=\int_{[0,1]}\int_{\mathbb{R}}[\hat m_{j,c,d}(t,x)-m_{j,0}(t,x)]^2 \dd t \dd x.
\end{equation}
Note that under the null hypothesis (\ref{eq_nullhypotheis}), $\mathsf{T}_{2j}$ should be small.

\begin{remark}\label{rmk_multiplecase}
Even though we focus our arguments for each specific $j,$ our discussion can be easily generalized to the multiple case. We take the second task as an example. Instead of considering (\ref{eq_nullhypotheis}), one may be interested in testing
\begin{equation*}
\mathbf{H}_0: \ m_j(t,x) \equiv m_{j,0}(t,x), \ \text{for all} \ 1 \leq j \leq r. 
\end{equation*}
In this setting, we shall use the statistic
\begin{equation}\label{eq_t2definition}
\mathsf{T}_{2}=\sum_{j=1}^r \int_{[0,1]}\int_{\mathbb{R}}[\hat m_{j,c,d}(t,x)-m_{j,0}(t,x)]^2 \dd t \dd x.
\end{equation}
As will be seen in Remark \ref{rmk_case2testdistribution} below, the distribution of $\mathsf{T}_2$ can be obtained in a similar fashion as $\mathsf{T}_{2j}.$ Similar arguments hold for the SCR. 
\end{remark}


\subsection{Asymptotic distributions of the statistics}
In this subsection, we provide the results of the asymptotic distributions of the statistics introduced in Section \ref{sec_problemsetup}. We prepare some notations.  
Recall (\ref{eq_defnxic}). For $\eta>0,$ denote the control parameter $\Theta(\eta)$ as
\begin{align}\label{eq_controlparameter}
\Theta(\eta)=\frac{p^{7/4}}{\sqrt{n}} \hd^3 m^2& +\left[\sqrt{p}\zeta \left(p \xi_c^2 m^{-\tau+1}+ pn \xi_c^4 \hd^{-(q-2)} \right) \right]^{\eta}+\left[p\sqrt{n} \zeta \xi_c^2 \hd^{-(q-1)}\right]^{\eta} \nonumber \\
&+\left[\zeta p^{2} \xi_c m^{-\tau+1} \right]^{\eta}+\frac{p\xi_c^2}{n} +\sqrt{p} \zeta\left(\frac{p\xi_c^2 n^{2/\tau}}{\sqrt{n}} +p\xi_c^2 n^{-1+\frac{2}{\tau+1}} \right) ,
\end{align}
where $m \equiv m(n)$ and $\hd \equiv \hd(n)$ are some large diverging values which can be chosen by the users. As can be seen in Theorem \ref{thm_gaussianapproximationcase}, $\Theta(\eta)$ is used to control the Gaussian approximation rates. For example, if $\Theta (1)=\oo(1),$ after properly scaled, $\mathsf{T}_{1j}$ in (\ref{eq_definitiont1}) will be asymptotically standard Gaussian, and if $\Theta(0.5)=\oo(1),$ $\mathsf{T}_{2j}$ will be asymptotically standard Gaussian with proper scaling.

We first state the results for the statistic $\mathtt{T}_{1j}$ and the critical value $c_{\alpha}.$ Recall (\ref{eq_betaolsform}) and (\ref{eq_proposedestimator}). Denote 
\begin{equation}\label{eq_ljdefinition}
l_j(t,x):=n^{-1}(\bm{b} \otimes \mathbf{I}_r)^\top \mathsf{I}_j \Pi^{-1}.
\end{equation} 
For $x \in \mathbb{R},$ using the mappings in Definition \ref{defn_mappings}, we write
\begin{equation}\label{eq_ll}
\widetilde{l}_j(t,\widetilde{x}):=l_j(t,g(2\widetilde{x}-1;s)) \equiv l_j(t,x), \ \widetilde{x} \in [0,1], 
\end{equation}  
and 
\begin{equation*}
\widetilde{\bm{b}} \equiv \widetilde{\bm{b}}(t, \widetilde{x})=\bm{b}(t, g(2\widetilde{x}-1);s) \equiv \bm{b}(t,x). 
\end{equation*}
Recall $\bm{\epsilon}=(\epsilon_i)_{r+1 \leq i \leq n} \in \mathbb{R}^{n-r}.$ Denote $T_j(t,\widetilde{x})$ as
\begin{equation}\label{eq_defntdefn}
T_j(t, \widetilde{x}):=\frac{\sqrt{n} \widetilde{l}_j(t,\widetilde{x}) \sqrt{\operatorname{Cov}(W^\top \bm{\epsilon})}}{\widetilde{h}_j(t, \widetilde{x})}, \  \widetilde{h}_j(t,\widetilde{x}):=h_j(t,g(2\widetilde{x}-1);s). 
\end{equation} 
Denote the manifold $\mathcal{M}_j, 1 \leq j \leq r$ as follows
\begin{equation}\label{eq_manifolddefinition}
\mathcal{M}_j=:\left\{ T_j(t, \widetilde{x}): (t,\widetilde{x}) \in [0,1] \times [0,1]  \right\}.
\end{equation}

\begin{theorem}\label{thm_asymptoticdistribution}
Suppose Assumptions \ref{assum_models}--\ref{assum_updc} and (\ref{eq_assumptionerrorreduce}) hold and $\mathsf{w}_1, \mathsf{w}_2 \geq 3$.  Recall $\Theta(\eta)$ in (\ref{eq_controlparameter}). We assume that 
\begin{equation}\label{eq_onebound}
\Theta(1)=\oo(1). 
\end{equation}
Then we have that:  \\
\noindent (1). For all $(t,x) \in [0,1] \times \mathbb{R}$   
\begin{equation}\label{eq_thmonepartone}
\frac{\sqrt{n} \mathsf{T}_{1j}(t,x)}{h_j(t,x)} \simeq \mathcal{N}(0,1),
\end{equation}
where $h_j(t,x)$ is defined in (\ref{eq_defhtx}).  \\
\noindent (2).  The critical value $c_{\alpha}$ in (\ref{eq_scrdefinition}) satisfies the following expansion
\begin{equation}\label{eq_expansionformula}
\alpha=\frac{c_{\alpha} \kappa_0}{\sqrt{2} \pi^{3/2}} \exp\left(-\frac{c_{\alpha}^2}{2} \right)+\frac{\zeta_0}{2 \pi} \exp\left(-\frac{c_{\alpha}^2}{2} \right)+2(1-\Phi(c_{\alpha}))+\oo\left( \exp\left(-\frac{c^2_{\alpha}}{2} \right)\right),
\end{equation}
where $\kappa_0, \zeta_0$ are the area and the length of the boundary of $\mathcal{M}_j$ in (\ref{eq_manifolddefinition}) and $\Phi(\cdot)$ is the cumulative distribution function of a standard Gaussian random variable. Moreover, suppose $c,d$ are chosen according to (\ref{eq_cdchoice}),  there exist some constants $\alpha_1, \alpha_2 \geq 0$ so that
\begin{equation}\label{eq_assumptionbasisderivative}
\int \frac{\| \nabla_t \widetilde{\bm{b}} \|_2}{\| \widetilde{\bm{b}} \|_2} \mathrm{d}t \dd \widetilde{x} \asymp n^{\alpha_1}, \ \int \frac{\| \nabla_{\widetilde{x}} \widetilde{\bm{b}} \|_2}{\| \widetilde{\bm{b}} \|_2} \mathrm{d}t \dd \widetilde{x} \asymp n^{\alpha_2}, \ 
\end{equation}
then we have 
\begin{equation}\label{eq_quantitiesbound}
c_{\alpha} \asymp  \log^{1/2} n.
\end{equation}

\end{theorem}

We point out that (\ref{eq_assumptionbasisderivative}) is a mild assumption which is frequently used in the literature of sieve nonparametric estimation and inference. It can be easily satisfied for the commonly used sieve basis functions under the choice (\ref{eq_cdchoice}); see Assumption 4 of \cite{CHEN2015447} for more details.


Next, we investigate the distribution and power of $\mathsf{T}_{2j}.$ Denote
\begin{equation}\label{eq_defnw}
\mathsf{W}_j=\Pi^{-1} \mathsf{B}_j \Pi^{-1}, \ \mathsf{B}_j:= \int_{[0,1]} \int_{\mathbb{R}} [(\bm{b} \otimes \mathbf{I}_r)\mathsf{I}_j] [(\bm{b} \otimes \mathbf{I}_r)\mathsf{I}_j]^\top \dd t \dd x.   
\end{equation}
Moreover, for $k=1,2,$ denote
\begin{equation}\label{eq_definemoments}
\mathfrak{m}_{jk}=\left( \operatorname{Tr}[\Omega^{1/2} \mathsf{W}_j \Omega^{1/2}]^k \right)^{1/k}, \ k \in \mathbb{N}. 
\end{equation}
To examine the power of our statistic $\mathsf{T}_{2j},$ we consider the following class of local alternative, i.e., a small perturbation of $m_{j,0}(t,x)$ 
\begin{equation}\label{eq_alternative}
\mathbf{H}_a:  m_j(t,x)=m_{j,0}(t,x)+v_{j,n}(t,x),
\end{equation}
where 
$v_{j,n}(t,x)$ depends on $n$ and  is some other deterministic function satisfying Assumption \ref{assum_smoothnessasumption}. 
\begin{theorem}\label{thm_poweranalysis}
Suppose Assumptions \ref{assum_models}--\ref{assum_updc} and (\ref{eq_assumptionerrorreduce}) hold. Moreover, we assume that for $\Theta(\eta)$ in (\ref{eq_controlparameter}) that
\begin{equation}\label{eq_05bound}
\Theta(0.5)=\oo(1), 
\end{equation}
and for some constants $\kappa_1, \mathcal{K}_2>0$
\begin{equation}\label{eq_beigenvaluebound}
\kappa_1 \leq  \lambda_{cd}(\mathsf{B}_j) \leq \lambda_{\max}(\mathsf{B}_j) \leq \mathcal{K}_2. 
\end{equation}
Then for $\mathsf{T}_{2j}$ in (\ref{eq_statistic}): \\
(1). When the null hypothesis $\mathbf{H}_0$ in (\ref{eq_nullhypotheis}) holds,  we have that
\begin{equation}\label{eq_limiting}
\frac{n \mathsf{T}_{2j}-\mathfrak{m}_{j1}}{\mathfrak{m}_{j2}} \simeq \mathcal{N}(0,2).
\end{equation}
(2). When the alternative hypothesis $\mathbf{H}_a$ in (\ref{eq_alternative}) holds such that
\begin{equation}\label{eq_alternativeassumption}
\int_{[0,1]} \int_{\mathbb{R}} v^2_{jn}(t,x) \dd t \dd x >C_{\alpha}\frac{\sqrt{p}}{n},
\end{equation}
where $C_{\alpha} \equiv C_{\alpha}(n) \rightarrow \infty$ as $n \rightarrow \infty,$ we have 
\begin{equation*}
\frac{n\mathsf{T}_{2j}-\mathfrak{m}_{j1}-n\int_{[0,1]} \int_{\mathbb{R}} v_{jn}^2(t,x) \dd t \dd x}{\mathfrak{m}_{j2}} \simeq \mathcal{N}(0,2).
\end{equation*}
Consequently, the power of our test will asymptotically be 1, i.e., 
\begin{equation*}
\lim_{n \rightarrow \infty} \mathbb{P}\left( \left| \frac{n \mathsf{T}_{2j}-\mathfrak{m}_{j1}}{\mathfrak{m}_{j2}} \right| \geq \sqrt{2} \mathcal{Z}_{1-\alpha} \right)=1,
\end{equation*}
where $\mathcal{Z}_{1-\alpha}$ is the $(1-\alpha)\%$ quantile of the standard Gaussian distribution. 
\end{theorem}

\begin{remark}\label{remark_t2distribution}
We remark that (\ref{eq_beigenvaluebound}) can be easily satisfied under certain choices of basis functions. We take $r=1$ for an example and omit the subscript $j$, if $\{\phi_i(t)\}$ is either Fourier series or orthogonal polynomials, due to orthonormality, $\mathsf{B}=\mathbf{I}.$ In these cases, for (\ref{eq_defnw}), $\mathsf{W}=\Pi^{-2}.$ Moreover, for the orthogonal wavelet, $\mathsf{B}$ can be controlled with straightforward calculations and (\ref{eq_beigenvaluebound}) can be easily satisfied. We refer the readers to \cite{DZ} for more details.     
\end{remark}

\begin{remark}\label{rmk_case2testdistribution}
As discussed earlier in Remark \ref{rmk_multiplecase}, we can easily generalize our discussions to the multiple functions testing case. For example, for the statistic (\ref{eq_t2definition}), Theorem \ref{thm_poweranalysis} still holds by replacing (\ref{eq_definemoments}) with 
\begin{equation*}
\mathfrak{m}_{k}=\left( \operatorname{Tr}[\Omega^{1/2} \mathsf{W} \Omega^{1/2}]^k \right)^{1/k}, \ k \in \mathbb{N},
\end{equation*}
where 
\begin{equation*}
\mathsf{W}=\Pi^{-1} \mathsf{B} \Pi^{-1}, \ \mathsf{B}:= \int_{[0,1]} \int_{\mathbb{R}} [(\bm{b} \otimes \mathbf{I}_r)] [(\bm{b} \otimes \mathbf{I}_r)]^\top \dd t \dd x.   
\end{equation*}

Moreover, we point out that the results can be generalized to Case (2) of Assumption \ref{assum_models} easily with some minor modifications as discussed in Remark \ref{rmk_case2ols}. 
\end{remark}

\subsection{Multiplier bootstrapping procedure and practical implementation}

In this subsection, we discuss the practical implementation for the statistical inference problems discussed in Section \ref{sec_problemsetup}. We again focus on Case (1) of Assumption \ref{assum_models}.  As we can see from Theorems \ref{thm_asymptoticdistribution} and \ref{thm_poweranalysis}, it is difficult to use them directly since the long-run covariance matrix of $\{\bm{x}_i\},$ i.e., $\Omega(t)$ defined in (\ref{eq_longrunwitht}), is hard to be estimated. To address this issue, in this subsection, we propose a practical method as in \cite{DZ2,MR3174655} which utilizes high-dimensional multiplier bootstrap statistics to mimic the distributions of $\{\bm{x}_i\}$. We first prepare some notations. Since $\{\epsilon_i\}$ cannot be observed directly, we need to use the residuals in practice, i.e., 
\begin{equation}\label{eq_defnresidual}
\widehat{\epsilon}_i=Y_i-\sum_{j=1}^r \widehat{m}_{j,c,d}(t_i, X_{j,i}),
\end{equation}   
where $\widehat{m}_{j,c,d}(\cdot,\cdot)$ is defined in (\ref{eq_proposedestimator}). Corresponding to (\ref{eq_ddd}), we set 
\begin{equation}\label{eq_xihatdefinition}
\widehat{x}_{ik}:=w_{ik} \widehat{\epsilon}_i, 
\end{equation}
and $\{\widehat{\bm{x}}_i\}$ can be defined based on the above notation. For some diverging parameter $m,$ denote
\begin{equation}\label{eq_defnstatisticXi}
\Xi:=\frac{1}{\sqrt{n-m-r} \sqrt{m}} \sum_{i=r+1}^{n-m} \left[ \left( \sum_{j=i}^{i+m} \widehat{\bm{x}}_j \right) \otimes \ab(t_i) \right] R_i,
\end{equation}
where $R_i, \ r+1 \leq i \leq n-m,$ are i.i.d. $\mathcal{N}(0,1)$ random variables. Note that $\Xi$ can be always calculated once we have the data set and the window-size (i.e., block length) parameter $m$.

First, for the SCR (\ref{eq_scrdefinition}), instead of using (\ref{eq_definitiont1}), based on the statistic (\ref{eq_defnstatisticXi}), we propose the following statistic
\begin{equation}\label{eq_defnt1k}
\widehat{\mathsf{T}}_{1j}:=\Xi^\top \widehat{\Pi}^{-1} (\bm{b} \otimes \mathbf{I}_r) \mathsf{I}_j, \  \ \widehat{\Pi}=\frac{1}{n} W^\top W.
\end{equation}
Before stating the theoretical results, we pause to heuristically discuss the motivation of using  $\widehat{\mathsf{T}}_{1j}$ to mimic the distribution of  $\mathsf{T}_{1j}$. For the ease of discussion, we consider that $r=1$ and omit the subscript $j.$ As can be seen in (\ref{eq_fundementalexpression}), $\sqrt{n} \mathsf{T}_1$ can be approximated  by 
\begin{equation*}
\bm{z}^\top \Pi^{-1}\bm{b}, \ \bm{z}:=\frac{1}{\sqrt{n}} \sum (\bm{x}_i \otimes \ab(t_i)).
\end{equation*}
According to Theorem \ref{thm_asymptoticdistribution}, the above quantity is asymptotically Gaussian. To calculate the variance, it suffices to estimate the variance of $\bm{z}$ and the value of $\Pi.$ On one hand,  as will be seen in (\ref{eq_consistencyconvergency}), $\Pi$ can be well approximated by $\widehat{\Pi}$ and so does the precision. On the other hand, $\Xi$ is a multiplier statistic based on $\bm{z}$ and can mimic the distribution of $\bm{z}$ conditional on the data, especially we have that $\operatorname{Cov}(\Xi) \approx \Omega$ where we recall that $\Omega$ is defined in (\ref{eq_longruncovariancematrix}); see Lemmas \ref{lem_stepone}--\ref{lem_residualclosepreparation} for more detailed discussions.

We now state the results formally.  Denote the control parameter $\Psi(m)$ as
\begin{equation}\label{eq_defnpsim}
\Psi(m)=d\zeta^2 \left(\frac{1}{m}+\left(\frac{m}{n} \right)^{1-\frac{1}{\tau}}+\sqrt{\frac{m}{n}} \right).
\end{equation}

\begin{theorem}\label{thm_boostrapping}
Suppose the assumptions of Theorem \ref{thm_asymptoticdistribution} hold. Moreover, we assume that 
\begin{equation}\label{eq_boothstrappingextraassumption}
\zeta \Psi(m)+ \zeta p\left( \frac{\xi^2_c}{\sqrt{n}}+\frac{\xi^2_c n^{\frac{2}{\tau+1}}}{n}\right)+\sqrt{p} \frac{d \zeta^3}{\sqrt{n}} \left[\xi_c \zeta \sqrt{\frac{p}{n}}+c^{-\mathsf{m}_1}+d^{-\mathsf{m}_2} \right]=\oo(1). 
\end{equation}
Then conditional on the data $\{X_i\},$ we have that
\begin{equation*}
\sup_{y \in \mathbb{R}} \left| \mathbb{P} \left( \frac{\widehat{\mathsf{T}}_{1j}}{h_j(t,x)} \leq y \right)-\mathbb{P}(\psi \leq y) \right|=\oo(1),
\end{equation*} 
where $\psi$ is a standard Gaussian random variable. 
\end{theorem}
\begin{remark}\label{rmk_practicalone}
Together with Theorems \ref{thm_asymptoticdistribution} and \ref{thm_boostrapping}, we conclude that it suffices to utilize (\ref{eq_defnt1k}) which is adaptive to the underlying data set, as our test statistics for practical implementation.  The error term (\ref{eq_defnpsim}) is used to control the closeness  between $\operatorname{Cov}(\Xi)$ and $\Omega.$  Furthermore, we point out that $m$ is a parameter which needs to be chosen carefully. Assuming that $\tau$ is large, then (\ref{eq_boothstrappingextraassumption}) indicates that the optimal  choice of $m$ should satisfy that $m \asymp n^{1/3}.$ In Section \ref{sec_parameterchoice}, we will discuss how to choose $m$ practically. 

\end{remark}
Next, for the hypothesis testing (\ref{eq_nullhypotheis}), in practice, instead of using (\ref{eq_statistic}), we can  employ the following statistic
\begin{equation}\label{eq_defnstatt2}
\widehat{\mathsf{T}}_{2j}:=\Xi^\top \widehat{\mathsf{W}}_j \Xi, 
\end{equation}
 where $\widehat{\mathsf{W}}_j$ is defined as 
 \begin{equation}\label{eq_defnhatsfw}
 \widehat{\mathsf{W}}_j:=\widehat{\Pi}^{-1} \mathsf{B}_j \widehat{\Pi}^{-1}, \ \widehat{\Pi}=\frac{1}{n} W^\top W,
 \end{equation}
 where $W$ is the design matrix defined in (\ref{eq_designmatrix}) and $\mathsf{B}_j$ is defined in (\ref{eq_defnw}). 

\begin{theorem}\label{thm_boostrapping2}
Suppose the assumptions of Theorem \ref{thm_poweranalysis} hold. Moreover, we assume that 
\begin{equation}\label{eq_boostrappingparamterassumption}
\sqrt{p} \Psi(m)+p^{3/2}\left( \frac{\xi^2_c}{\sqrt{n}}+\frac{\xi^2_c n^{\frac{2}{\tau+1}}}{n}\right)+p \frac{d \zeta^2}{\sqrt{n}} \left[\xi_c \zeta \sqrt{\frac{p}{n}}+c^{-\mathsf{m}_1}+d^{-\mathsf{m}_2} \right]=\oo(1). 
\end{equation}
 Recall (\ref{eq_definemoments}). When $\mathbf{H}_0$ in (\ref{eq_nullhypotheis}) holds, there exists some set $\mathcal{A} \equiv \mathcal{A}_n$ such that $\mathbb{P}(\mathcal{A}_n)=1-\oo(1),$ and under the event $\mathcal{A}_n,$ we have that conditional on the data $\{X_i\}$ 
\begin{equation*}
\sup_{x \in \mathbb{R}} \left| \mathbb{P} \left( \frac{\widehat{\mathsf{T}}_{2j}-\mathfrak{m}_{j1}}{\sqrt{2} \mathfrak{m}_{j2}}  \leq x\right)-\mathbb{P}(\psi \leq x) \right|=\oo(1),
\end{equation*}
where $\psi$ is a standard Gaussian random variable. 
\end{theorem}

Similar to the discussions in Remark \ref{rmk_practicalone}, Theorems \ref{thm_boostrapping2} and \ref{thm_poweranalysis} imply that we can use (\ref{eq_defnstatt2}) to test $\mathbf{H}_0$ in (\ref{eq_nullhypotheis}). Before concluding this section, we summarize our methodologies in the Algorithms \ref{alg:boostrapping} and \ref{alg:boostrapping2} below. For notational simplicity, we focus on the case $r=1$ and omit the subscript $j.$ In Algorithm \ref{alg:boostrapping}, we present the detailed procedure for constructing SCR using the multiplier statistic (\ref{eq_defnt1k}). Based on the SCR, we can further consider the hypothesis testing problems as discussed in Examples \ref{exam_stationarytest} and \ref{exam_seperabletest} using the strategies discussed thereby. In Algorithm \ref{alg:boostrapping2}, we state the detailed steps to test (\ref{eq_nullhypotheis}) using (\ref{eq_defnstatt2}). Finally, we point out that the parameters $c,d$ and $m$ should be chosen before using these algorithms. This will be discussed in Section \ref{sec_parameterchoice} below.


\begin{algorithm}[!t]
\caption{\bf Multiplier Bootstrapping for constructing SCR}
\label{alg:boostrapping}
\vspace{3pt}
\normalsize
\begin{flushleft}
\noindent{\bf Inputs:}   tuning parameters $c$, $d$ and $m$ chosen by the data-driven procedure demonstrated in Section \ref{sec_parameterchoice}, time series $\{X_i\},$ and sieve basis functions.
\vspace{2pt}

\noindent{\bf Step one:} Compute $\widehat{\Pi}^{-1}$ as in (\ref{eq_defnhatsfw}), the estimate $\widehat{m}_{j,c,d}(t,x)$ as in (\ref{eq_proposedestimator}), and the residuals $\{\widehat{\epsilon}_i\}$ according to (\ref{eq_defnresidual}).
\vspace{2pt}

\noindent{\bf Step two:} For each given pair $(t,x),$ generate $B$ (say 1000) i.i.d. copies of $\{\Xi^{(k)}\}_{k=1}^B$ as in (\ref{eq_defnstatisticXi}). Compute $\widehat{\mathsf{T}}_{1,k}, k=1,2,\cdots, B,$  correspondingly as in (\ref{eq_defnstatt2}).  Calculate the sample standard deviation (s.t.d.) of $\{\widehat{\mathsf{T}}_{1,k}\}$ and denote it as $\widehat{h}(t,x).$
\vspace{2pt}

\noindent{\bf Step three:} Construct a sequence of uniform grids of $[0,1] \times [0,1],$ denoted as $(t_i, y_j)$, $1 \leq i$ $\leq C_1$, $1 \leq j \leq C_2,$ where $C_1$ and $C_2$ are some large integers (say $C_1=C_2=2,000$). Using the mappings in Definition \ref{defn_mappings} and calculate $x_j=g(2y_j-1;s).$  
\vspace{2pt}

\noindent{\bf Step four:} For each pair of $(t_i, x_j),$ calculate the associated s.t.d. following Steps one and two and denote them as $\widehat{h}(t_i, x_j).$ 
\vspace{2pt}

\noindent{\bf Step five:} For each pair of $(t_i, x_j),$  generate $M$ (say 1000) i.i.d. copies of  $\{\widehat{\mathsf{T}}_{1,k}(t_i, x_j)\}$, $1 \leq$ $k \leq M.$ For each $k,$ calculate $\mathcal{T}_k$ as follows
\begin{equation*}
\mathcal{T}_k:=\sup_{i,j}\left| \frac{\widehat{\mathsf{T}}_{1,k}(t_i, x_j)}{\widehat{h}(t_i, x_j)} \right|.
\end{equation*}
Let $\mathcal{T}_{(1)} \leq \mathcal{T}_{(2)} \leq \cdots \leq \mathcal{T}_{(M)}$ as the order statistics of $\{\mathcal{T}_k\}.$ Calculate $\widehat{c}_{\alpha}$ as
\begin{equation*}
\widehat{c}_{\alpha}=\mathcal{T}_{\lfloor M(1-\alpha) \rfloor},
\end{equation*}
where $\lfloor x \rfloor$ denotes the largest integer smaller or equal to $x.$
\vspace{2pt}

\noindent{\bf Output:} The SCR of the nominal level $\alpha$ can be represented as $\widehat{m}_{j,c,d}(t,x) \pm \frac{\widehat{c}_{\alpha}}{\sqrt{n}} \widehat{h}(t,x).$
\end{flushleft}

\end{algorithm}

\begin{algorithm}[!t]
\caption{\bf Multiplier Bootstrapping for inferring (\ref{eq_nullhypotheis})}
\label{alg:boostrapping2}
\vspace{3pt}
\normalsize
\begin{flushleft}
\noindent{\bf Inputs:}   tuning parameters $c$, $d$ and $m$ chosen by the data-driven procedure demonstrated in Section \ref{sec_parameterchoice}, time series $\{X_i\},$ and sieve basis functions.
\vspace{2pt}

\noindent{\bf Step one:} Compute $\widehat{\Pi}^{-1}$ as in (\ref{eq_defnhatsfw}) and the residuals $\{\widehat{\epsilon}_i\}$ according to (\ref{eq_defnresidual}).
\vspace{2pt}

\noindent{\bf Step two:}  Generate $B$ (say 1000) i.i.d. copies of $\{\Xi^{(k)}\}_{k=1}^B.$ Compute $\widehat{\mathsf{T}}_{2,k}, k=1,2,\cdots, B,$  correspondingly as in (\ref{eq_defnstatt2}). 
\vspace{2pt}

\noindent{\bf Step three:} Let $\widehat{\mathsf{T}}_{2,(1)} \leq \widehat{\mathsf{T}}_{2,(2)} \leq \cdots \leq \widehat{\mathsf{T}}_{2,(B)}$ be the order statistics of $\widehat{\mathsf{T}}_{2,k}, k=1,2,\cdots, B.$ Reject $\mathbf{H}_0$ in (\ref{eq_nullhypotheis}) at the level $\alpha$ if $n\mathsf{T}_2>\widehat{\mathsf{T}}_{2,(\lfloor B(1-\alpha)\rfloor)},$ where $\lfloor x \rfloor$ denotes the largest integer smaller or equal to $x.$  Let $B^*=\max\{r: \widehat{\mathsf{T}}_{2,r} \leq n\mathsf{T}_2\}.$

\noindent{\bf Output:} $p$-value of the test can be computed as $1-\frac{B^*}{B}.$
\end{flushleft}
\end{algorithm}

\section{Numerical simulations}\label{sec_numerical}
In this section, we conduct extensive numerical simulations to illustrate the usefulness of our results. As before, we focus our discussion on Case (1)  of Assumption \ref{assum_models}.  
\subsection{Parameter selection}\label{sec_parameterchoice}
In this subsection, we discuss how to choose the important parameters in practice using the data driven approach as in \cite{bishop2013pattern}.    

We first discuss how to choose $c$ and $d.$ For a given integer $l,$ say $l=\lfloor 3 \log_2 n \rfloor,$ we divide the time series into two parts: the training part $\{X_i\}_{i=1}^{n-l}$ and the validation part $\{X_i\}_{i=n-l+1}^n.$  With some preliminary initial pair $(c,d)$, we propose a sequence of candidate pairs  $(c_i, d_j), \ i=1,2,\cdots, u, \ j=1,2,\cdots, v,$ in an appropriate neighborhood of $(c,d)$ where $u, v$ are some given large integers. For each pair of the choices $(c_i, d_j),$ we estimate $\widehat{m}_{j,c,d}(t,x)$ as in (\ref{eq_proposedestimator}).  Then using the fitted model, we forecast the time series in the validation part of the time series.  Let $\widehat X_{n-l+1,ij}, \cdots, \widehat X_{n,ij}$ be the forecast of $X_{n-l+1},..., X_n,$ respectively using the parameter pair $(c_i, d_j)$. Then we choose the pair $(c_{i_0},d_{j_0})$ with the minimum sample MSE of forecast, i.e.,
 \begin{equation*} 
({i_0},{j_0}):= \argmin_{((i,j): 1 \leq i \leq u, 1 \leq j \leq v)} \frac{1}{l}\sum_{k=n-l+1}^n (X_k-\widehat X_{k,ij})^2.
 \end{equation*}

Then we discuss how to choose $m$ for practical implementation. In \cite{MR3174655}, the author used the minimum volatility (MV) method to choose the window size $m$ for the scalar covariance function. The MV method does not depend on the specific form of the underlying time series dependence structure and hence is robust to misspecification of
the latter structure \cite{politis1999subsampling}. The MV method utilizes the fact that the covariance structure of $\widehat{\Omega}$ becomes stable when the
block size $m$ is in an appropriate range, where $\widehat{\Omega}=E[\Phi\Phi^*|(X_1,\cdots,X_n)]=$ is defined as 
{ 
\begin{equation}\label{eq_widehatomega}
\widehat{\Omega}:=\frac{1}{(n-m-r+1)m} \sum_{i=b+1}^{n-m} \Big[ \Big(\sum_{j=i}^{i+m} \widehat{\bm{x}}_i \Big) \otimes \Big( \mathbf{a}(\frac{i}{n}) \Big) \Big] \times \Big[ \Big(\sum_{j=i}^{i+m} \widehat{\bm{x}}_i \Big) \otimes \Big( \mathbf{a}(\frac{i}{n}) \Big) \Big]^\top.
\end{equation}
}    Therefore, it desires to minimize the standard errors of the latter covariance structure in a suitable range of candidate $m$'s.

In detail, for a give large value $m_{n_0}$ and a neighbourhood control parameter $h_0>0,$  we can choose a sequence of window sizes $m_{-h_0+1}<\cdots<m_1< m_2<\cdots<m_{n_0}<\cdots<m_{n_0+h_0}$  and obtain $\widehat{\Omega}_{m_j}$ by replacing $m$ with $m_j$ in (\ref{eq_widehatomega}), $j=-h_0+1,2, \cdots, n_0+h_0.$ For each $m_j, j=1,2,\cdots, m_{n_0},$ we calculate the matrix norm error of $\widehat{\Omega}_{m_j}$ in the $h_0$-neighborhood, i.e., 
\begin{equation*}
\mathsf{se}(m_j):=\mathsf{se}(\{ \widehat{\Omega}_{m_{j+k}}\}_{k=-h_0}^{h_0})=\left[\frac{1}{2h_0} \sum_{k=-h_0}^{h_0} \| \overline{\widehat{\Omega}}_{m_j}-\widehat{\Omega}_{m_j+k} \|^2 \right]^{1/2},
\end{equation*}
where $\overline{\widehat{\Omega}}_{m_j}=\sum_{k=-h_0}^{h_0} \widehat{\Omega}_{m_j+k} /(2h_0+1).$
Therefore, we choose the estimate of $m$ using 
\begin{equation*}
\widehat{m}:=\argmin_{m_1 \leq m \leq m_{n_0}} \mathsf{se}(m).
\end{equation*}
Note that in \cite{MR3174655} the author used $h_0=3$ and we also use this choice in the current paper.

%

\subsection{Simulations setup}\label{sec_simulationsettup}
In this subsection, we introduce our simulation setup in the time series setting. We consider that $r=1$ and consider the following specifications of $m(t,x)$ which  have been studied in the literature. Consider that
\begin{equation*}
X_i=m(t, X_{i-1})+\sigma(t, X_{i-1})\epsilon_i.
\end{equation*}
Note that we will choose $\epsilon_i$ to be a locally stationary time series so that the variance of $\sigma(t, X_{i-1})\epsilon_i$ depends on both $X_{i-1}$ and $t.$ The settings of $m(t,x), \sigma(t,x)$ and $\{\epsilon_i\}$ are listed below.

In what follows, we consider that $m(t,x), \sigma(t,x)$ take the following forms. For some $0 \leq \delta<1,$
\begin{enumerate}
\item[(1).] $m(t,x)=5t+4 \cos (2 \pi t x), \ \sigma(t, x)=1.5 \exp(-x^2/2)\left(2+\sin(2 \pi t) \right).$
\item[(2).] $m(t,x)=( \delta \sin (2 \pi t)+1)\exp(-x^2/2),$ and $\sigma(t,x)$ is defined as  $$\sigma(t, x)=\frac{1.5 \exp(x)}{1+\exp(x)} \left(0.5 \cos(2 \pi t x)+1 \right).$$ 
\item[(3).] $m(t,x)=4t( \delta \cos(2 \pi t x)-\frac{1}{2} \exp(-x^2/2)),$ and $\sigma(t,x)$ is defined as
\begin{equation*}
\sigma(t, x):=
\begin{cases}
0.7(1+x^2), & |x| \leq 1; \\
1.4, &  |x| >1 \ \text{and} \ 0 \leq t<0.5, \\
2, & |x|>1 \ \text{and} \ 0.5 \leq t<1. 
\end{cases}
\end{equation*}
\end{enumerate}

Moreover, for the locally stationary time series $\{\epsilon_i\}$, denote
\begin{equation*}
a_1(\frac{i}{n}) \equiv 0.4, \ a_2(\frac{i}{n})=0.4 \sin(2 \pi \frac{i}{n}),  
\end{equation*}
and $\eta_i, i=1,2,\cdots,n$ are i.i.d. standard Gaussian random variables,  we consider the following settings
\begin{enumerate}
\item[(a).] Time-varying linear AR(2) model  
\begin{equation*}
\epsilon_i=\sum_{j=1}^2 a_j(\frac{i}{n}) \epsilon_{i-j}+\eta_i.
\end{equation*} 
\item[(b).] Self-exciting threshold auto-regressive (SETAR) model
\begin{equation*}
\epsilon_i=
\begin{cases}
a_1(\frac{i}{n}) \epsilon_{i-1}+\eta_i, & \epsilon_{i-1} \geq 0, \\
a_2(\frac{i}{n}) \epsilon_{i-1}+\eta_i, & \epsilon_{i-1}<0.
\end{cases}
\end{equation*}
\item[(c).] First order bilinear model 
\begin{equation*}
\epsilon_i=\left(a_1(\frac{i}{n}) \eta_{i-1}+a_2(\frac{i}{n}) \right)\epsilon_{i-1}+\eta_i.
\end{equation*}
\end{enumerate}
\subsection{Simultaneous confidence region}
In this subsection, we examine the simultaneous estimation performance of our proposed sieve estimators for models (1)--(3) of Section \ref{sec_simulationsettup} using the simulated coverage probabilities. In Models (2)--(3), we choose $\delta=1.$  We also compare our methods with the kernel based estimation methods as in \cite{MV, ZW}. Since model (2) is separable, we utilize the kernel method in \cite{CSW} which shows better performance as the general method in \cite{MV, ZW}. Moreover, since our our regressors are unbounded, to facilitate the comparison with the kernel method, we focus on the region $(t,x) \in [0,1] \times [-2000, 2000].$ For the kernel method, we use the Epanechnikov kernel and the cross-validation approach as in \cite{10.1214/18-AOS1743} to select the bandwidth. For our sieves method, the parameters $c$ and $d$ are chosen according to the discussion of Section \ref{sec_parameterchoice}. The results are reported in Table \ref{table_cp}. We conclude that our estimators achieve  reasonably high accuracy and outperforms the kernel estimators for all the commonly used sieve basis functions. Especially, our estimators have already had a good performance even the sample size is relatively small when $n=500.$ Finally, in practice, based on our simulations, we recommend to use orthogonal wavelet basis functions.

\begin{table}[!ht]
\begin{center}
\setlength\arrayrulewidth{1pt}
\renewcommand{\arraystretch}{1.5}
{\fontsize{10}{10}\selectfont 
\begin{tabular}{|c|ccc|ccc|ccc|ccc|}
\hline
& \multicolumn{6}{c|}{\large nominal level: $90\%$} & \multicolumn{6}{c|}{\large nominal level: $95\%$} \\
\hline
      & \multicolumn{3}{c|}{$n=500$}                                                                                                                       & \multicolumn{3}{c|}{$n=800$}                                                                                                                        & \multicolumn{3}{c|}{$n=500$}                                                                                                                       & \multicolumn{3}{c|}{$n=800$}                                                                                                                        \\ \hline
Model/$\epsilon_i$ & \multicolumn{1}{c|}{(a)} & \multicolumn{1}{c|}{(b)} & \multicolumn{1}{c|}{(c)} &  \multicolumn{1}{c|}{(a)} & \multicolumn{1}{c|}{(b)} & \multicolumn{1}{c|}{(c)}  & \multicolumn{1}{c|}{(a)} & \multicolumn{1}{c|}{(b)} & \multicolumn{1}{c|}{(c)} &  \multicolumn{1}{c|}{(a)} & \multicolumn{1}{c|}{(b)} & \multicolumn{1}{c|}{(c)} \\ 
\hline
     & \multicolumn{6}{c|}{Sieve estimators (Fourier basis)}                                                                                                                                                                                                                                                                                           & \multicolumn{6}{c|}{Sieve estimators (Fourier basis)}                                                                                                                                                                                                                                                                                          \\
   \hline
(1)     & 0.853             &  0.835                         &                0.862           &  0.856                          &                0.865           &  0.884                                     & 0.931 & 0.939 & 0.929 & 0.938 & 0.94 & 0.938 \\
(2)    & 0.842         &  0.848 &   0.855                    &         0.876                & 0.885                         & 0.908 & 0.929 & 0.961 & 0.935 & 0.939 & 0.958 & 0.945                \\
(3)    & 0.852  &  0.835  &   0.864                       &        0.913           &    0.874                      & 0.865 & 0.963 & 0.936 & 0.961 & 0.955 & 0.954 & 0.947   \\
\hline
      & \multicolumn{6}{c|}{Sieve estimators (Legendre basis)}                                                                                                                                                                                                                                                                                          & \multicolumn{6}{c|}{Sieve estimators (Legendre basis)}                                                                                                                                                                                                                                                                                         \\
       \hline
(1)     &   0.826            &   0.845              &         0.852                  &     0.843                      &    0.849 &                               0.857  & 0.934 & 0.962 & 0.932 & 0.948 & 0.943 & 0.953          \\
(2)     &  0.928       &  0.875   &  0.918                  &          0.921                & 0.883                          &  0.894 & 0.962 & 0.939 & 0.943 & 0.957 & 0.945 & 0.944  \\
(3)     & 0.858   &  0.876    & 0.925                        &         0.861                   & 0.857                           &       0.91   & 0.938 & 0.941 & 0.963 & 0.943 & 0.954 & 0.948                        \\
 \hline
  & \multicolumn{6}{c|}{Sieve estimators (Daubechies-9 basis)}                                                                                                                                                                                                                                                                                         & \multicolumn{6}{c|}{Sieve estimators (Daubechies-9 basis)}                                                                                                                                                                                                                                                                                         \\
       \hline
(1)     &     0.858             &   0.861               &                           0.914 & 0.883                           &   0.867 &     0.91 & 0.941 & 0.959 & 0.961 & 0.945 & 0.948 & 0.956                                   \\
(2)     & 0.867        &   0.855  &  0.859                 &       0.875                    &  0.868                         &  0.877 & 0.961 & 0.959 & 0.939 & 0.956& 0.948 & 0.943    \\
(3)     & 0.913   & 0.924   & 0.863                        &       0.89                    &    0.906                    &     0.879 & 0.937 & 0.964 & 0.958 & 0.944 & 0.956 & 0.947                              \\
 \hline
       & \multicolumn{6}{c|}{Kernel estimators}                                                                                                                                                                                                                                                                                         & \multicolumn{6}{c|}{Kernel estimators}                                                                                                                                                                                                                                                                                         \\
       \hline
(1)     &        0.756         &   0.811                &                           0.765 &            0.754           &   0.834 & 0.817                                    & 0.913 & 0.9 & 0. 886 & 0.896 & 0.922 & 0.915    \\
(2)    & 0.788      &  0.812   &  0.797                                                &        0.814                   &                       0.853 &  0.84  & 0.899 & 0.904 & 0.908 & 0.912 & 0.9 & 0.913  \\
(3)     & 0.746   &  0.798   &    0.754                      &         0.81                   &  0.807                          &           0.798   & 0.903 & 0.899 & 0.906 & 0.921 & 0.918 & 0.918                     \\
 \hline
\end{tabular}
}
\end{center}
\caption{Simulated coverage probabilities at $90\%$ and  $95\%$ nominal levels. For models (2) and (3) we set $\delta=1.$ The coverage probabilities are based on 5,000 simulations on the region $(t,x) \in [0,1] \times [-2000, 2000].$ }
\label{table_cp}
\end{table}


\subsection{Accuracy and power of multiplier bootstrap}
In this subsection, we examine the performance of our proposed multiplier bootstrap procedure for three different hypothesis testing problems using Algorithms \ref{alg:boostrapping} and \ref{alg:boostrapping2}.  In the first test, we consider Example \ref{exam_stationarytest} to test whether the smooth function is time invariant. Especially, we use model (2) of Section \ref{sec_simulationsettup} to test 
\begin{equation}\label{eq_equaivalent}
\mathbf{H}_0: \delta=0 \  \text{Vs}  \ \mathbf{H}_a: \delta>0.
\end{equation}    
In the second test, we consider Example \ref{exam_seperabletest} to test whether there exists a separable structure in the smooth function. Specifically, we use model (3) of Section \ref{sec_simulationsettup} to test (\ref{eq_equaivalent}).  In the third test, we consider (\ref{eq_nullhypotheis}) to check whether the smooth function is equal to a given function. Under the null hypothesis, we consider model (1) of Section \ref{sec_simulationsettup}. For the alternative, we consider that 
\begin{equation*}
\mathbf{H}_a: \ m(t,x)=5t+4\cos(2 \pi t x)+ \delta \sin (2 \pi t x), \ 0<\delta \leq 1. 
\end{equation*}

In Table \ref{table_typeoneerror}, we report the type one error rates of the above three tests under the null that $\delta=0.$ For the ease of statements, we call the above three tests as stationarity, separability and exam form tests, respectively. We find that our multiplier bootstrapping procedure are reasonably accurate for all of them.  Then we examine the power of our methodologies as $\delta$ increases away from zero. In Figure \ref{fig_selfcomparison}, we report the results for all the three tests using our multiplier bootstrapping method based on Daubechies-9 basis functions. It can be concluded that once $\delta$ deviates from $0$ a little bit, our method will be able to reject the null hypothesis.  Moreover, in \cite{HHY}, the authors proposed a weighted $L_2$-distance test statistic to test the separability hypothesis for $m(t,x)$ defined on compact domains. In Figure \ref{fig_jasacomparison}, we compare our multiplier bootstrapping method with the  weighted $L_2$-distance based method. It can be seen that our method has better performance when $\delta$ is small, i.e., weak alternatives.


\begin{table}[!ht]
\begin{center}
\setlength\arrayrulewidth{1pt}
\renewcommand{\arraystretch}{1.5}
{\fontsize{10}{10}\selectfont 
\begin{tabular}{|c|ccc|ccc|}
\hline
      & \multicolumn{3}{c|}{$n=500$}                                                                                                                       & \multicolumn{3}{c|}{$n=800$}                                                                                                                        \\ \hline
Testing/$\epsilon_i$ & \multicolumn{1}{c|}{(a)} & \multicolumn{1}{c|}{(b)} & \multicolumn{1}{c|}{(c)} &  \multicolumn{1}{c|}{(a)} & \multicolumn{1}{c|}{(b)} & \multicolumn{1}{c|}{(c)}  \\ 
\hline
     & \multicolumn{6}{c|}{Fourier basis}                                                                                                                                                                                                                                                                                          \\
   \hline
Stationarity     &          0.114  & 0.091 &  0.089                         &             0.11              &      0.087                     &       0.108                                  \\
Separability    &   0.118         &                        0.107   &       0.113                    & 0.105                           &     0.092                      &  0.094                 \\
Exact form   &  0.116  &  0.113   &  0.109                        & 0.108                   &             0.093             &          0.094          \\
\hline
      & \multicolumn{6}{c|}{Legendre basis}                                                                                                                                                                                                                                                                                         \\
       \hline
Stationarity       & 0.112                  & 0.11                  &                   0.088        &  0.11                         &0.107    &  0.106                                          \\
Separability     & 0.09        & 0.092    & 0.111                                           &      0.108                     &  0.109                       &  0.094  \\
Exact form    &  0.118  &  0.113   & 0.114                        &        0.11                    & 0.093                          &      0.096                            \\
 \hline
  & \multicolumn{6}{c|}{Daubechies-9 basis}                                                                                                                                                                                                                                                                                         \\
       \hline
Stationarity       &   0.112              & 0.087 & 0.091                          & 0.092                         &  0.093  &  0.107                                         \\
Separability      &  0.117      &   0.085  & 0.092                                                  & 0.109                          &  0.09                       &   0.092  \\
Exact form     &  0.118 &  0.087  & 0.115                       &            0.112               &  0.093                        &       0.11    \\
 \hline
\end{tabular}
}
\end{center}
\caption{ Simulated type one error rates under nominal level $0.1$. The results are reported based on 5,000 simulations. As mentioned earlier,  stationarity refers to the test on $m(t,x) \equiv m(x),$ separability refers to the test on $m(t,x)=\rho(t) g(x)$ for some smooth functions $\rho(\cdot)$ and $g(\cdot)$ and exact form refers to the test on $m(t,x)=m_0(t,x)$ for some pre-given function $m_0(t,x).$   
}
\label{table_typeoneerror}
\end{table}

\begin{figure}[!ht]

\begin{subfigure}{0.32\textwidth}
\includegraphics[width=6.4cm,height=5cm]{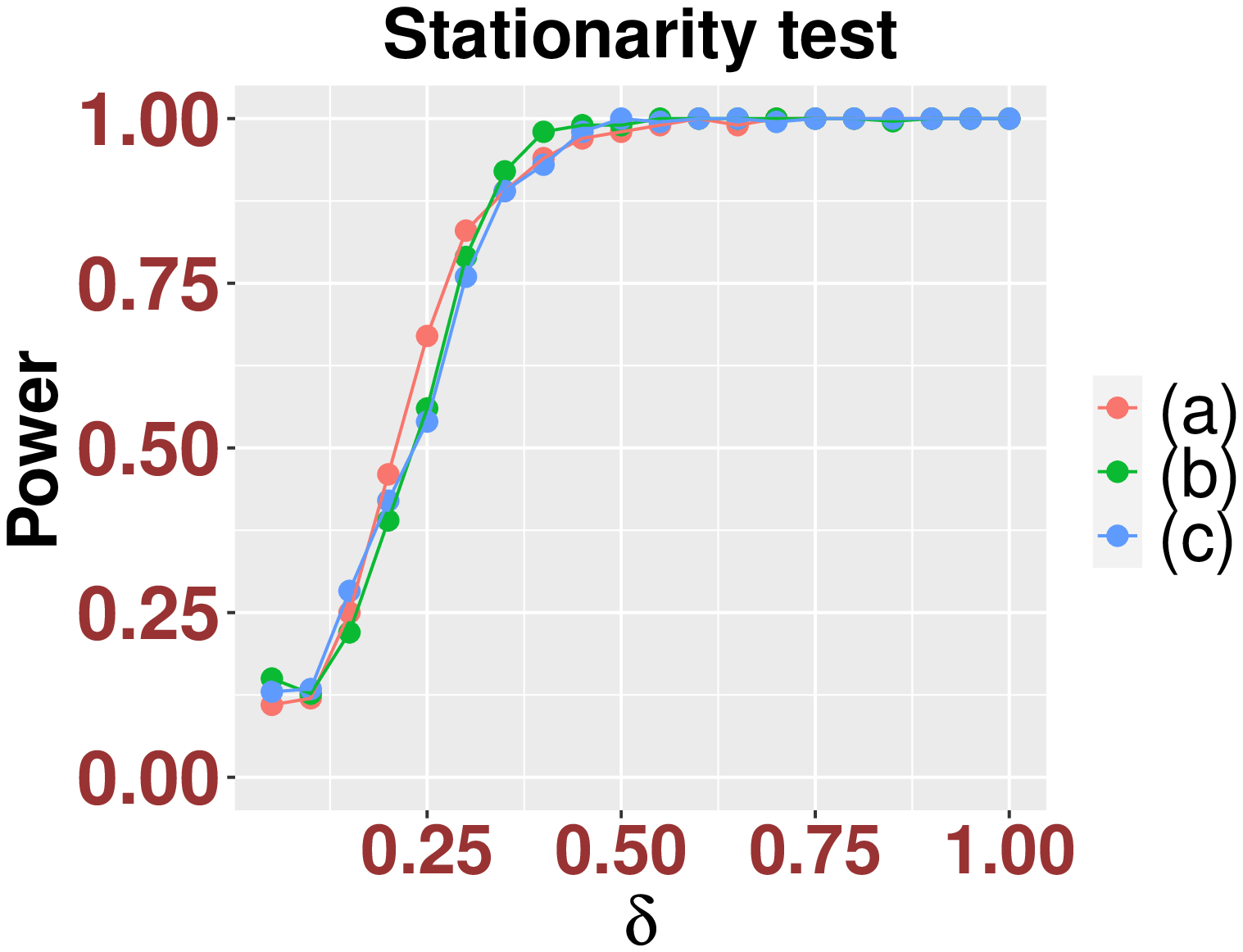}
\end{subfigure}
\begin{subfigure}{0.32\textwidth}
\includegraphics[width=6.4cm,height=5cm]{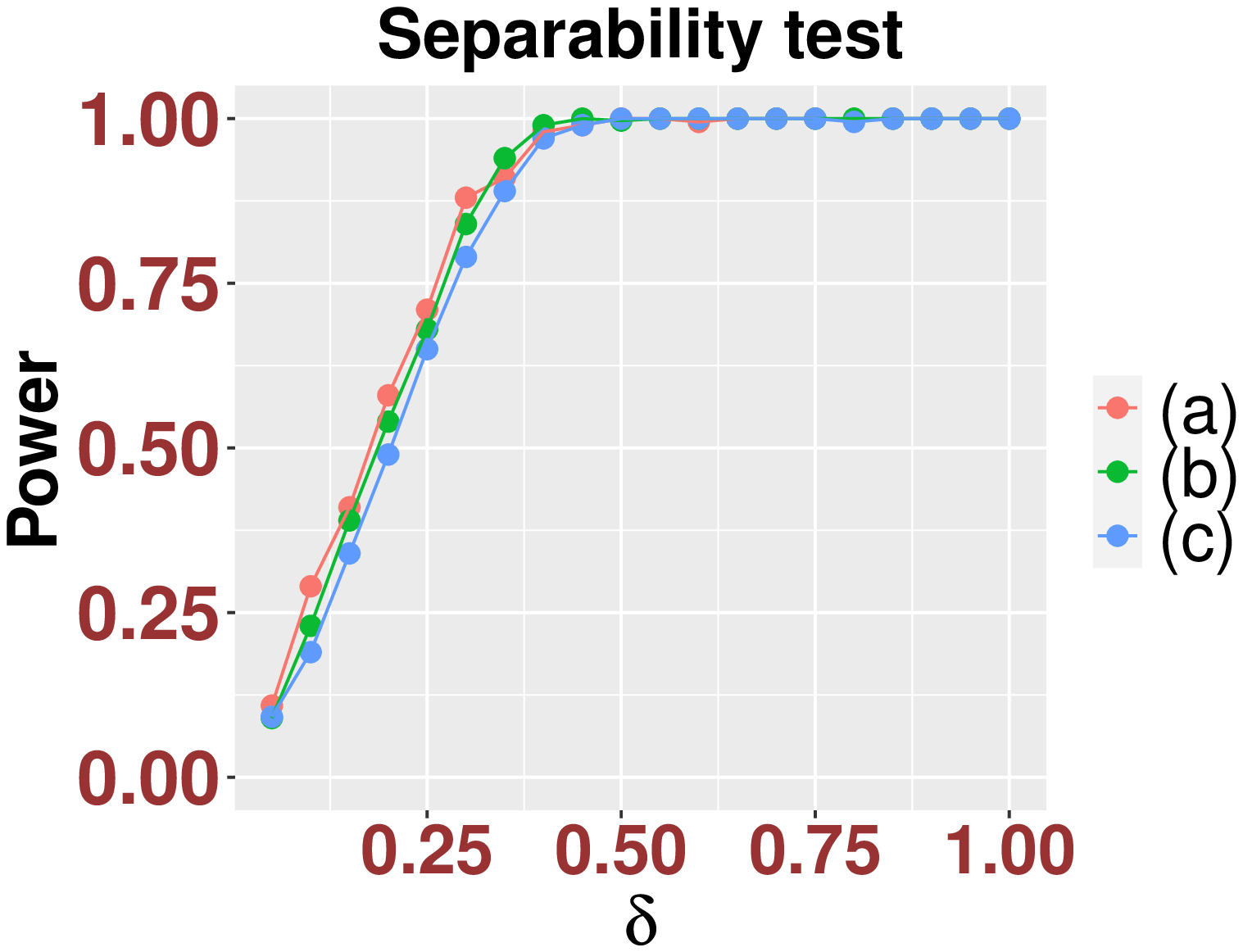}
\end{subfigure}
\begin{subfigure}{0.32\textwidth}
\includegraphics[width=6.4cm,height=5cm]{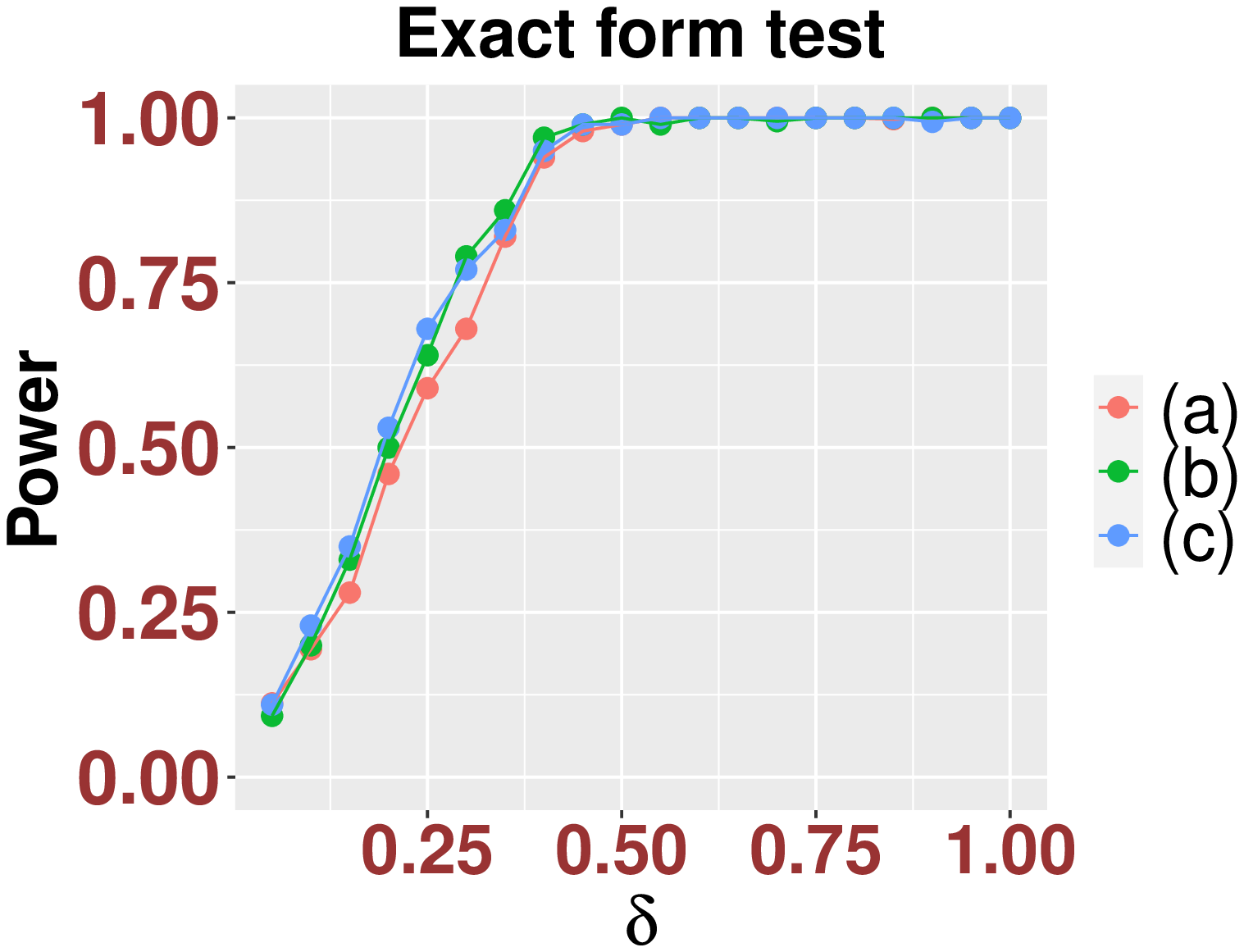}
\end{subfigure}
\caption{{ \footnotesize Simulated power for our proposed multiplier boostrapped statistics under the nominal level $0.1$. Here we used Daubechies-9 basis, $n=800$ and $(a), (b), (c)$ refer to the models for $\epsilon_i$ as in Section \ref{sec_simulationsettup}.  Our results are based on 5,000 simulations.  }  }
\label{fig_selfcomparison}
\end{figure}

\begin{figure}[!ht]
\begin{subfigure}{0.32\textwidth}
\includegraphics[width=6.4cm,height=5cm]{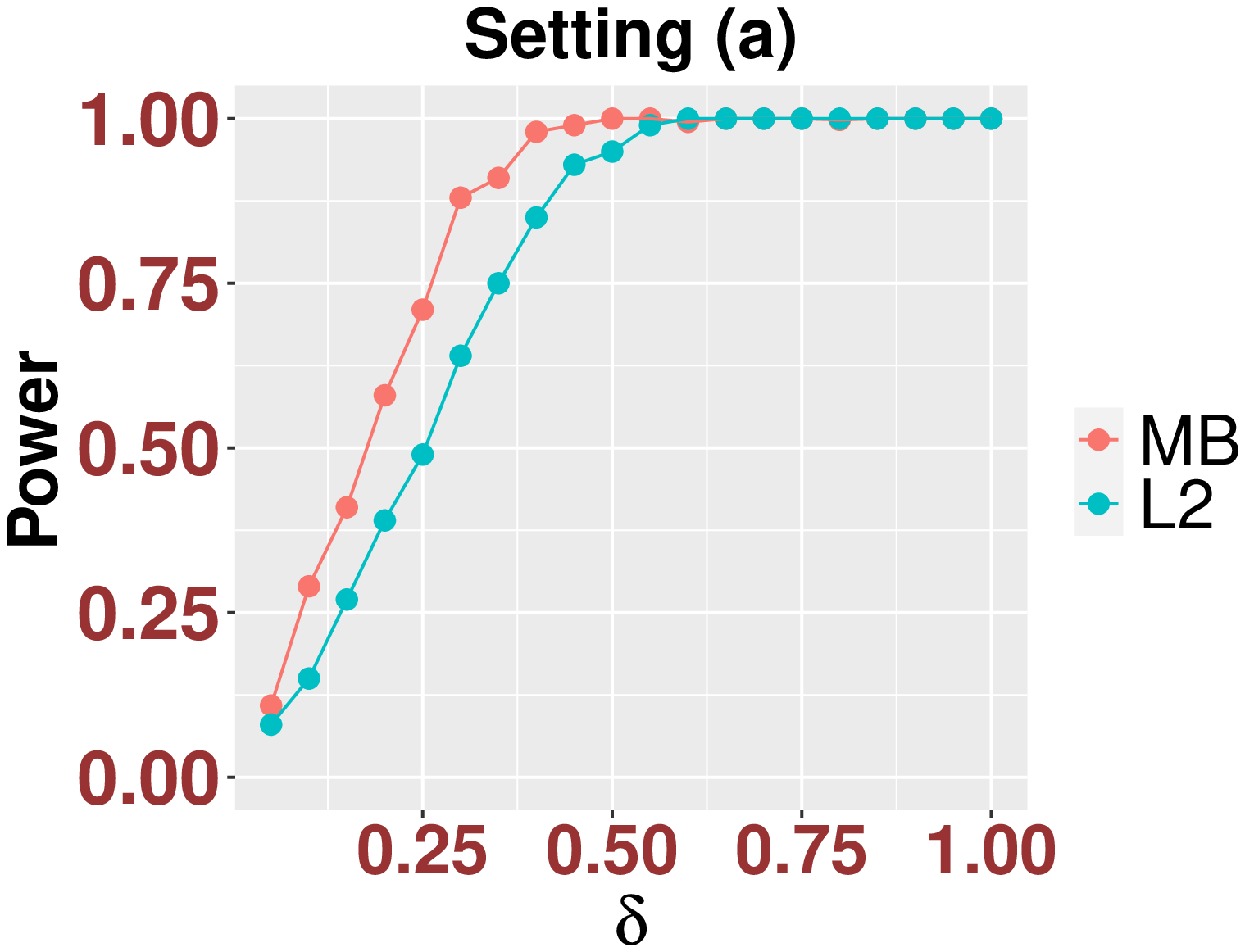}
\end{subfigure}
\begin{subfigure}{0.32\textwidth}
\includegraphics[width=6.4cm,height=5cm]{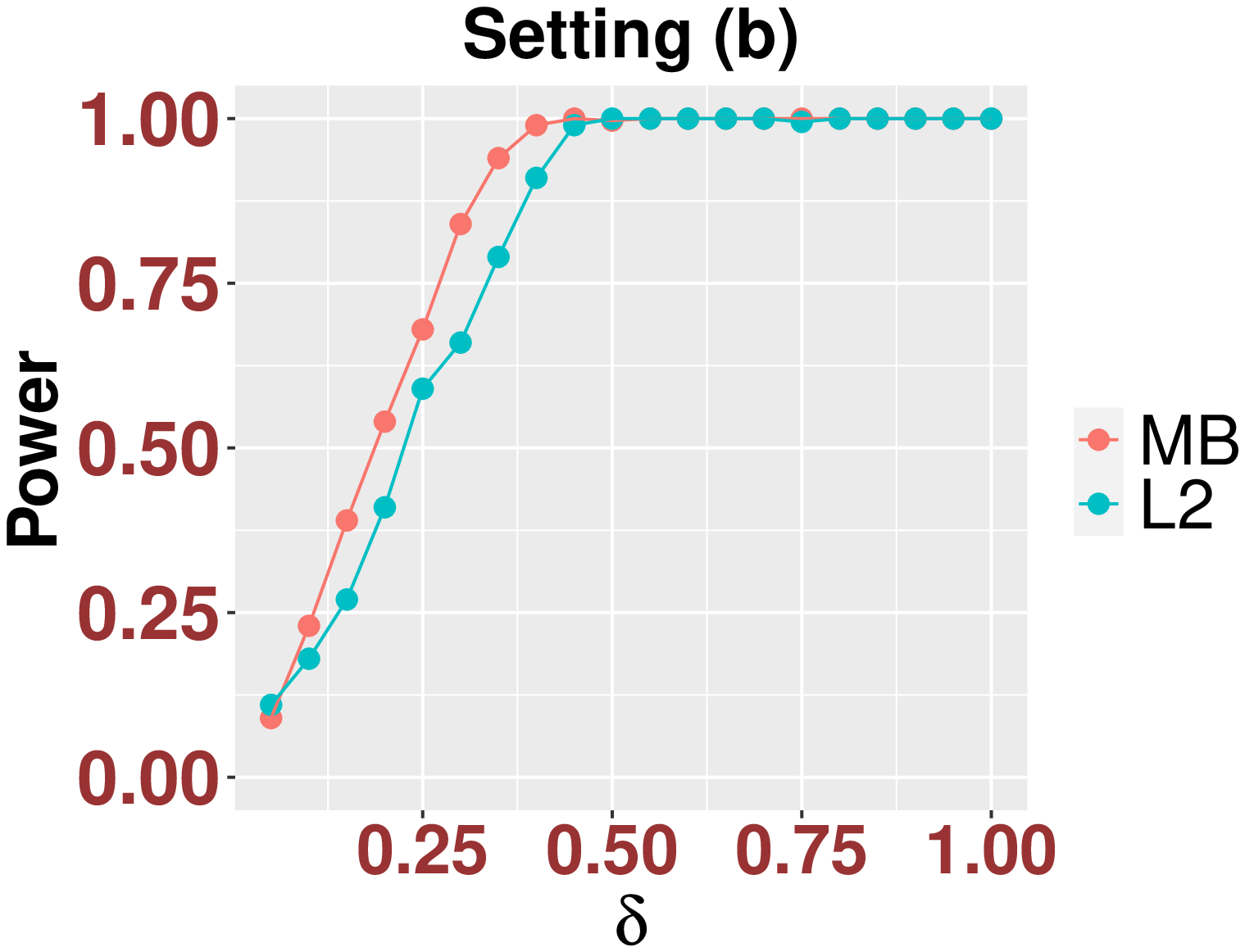}
\end{subfigure}
\begin{subfigure}{0.32\textwidth}
\includegraphics[width=6.4cm,height=5cm]{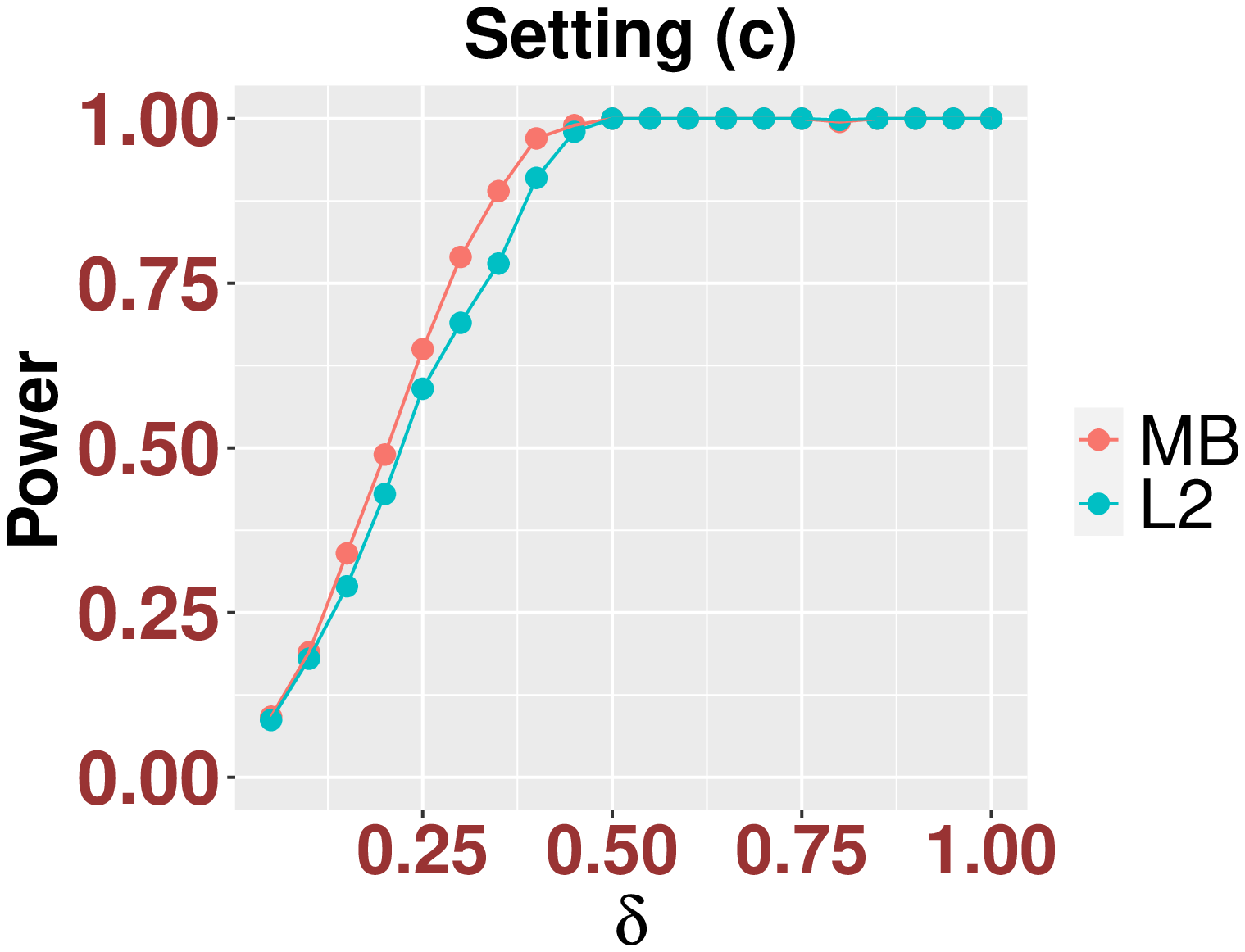}
\end{subfigure}
\caption{{ \footnotesize Comparison of the simulated power for our proposed multiplier boostrapped (MB) methd and the weighted $L_2$-distance based statistic (L2) in \cite{HHY}.  Our results are based on 5,000 simulations.  }  }
\label{fig_jasacomparison}
\end{figure}

\section{Real data analysis}\label{sec_realdata}
In this section, we apply our methodologies to study the shape of the monthly risk premium for S$\&$P 500 index as in \cite[Section 6.2]{CSW}. Let $\mu_t$ and $\sigma_t$ be the conditional mean and conditional volatility of the excess return of the market portfolio, respectively. Our goal is to uncover how $\mu_t$ relates to $\sigma_t.$ Especially, we want to understand the functional relation that $\mu_t=m(t, \sigma_t^2)$ for some unknown function $m.$ Such a problem has been studied extensively in the literature of financial economics. In the seminal work \cite{MR441271}, Merton modeled that $\mu_t=\gamma \sigma_t^2$ for some constant $\gamma$ representing the risk aversion of the agent. Recently, it has been argued in \cite{BOLLERSLEV2013409,DUKE, CHAIEB2021669, GHYSELS2014118, 10.1093/rfs/hhaa009} that a general nonlinear and time-varying function will facilitate the modeling and interpretation of the shape of the market risk premium. 

Very recently, in \cite{CSW}, the authors modeled the relationship between $\mu_t$ and $\sigma_t$ using $\mu_t=\rho(t) g(\sigma_t)$ for some unknown functions $\rho$ and $g.$ They also estimated these functions using kernel methods. Moreover, on one hand, they found their estimation were consistent with real observations empirically and statistically analysis theoretically. On the other hand, they provided some insights for the shape of the monthly risk premium for S$\&$P 500 Index based on their analysis. However, they did not justify why a separable structure was valid for modeling the risk premium. In what follows, we use our proposed multiplier bootstrap statistics to justify the separability assumption, i.e., testing (\ref{eq_hotestseparability}). 

We follow the setting of \cite[Section 6.2]{CSW}. In the notation of our model (\ref{eq:model}), $r=1$ and we set $Y_t$ to be $\mu_t=r_{mt}-r_{ft}$ which is the excess return on the S$\&$P 500 Index calculated as the difference between the monthly continuously compounded cumulative return on the index minus the monthly return on 30-day Treasury Bills. For the conditional volatility   $\sigma_t,$ we used the realized volatility (RV) measure from Oxford Man Realized Library (\url{https://realized.oxford-man.ox.ac.uk/}). To be more precise, let $RV_t$ denote the daily annualized RV during the $t$th month. Then we obtain one-month-ahead RV forecast $\mathbb{E}_{t-1}(RV_t)$ using the HAR-RV model as in \cite{10.1093/jjfinec/nbp001}, which is our $X_t.$ Our data covers the period of 31 January 2001--31 December 2018. Under the nominal level $0.05,$ using the Daubechies-9 basis functions with $c=4, d=3$ and $m=7,$ we find that the $p$-value is $0.38$ so we can conclude that the separable structural assumption is reasonable for modeling the risk premium. This supports the analysis of \cite{CSW}.   

Finally, we estimate the functions  $\rho(\cdot)$ and $g(\cdot)$ using our sieve estimators. Due to separability, as discussed in Example \ref{exam_seperabletest}, we can estimate them separability instead of using the hierarchical sieve basis functions. We also provide the simultaneous confidence bands for these functions. In Figure \ref{fig_functionsplot} blow, we report the results. We find that when the RV is fixed, the risk premium changes in a nonlinear way of time which suggests the existence of a time-varying risk aversion. Moreover, for the $g(\cdot)$ function, it is more flat and seems to have a monotone pattern.  Our findings are consistent with those in \cite{CSW}; see Figure 7 and the discussion therein. 

\begin{figure}[!ht]
\begin{subfigure}{0.5\textwidth}
\includegraphics[width=8cm,height=6.5cm]{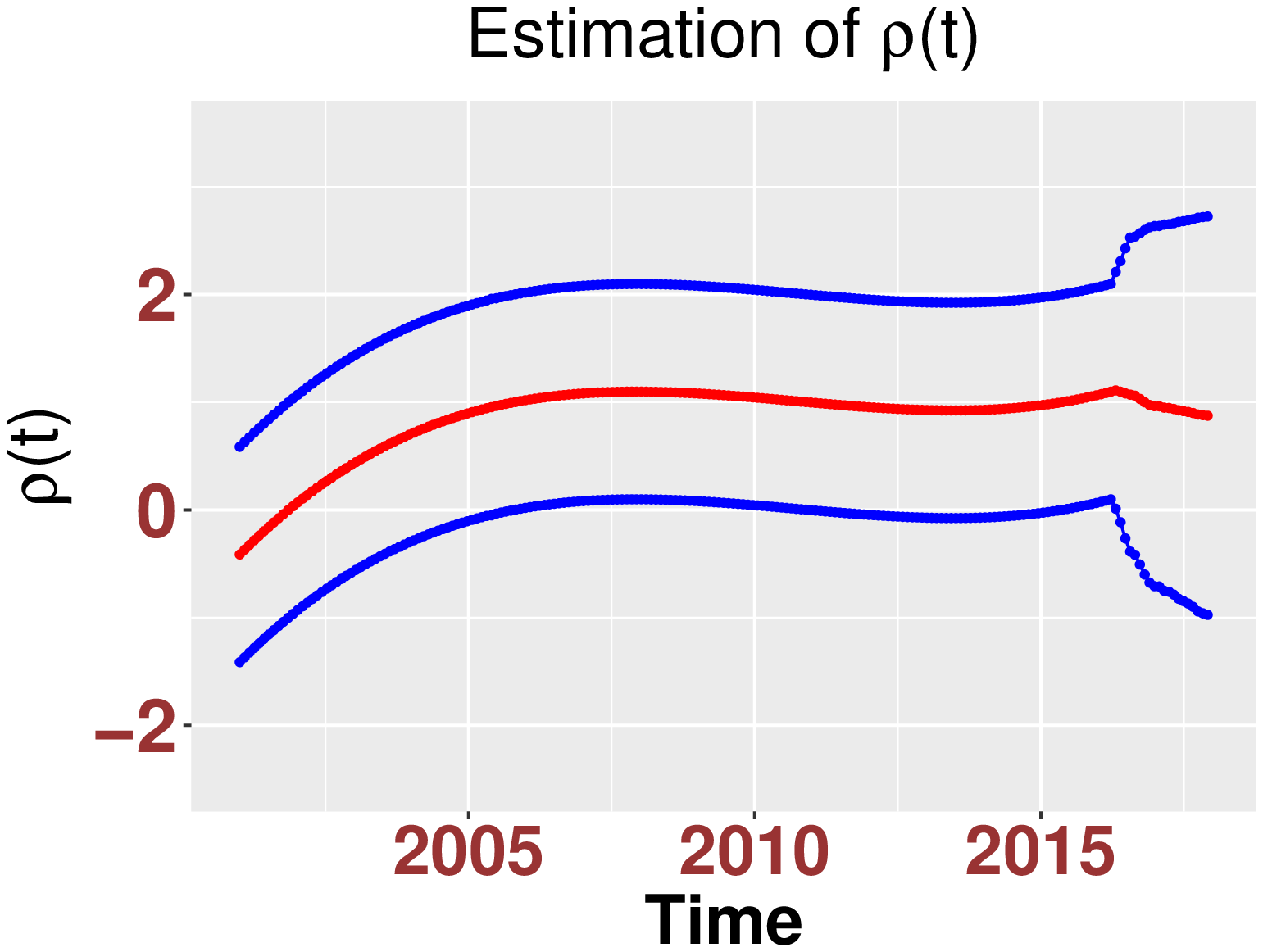}
\end{subfigure}
\begin{subfigure}{0.5\textwidth}
\includegraphics[width=8cm,height=6.5cm]{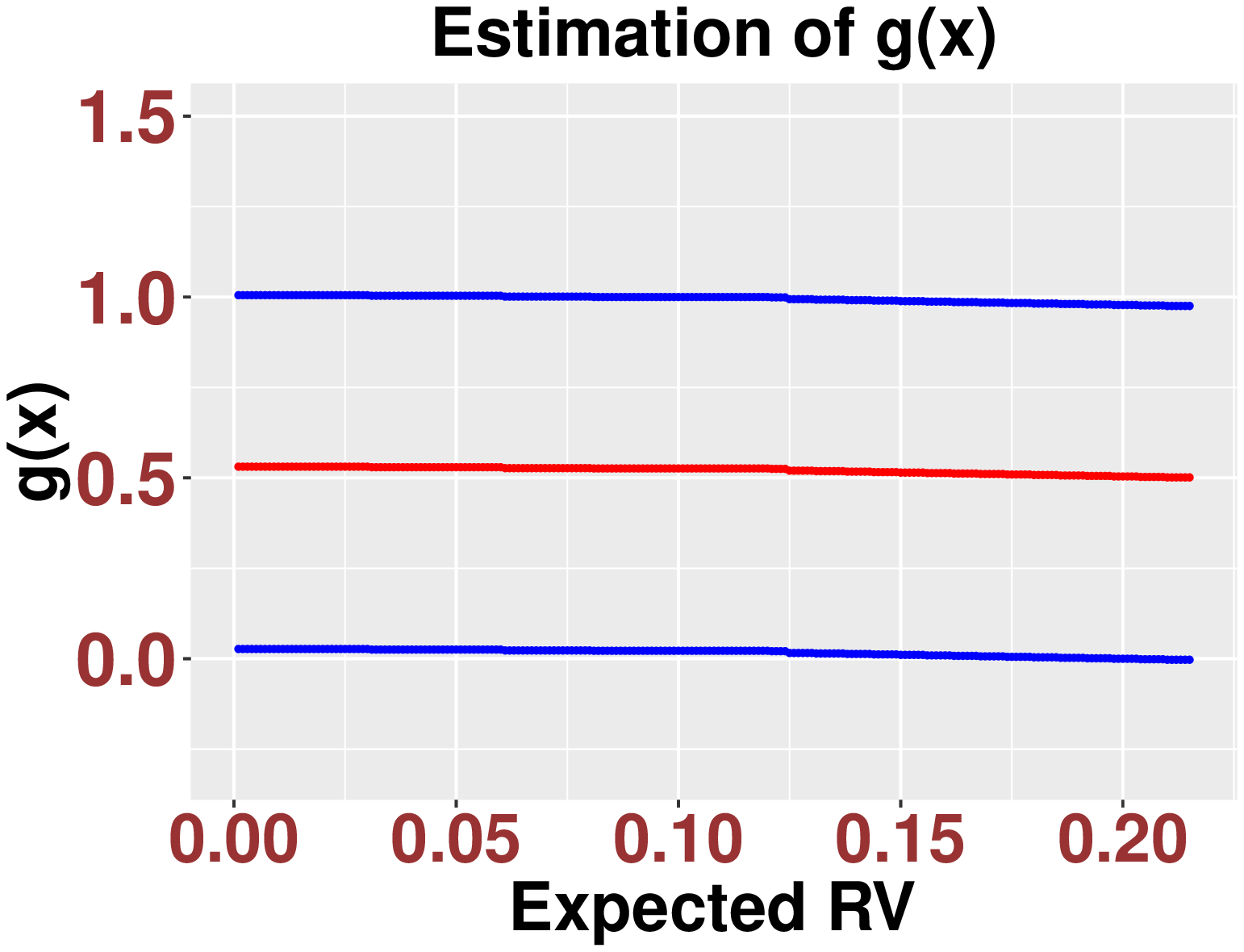}
\end{subfigure}
\caption{{ \footnotesize Estimated $\rho(t)$ and $g(x)$. The red lines are the estimated $\rho$ and $g$ and the blue lines the corresponding  simultaneous 95 $\%$ confidence bands constructed using the multiplier bootstrapping. }  }
\label{fig_functionsplot}
\end{figure}


%
%
%
%
%

\section*{Acknowledgments.} The authors are grateful to Likai Chen and Ekaterina Smetanina for providing resources of the real data analysis and many useful comments. 


\appendix
\numberwithin{equation}{section}
\section{Technical proofs}\label{sec_techinicalproof}

In this section, we provide the main technical proofs. 

\subsection{Proof of Theorem \ref{thm_approximation} and Lemma \ref{lem_locallystationaryform}}

First, the proof of Theorem \ref{thm_approximation} is standard. When $\mathbb{R}=[0,1],$ the results have been established without using the mapped basis functions; see \cite[Section 2.3.1]{CXH}. In our setting, since we are using the mapped sieves to map $\mathbb{R}$ to $[0,1],$ the proof is similar. We only point out the main routine here.

\begin{proof}[\bf Proof of Theorem \ref{thm_approximation}] For any fixed $x \in \mathbb{R},$ the approximation rate has been summarized in Lemma \ref{lem_deterministicapproximation} for $t.$ Similarly, for any fixed $t \in \mathbb{R},$ the results have been recorded in Lemma \ref{lem_deterministicapproximation2}. Based on these results, we can follow the arguments of \cite[Section 5.3]{MR1262128} to conclude the proof. We point out the \cite[Section 5.3]{MR1262128} deals with the $L_2$ norm assuming that the $L_2$ norm of the partial derivatives are bounded. However, as discussed in \cite[Section 6.2]{MR1176949}, the results can be generalized to the sup-norm under Assumption \ref{assum_smoothnessasumption}.
\end{proof}


Then we prove Lemma \ref{lem_locallystationaryform}. 

\begin{proof}[\bf Proof of Lemma \ref{lem_locallystationaryform}]
Due to similarity, we only prove the results for $\{\bm{x}_i\}$ under the setting Case (1).  First of all, under Assumption \ref{assum_models}, we notice that for any $1 \leq i \leq d, 1 \leq j \leq r,$ 
\begin{equation}\label{eq_meanzero}
\mathbb{E}(\varphi_i(X_{k-j}) \epsilon_k)=\mathbb{E}(\mathbb{E}(\varphi_i(X_{k-j}) \epsilon_k| X_{k-j}))=0. 
\end{equation}
This shows that $\{\bm{x}_i\}$ is mean zero. Second, for any $1 \leq k \leq rd,$ by (\ref{eq_ddd}) and (\ref{eq_desigmatrixform}), we have that
\begin{equation}\label{eq_exactform}
x_{ik}=\varphi_{\ell_2(k)}(X_{i-\ell_1(k)-1})\epsilon_i.
\end{equation}
Under the model setting (\ref{eq_epsilons1form}), it is clear that for some measurable function $U_k(t,\cdot),$ we can write
\begin{equation*}
\bm{x}_{ik}=U_k\left(\frac{i}{n}, \mathcal{F}_i\right).
\end{equation*} 
Third, for $s,t \in [0,1],$ using (\ref{eq_exactform}), the smoothness of the basis functions and the assumption (\ref{eq_slc}), we have that for some constant $C>0$
\begin{equation*}
\sup_i \| U_k(s, \mathcal{F}_i)-U_k(t, \mathcal{F}_i) \|_q \leq C|t-s|.
\end{equation*}
Finally, for any $1 \leq k \leq rd,$ we have that for some constant $C>0$
\begin{align*}
\sup_t \| U_k(t, \mathcal{F}_0)-U_{k}(t, \mathcal{F}_{0,j}) \|_q & \leq  \sup_t \| \varphi_{\ell_2(k)}(G_1(t,\mathcal{F}_0)) G_2(t, \mathcal{F}_0)-\varphi_{\ell_2(k)}(G_1(t,\mathcal{F}_0)) G_2(t, \mathcal{F}_{0,j}) \|_q \\
& + \sup_t \|\varphi_{\ell_2(k)}(G_1(t,\mathcal{F}_0)) G_2(t, \mathcal{F}_{0,j})-\varphi_{\ell_2(k)}(G_1(t,\mathcal{F}_{0,j})) G_2(t, \mathcal{F}_{0,j}) \|_q  \\
& \leq C j^{-\tau},
\end{align*}
where in the last inequality we used Assumption \ref{assum_physical} and the smoothness of the basis function. This finishes our proof. The other case and $\{\bm{w}_i\}$ can be proved similarly. We omit the details. 
\end{proof}

\subsection{Consistency of the proposed estimators: proof of Theorem \ref{thm_consistency}}\label{sec_consistencyproof}

In this subsection, we prove Theorem \ref{thm_consistency}. Till the end of the paper, for a positive definite $H,$ we denote $\|H \|_{\op}$ as the operator norm of $H,$ i.e., the largest eigenvalue of $H.$

\begin{proof}[\bf Proof of Theorem \ref{thm_consistency}] Since $r$ is finite, without loss of generality, we assume $r=1.$ For general $r,$ the proof can be modified verbatim with only additional notational complicatedness. In what follows, we omit the subscript $j,$ i.e., $m \equiv m_{1}$ and $\widehat{m}_{c,d} \equiv \widehat{m}_{1,c,d}.$  The starting point of our proof is the following decomposition 
\begin{align}\label{eq_initialdecomposition}
\|\widehat{m}_{c,d}(t,x)-m(t,x)\| & \leq \|\widehat{m}_{c,d}(t,x)-m_{c,d}(t,x)\|+|m_{c,d}(t,x)-m(t,x)|.  
\end{align}   
Since the second term of the right-hand side of (\ref{eq_initialdecomposition}) can be bounded using Theorem \ref{thm_approximation}, it suffices to control the first term. Moreover, using the definitions (\ref{eq_firststeptruncation1}), (\ref{eq_firststeptruncation}) and (\ref{eq_proposedestimator}), by Cauchy-Schwarz inequality, 
\begin{equation}\label{eq_ll1}
\|\widehat{m}_{c,d}(t,x)-m_{c,d}(t,x)\| \leq  \| \widehat{\bm{\beta}}-\bm{\beta} \| \zeta.  
\end{equation}
The rest of the proof leaves to control $\widehat{\bm{\beta}}-\bm{\beta}.$ Recall that
\begin{equation}\label{eq_betadifferenceexpression}
\widehat{\bm{\beta}}-\bm{\beta}=(W^\top W)^{-1} W^\top \bm{\epsilon}. 
\end{equation}
Therefore, we need to control $\| \widehat{\bm{\beta}}-\bm{\beta} \|$ which satisfies
\begin{equation}\label{eq_betabound1}
\| \widehat{\bm{\beta}}-\bm{\beta} \| \leq \left\| \frac{1}{\lambda_p((n^{-1} W^\top W))} \right\| \| n^{-1} W^\top \bm{\epsilon} \|. 
\end{equation} 


First, we establish the convergence results for the matrix $W^\top W. $ Denote $W_i$ as the $i$th column of $W.$ It suffices to analyze the terms $W_i^\top W_j.$ Without loss of generality, we focus our discussion on the terms $W_1^\top W_1$ and $W_1^\top W_2.$ The general cases of $i,j$ can be handled similarly. We start with the term $W_1^\top W_1.$ Let $W_1=(W_{12}, \cdots, W_{1n})^\top.$ Recall (\ref{eq_designmatrix}). We have 
\begin{equation*}
W_{1k}=\phi_1(t_k) \varphi_1(X_{k-1}), \ 2 \leq k \leq n.
\end{equation*}
Consequently, we have that
\begin{equation*}
\frac{1}{n} W_1^\top W_1=\frac{1}{n} \sum_{k=2}^n \phi_1(t_k)^2 \varphi_1(X_{k-1})^2.  
\end{equation*}
On one hand, by a discussion similar to Lemma \ref{lem_locallystationaryform}, we can show that $\mathfrak{h}_k:=\phi_1(t_k)^2 \varphi_1(X_{k-1})^2, 2 \leq k \leq n,$ is a sequence of locally stationary time series whose physical dependence measure satisfies that 
\begin{equation*}
\delta_{\mathfrak{h}}(j,q) \leq C \xi_c^2 j^{-\tau}, \ \text{for some constant} \ C>0. 
\end{equation*}
Together with (1) of Lemma \ref{lem_concentration}, we obtain that
\begin{equation*}
\left\|\frac{1}{n} W_1^\top W_1-\frac{1}{n}\sum_{k=2}^n \phi_1(t_k)^2 \mathbb{E} \varphi_1(X_{k-1})^2  \right\|_q \leq \frac{C \xi_c^2}{\sqrt{n}},
\end{equation*}
where we used the assumption that $q>2.$ 
Moreover, by (2) of Lemma \ref{lem_concentration}, we have that 
\begin{equation}\label{eq_cccc}
\left| \frac{1}{n}\sum_{k=2}^n \phi_1(t_k)^2 \mathbb{E} \varphi_1(X_{k-1})^2 -\frac{1}{n} \sum_{k=2}^n \phi_1(t_k)^2 \Pi_{11}(t_k) \right| \leq C \xi_c^2 n^{-1+\frac{2}{\tau+1}},
\end{equation}
where $\Pi_{11}(t)$ is the first entry of $\Pi(t)$ defined in (\ref{eq_longrunwitht}).
Further, by Lemma \ref{lem_intergralappoximation}, we 
have that
\begin{equation}\label{eq_ccc2}
\left|\frac{1}{n} \sum_{k=2}^n \phi_1(t_k)^2 \Pi_{11}(t_k)-\int_0^1 \phi_1(t)^2 \Pi_{11}(t) \dd t \right|=\OO(n^{-2}). 
\end{equation}
We point out that $\int_0^1 \phi_1(t)^2 \Pi_{11}(t)\dd t=\Pi_{11}$ which is the first entry of $\Pi$ as in (\ref{eq_longruncovariancematrix}). Combining with the above arguments, we conclude that 
\begin{equation*}
\left\| \frac{1}{n} W_1^\top W_1- \Pi_{11} \right\|_q \leq C \left( \frac{\xi^2_c}{\sqrt{n}}+\frac{\xi^2_c n^{\frac{2}{\tau+1}}}{n}\right).
\end{equation*}
Similarly, we can show that 
\begin{equation*}
\left\| \frac{1}{n} W_1^\top W_2- \Pi_{12} \right\|_q \leq C \left( \frac{\xi^2_c}{\sqrt{n}}+\frac{\xi^2_c n^{\frac{2}{\tau+1}}}{n}\right).
\end{equation*}
Together the above results with Lemma \ref{lem_circle}, we conclude that 
\begin{equation}\label{eq_consistencyconvergency}
\left\| \frac{1}{n} W^\top W-\Pi \right\|_{\op} =\OO_{\mathbb{P}} \left( p\left( \frac{\xi^2_c}{\sqrt{n}}+\frac{\xi^2_c n^{\frac{2}{\tau+1}}}{n}\right) \right),
\end{equation}
where we used the fact that $q>2$ and recall that $p$ is defined in (\ref{eq_defnp}). Together with Assumption \ref{assum_updc} and (\ref{eq_parameterassumption}), we find that 
\begin{equation}\label{eq_boundone}
\left\| \frac{1}{\lambda_p(n(W^\top W)^{-1})} \right\|=\OO(1). 
\end{equation}

Second, we control the error term $\frac{W^\top \bm{\epsilon}}{n}.$ Without loss of generality, we focus on its first entry and consider
\begin{equation}\label{eq_entrywiseexpansion}
\frac{\left[W^\top \bm{\epsilon} \right]_{11}}{n}=\frac{1}{n}\sum_{k=2}^n \phi_1(t_k) \varphi_1(X_{k-1}) \epsilon_{k}.
\end{equation} 
By Lemma \ref{lem_locallystationaryform}, (1) of Lemma \ref{lem_concentration} and $q>2$, we conclude that 
\begin{equation*}
\left \|\frac{\left[W^\top \bm{\epsilon} \right]_{11}}{n} \right \|^2 \leq \frac{C\xi_c^2}{n}.
\end{equation*} 
Consequently, we have that 
\begin{equation}\label{eq_boundtwo}
\left \|\frac{W^\top \bm{\epsilon}}{n} \right \| \leq  \frac{C \xi_c \sqrt{p}}{\sqrt{n}}.   
\end{equation}

In summary,  by (\ref{eq_boundone}) and (\ref{eq_boundtwo}), in view of (\ref{eq_betabound1}), we obtain that 
\begin{equation}\label{eq_betabound}
\|\widehat{\bm{\beta}}-\bm{\beta} \| \leq C \xi_c \sqrt{\frac{p}{n}}. 
\end{equation} 
Together with (\ref{eq_initialdecomposition}) and (\ref{eq_ll1}), we have completed the proof. 

%
%
%
 

\end{proof}

\subsection{High dimensional Gaussian approximation and proof of Theorems \ref{thm_asymptoticdistribution} and \ref{thm_poweranalysis}}\label{sec_gassuianapproximation}
In this subsection, we prove Theorems \ref{thm_asymptoticdistribution} and \ref{thm_poweranalysis}. The key ingredient is to establish the Gaussian approximation for the statistics $\mathsf{T}_{1j}$ and $\mathsf{T}_{2j}.$ The starting point is to rewrite the statistics $\mathsf{T}_{1j}$ and $\mathsf{T}_{2j}$ in more explicit forms. We first prepare some notations. Without loss of generality, we only explain the proof for $r=1$ as  discussed in the beginning of the proof of Theorem \ref{thm_consistency} and omit the subscript $j$ in the sequel. According to (\ref{eq_initialdecomposition}), under the assumption of (\ref{eq_parameterassumption}),
\begin{align} \label{eq_fundementalexpression}
\widehat{m}_{c,d}(t,x)-m_{c,d}(t,x)&= \bm{b}^\top (\widehat{\bm{\beta}}-\bm{\beta}) \nonumber \\
& =\frac{1}{\sqrt{n}}\bm{b}^\top \Pi^{-1} (\frac{1}{\sqrt{n}} W^\top \bm{\epsilon})\left(1+\oo_{\mathbb{P}}(1)  \right),
\end{align}
where we used (\ref{eq_betadifferenceexpression}) and  (\ref{eq_consistencyconvergency}). Recall that the deterministic vector $\bm{l}$ and random vector $\bm{z}$ are defined as
\begin{equation}\label{eq_originaldefinition}
\bm{l}=\Pi^{-1} \bm{b}, \ \bm{z}=\frac{1}{\sqrt{n}} W^\top \bm{\epsilon},
\end{equation}
and $\bm{x}_i$ is defined in (\ref{eq_locallystationaryform}) and (\ref{eq_ddd}). Denote $\bm{y}_i \in \mathbb{R}^p$ such that 
\begin{equation}\label{eq_decompositionkronecker}
\bm{y}_i=\bm{x}_i  \otimes \ab(t_i),
\end{equation}
where we recall $\ab(t_i)=(\phi_1(t_i), \cdots, \phi_c(t_i))^\top$ and $\otimes$ stands for the Kronecker product. Using (\ref{eq_designmatrix}), it is easy to see that  
\begin{equation}\label{eq_bmzgreatform}
\bm{z}=\frac{1}{\sqrt{n}}\sum_{i=1}^n \bm{y}_i.
\end{equation}
Therefore, in view of (\ref{eq_originaldefinition}),  the analysis of $\mathsf{T}_1$ reduces to study $\bm{l}^\top \bm{z}$ using (\ref{eq_bmzgreatform}) and (\ref{eq_decompositionkronecker}). 

Moreover, in view of (\ref{eq_fundementalexpression}), we have 
\begin{equation}\label{eq_t2expansion}
\int_{[0,1]}\int_{\mathbb{R}}[\hat m_{c,d}(t,x)-m_{0, c,d}(t,x)]^2 \dd t \dd x=\frac{1}{n} \bm{z}^\top \mathsf{W} \bm{z}(1 +\oo_{\mathbb{P}}(1)),
\end{equation}
where $\mathsf{W}$ is independent of $t$ and $x$ and defined in (\ref{eq_defnw}) and $m_{0,c,d}(t,x)$ is defined similarly to (\ref{eq_firststeptruncation1}) for $m_{0}(t,x)$ using the basis $\{b_{\ell_1, \ell_2}(t,x)\}, 1 \leq \ell_1 \leq c, 1 \leq \ell_2 \leq d$. Consequently, the analysis of $\mathsf{T}_2$ reduces to  exploring $\bm{z}^\top \mathsf{W} \bm{z}.$ 

Based on the above discussion, we have seen that it suffices to establish the Gaussian approximation theory for either the affine form $\bm{l}^\top \bm{z}$ or the quadratic form $\bm{z}^\top \mathsf{W} \bm{z}.$ Before stating the Gaussian approximation results,  we first pause to record the covariance structure of $\bm{z}$ in Lemma \ref{lem_covarianceofz}. It indicates that $\operatorname{Cov}(\bm{z})$ is close to $\Omega$ as defined in (\ref{eq_longruncovariancematrix}).

\begin{lemma}\label{lem_covarianceofz}
Denote $\operatorname{Cov}(\bm{z})$ as the covariance matrix of $\bm{z}.$ Suppose the assumptions of Theorem \ref{thm_consistency} hold. Then we have that
\begin{equation*}
\|\operatorname{Cov}(\bm{z})-\Omega \|_{\op}=\OO\left( \frac{p\xi_c^2}{n} +\frac{p\xi_c^2 n^{2/\tau}}{\sqrt{n}} +p\xi_c^2 n^{-1+\frac{2}{\tau+1}}\right).
\end{equation*}
\end{lemma}
\begin{proof}
Denote $\bm{z}=(z_1, \cdots, z_p)^\top.$ We control the error entrywisely and focus on $[\operatorname{Cov}(\bm{z})]_{11}=\operatorname{Var}(z_1)$. Note that 
\begin{equation}\label{eq_z1form}
z_1=\frac{1}{\sqrt{n}}\sum_{k=2}^n \phi_1(t_k) \varphi_1(X_{k-1}) \epsilon_{k}.
\end{equation}
Using the notation (\ref{eq_ddd}), we conclude that 
\begin{align}\label{eq_decompositionvariance}
\operatorname{Var}(z_1) &=\frac{1}{n} \sum_{k_1, k_2=2}^n \phi_1(t_{k_1}) \phi_1(t_{k_2}) \mathbb{E} x_{k_1 1} x_{k_2 1} \nonumber \\
&=\frac{1}{n}\sum_{i=2}^n \phi_1(t_{i})^2 \mathbb{E} x^2_{i 1}+\frac{1}{n} \sum_{k_1 \neq k_2}^n \phi_1(t_{k_1}) \phi_1(t_{k_2}) \mathbb{E} x_{k_1 1} x_{k_2 1}:=\mathsf{E}_1+\mathsf{E}_2,
\end{align}

First, by a discussion similar to (\ref{eq_cccc}) and (\ref{eq_ccc2}), we  find that
\begin{equation}\label{eq_e1finalerror}
\left| \mathsf{E}_1-\Omega_{11}\right|=\OO\left(\xi_c^2 n^{-1+\frac{2}{\tau+1}}\right).
\end{equation}
Second, for $\mathsf{E}_2,$ by Lemma \ref{lem_locallystationaryform} and (3) of Lemma \ref{lem_concentration}, we find that for some constant $C>0$ 
\begin{equation*}
|\mathbb{E} x_{k_1 1} x_{k_2 1}| \leq C|k_1-k_2|^{-\tau}. 
\end{equation*}
Consequently,we obtain that
\begin{equation}\label{eq_e2part}
\left| \mathsf{E}_2 \right| \leq \frac{\xi_c^2}{n}+\left|\frac{1}{n} \sum_{k_1=2}^n \phi_1(t_{k_1}) \sum_{|k_1-k_2| \leq n^{2/\tau}}  \phi_1(t_{k_2})  \mathbb{E} x_{k_1 1} x_{k_2 1} \right|.
\end{equation}
For each fixed $k_2,$ by a discussion similar to Lemma \ref{lem_locallystationaryform}, we see that $\{x_{k_1 1} x_{k_2 1}\}$ is a locally stationary time series whose physical dependence measure satisfies that $\delta(i,q) \leq C i^{-\tau}.$ As $q>2,$ applying (1) of Lemma \ref{lem_concentration}, we readily see that 
\begin{equation*}
\left| \frac{1}{n} \sum_{k_1=2}^n \phi_1(t_{k_1}) \mathbb{E} x_{k_1 1} x_{k_2 1} \right| \leq \frac{C \xi_c}{\sqrt{n}}.
\end{equation*}
Together with (\ref{eq_e2part}), we arrive at
\begin{equation}\label{eq_e2finalerror}
|\mathsf{E}_2| \leq C \left( \frac{\xi_c^2}{n} +\frac{\xi^2_c n^{2/\tau}}{\sqrt{n}}\right).
\end{equation} 
By (\ref{eq_decompositionvariance}), (\ref{eq_e1finalerror}) and (\ref{eq_e2finalerror}), we obtain that 
\begin{equation*}
|\operatorname{Var}(z_1)-\Omega_{11}|=\OO\left( \frac{\xi_c^2}{n} +\frac{\xi_c^2 n^{2/\tau}}{\sqrt{n}} +\xi^2_c n^{-1+\frac{2}{\tau+1}}\right).
\end{equation*}
The general term $\operatorname{Cov}(z_i, z_j)$ can be analyzed similarly. We can therefore conclude our proof using Lemma \ref{lem_circle}. 
\end{proof}

Next, we state the Gaussian approximation results for both the affine and quadratic forms. Consider a sequence of centered Gaussian random vectors $\{\bm{n}_i\}_{i=2}^n$ in $\mathbb{R}^p$ which preserve the covariance structure of $\{\bm{x}_i\}.$ Denote
\begin{equation*}
\bm{g}_i=\bm{n}_i \otimes \ab(t_i),
\end{equation*} 
and corresponding to (\ref{eq_bmzgreatform}), denote
\begin{equation}\label{eq_wgaussiandefinition}
\bm{w}=\frac{1}{\sqrt{n}} \sum_{i=2}^n \bm{g}_i. 
\end{equation}
Recall $\bm{l}$ in (\ref{eq_originaldefinition}) and $\mathsf{W}$ in (\ref{eq_defnw}). The Gaussian approximation result, Theorem \ref{thm_gaussianapproximationcase}, provides controls on the following Kolmogorov distances
\begin{equation*}
\mathcal{K}_1(\bm{z}, \bm{w}):=\sup_{x \in \mathbb{R}} \left| \mathbb{P}(\bm{z}^\top \bm{l} \leq x)-\mathbb{P}(\bm{w}^\top \bm{l} \leq x) \right|, \ \mathcal{K}_2(\bm{z}, \bm{w}):=\sup_{x \in \mathbb{R}} \left| \mathbb{P}(\bm{z}^\top \mathsf{W} \bm{z} \leq x)-\mathbb{P}(\bm{w}^\top \mathsf{W} \bm{w} \leq x) \right|. 
\end{equation*}

\begin{theorem}\label{thm_gaussianapproximationcase}
Suppose Assumptions \ref{assum_models}--\ref{assum_updc} hold. Moreover, we assume that 
for some large values $m$ and $\hd,$ $\Theta(\eta)$ in (\ref{eq_controlparameter}) satisfies 
\begin{equation}\label{eq_mhchoice}
\Theta(\eta)=\oo(1), \ \eta>0.
\end{equation}
Then we have that 
\begin{equation*}
\mathcal{K}_1(\zb, \wb)=\OO(\Theta(1)), \  \mathcal{K}_2(\zb, \wb)=\OO(\Theta(0.5)). 
\end{equation*}
\end{theorem}

\begin{remark}
Theorem \ref{thm_gaussianapproximationcase} establishes the asymptotic Gaussian fluctuation for the affine and quadratic forms by controlling their Kolmogorov distances with their Gaussian counterparts. We point that (\ref{eq_mhchoice}) is a mild assumption and can be easily satisfied. For example, when $\tau$ is large enough, we only require $p \ll n^{2/7},$ which matches the best known dimension as in \cite{MR3571252}.  
\end{remark}

Before proving Theorem \ref{thm_gaussianapproximationcase}, we show how it implies Theorems \ref{thm_asymptoticdistribution} and \ref{thm_poweranalysis}.

\begin{proof}[\bf Proof of Theorem \ref{thm_asymptoticdistribution}] As before, without loss of generality, we focus on the case $r=1$ and omit the subscript $j.$ Under the assumption of (\ref{eq_onebound}), we find that Theorem \ref{thm_gaussianapproximationcase} holds for $\mathcal{K}_1(\bm{z}, \bm{w}).$ Together with   (\ref{eq_fundementalexpression}) and (\ref{eq_originaldefinition}), by Theorem \ref{thm_consistency} and the assumption of (\ref{eq_assumptionerrorreduce}), we find that 
\begin{equation}\label{eq_t1c}
\frac{\sqrt{n}\mathsf{T}_1}{\sqrt{\operatorname{Var}(\bm{w}^\top \bm{l})}} \simeq \mathcal{N}(0,1). 
\end{equation}
Since $\bm{w}$ is Gaussian, we have that
\begin{equation}\label{eq_t1cc}
\bm{l}^\top \bm{w} \sim \mathcal{N}(0, \bm{l}^\top \operatorname{Cov}(\bm{w}) \bm{l}). 
\end{equation}
By the construction of $\bm{w},$ the assumption (\ref{eq_onebound}) and  Lemma \ref{lem_covarianceofz}, we conclude that
\begin{equation}\label{eq_t1ccc}
\| \operatorname{Cov}(\bm{w})-\Omega \|_{\op}=\oo(1). 
\end{equation}
Recall (\ref{eq_defhtx}). We can conclude our proof of (\ref{eq_thmonepartone}) using (\ref{eq_t1c}), (\ref{eq_t1cc}) and (\ref{eq_t1ccc}).

Next, we prove (\ref{eq_expansionformula}) using Lemma \ref{lem_volumeoftube}. Since $\sqrt{n} \mathtt{T}_1$ is asymptotically a Gaussian process whose convergence rate can be controlled in Theorem \ref{thm_gaussianapproximationcase}, it suffices to focus on the Gaussian case. Recall (\ref{eq_ljdefinition}).  By (\ref{eq_betaolsform}) and (\ref{eq_proposedestimator}), we can write
\begin{equation*}
\widehat{m}_{c,d}(t,x)=\bm{b}^\top (W^\top W)^{-1} W^\top \bm{Y}=l(t,x)^\top W^\top \bm{Y}+\bm{b}^\top \left( (n^{-1}W^\top W)^{-1}-\Pi^{-1}) \right)\left[ n^{-1} W^\top \bm{Y} \right]. 
\end{equation*}
By (\ref{eq_consistencyconvergency}) and Assumption \ref{assum_updc}, we have that for some constant $C>0,$
\begin{equation*}
\left\| \bm{b}^\top \left( (n^{-1}W^\top W)^{-1}-\Pi^{-1}) \right)\left[ n^{-1} W^\top \bm{Y} \right]\right \| \leq C \zeta \sqrt{p} \left[  p\left( \frac{\xi^2_c}{\sqrt{n}}+\frac{\xi^2_c n^{\frac{2}{\tau+1}}}{n}\right) \right],
\end{equation*}
where in the second step we used a discussion similar to (\ref{eq_boundtwo}). Consequently, under the assumption of  (\ref{eq_onebound}), in view of (\ref{eq_ll}), we can further write
\begin{equation}\label{eq_rrr}
\widehat{m}_{c,d}(t,x)=\widetilde{l}(t, \widetilde{x})^\top W^\top \bm{Y}+\oo_{\mathbb{P}}(1).
\end{equation}
Since the expression (\ref{eq_alphatheortrep}) is independent of the temporal relation of $\bm{\epsilon},$ using (\ref{eq_defntdefn}) and (\ref{eq_rrr}), we can write 
\begin{equation*}
\alpha=\p\left( \sup_{x \in \mathcal{X}} \left| T(t,\widetilde{x})^\top \widetilde{\bm{\epsilon}} \right| \geq c_{\alpha}+\oo(1)  \right),
\end{equation*} 
where $\widetilde{\bm{\epsilon}} \sim \mathcal{N}(0, \mathbf{I}_{n-1})$ and $\bm{\epsilon}=\sqrt{\operatorname{Cov}(W^\top \bm{\epsilon})}\widetilde{\bm{\epsilon}}.$ By Lemma \ref{lem_volumeoftube} and the fact that $c_{\alpha}>0$, we complete the proof of (\ref{eq_expansionformula}).

Finally, we prove (\ref{eq_quantitiesbound}). According to equation (3.2) of \cite{MR1311978}, we have that
\begin{equation}\label{eq_kappodefinition}
\kappa_0=\int_0^1 \int_0^1 \sqrt{\det(A A^\top )} \dd t \dd x,
\end{equation}
where $A \in \mathbb{R}^{ 2 \times (n-1)}$ is defined as $A^\top=\left( \frac{\partial}{\partial t } T,  \frac{\partial}{\partial \widetilde{x} } T \right).$ Here we recall (\ref{eq_defntdefn}) for the definition of $T.$ By an elementary calculation, it is easy to see that 
\begin{equation*}
\frac{\partial T}{\partial t}=\frac{\sqrt{n} l_t \sqrt{\operatorname{Cov}(W^\top\bm{\epsilon})}}{\widetilde{h}(t, \widetilde{x})}-\frac{1}{2} \frac{T^\top(t,\widetilde{x})}{\widetilde{h}^2(t, \widetilde{x})} s_t, \ 
\end{equation*} 
where we use the notations that
\begin{equation*}
l_t= n^{-1} (\nabla_t \bm{\phi}(t) \otimes \bm{\varphi}(\widetilde{x}))^\top \Pi^{-1}, \ \bm{\phi}(t)=(\phi_1(t), \cdots, \phi_c(t))^\top, \ \bm{\varphi}(\widetilde{x})=(\varphi_1(\widetilde{x}), \cdots, \varphi_b(\widetilde{x}))^\top,
\end{equation*}
and 
\begin{equation*}
s_t=(\nabla_t \bm{\phi}(t) \otimes \bm{\varphi}(\widetilde{x})) \Pi^{-1} \Omega \Pi^{-1} \bm{b}+ \bm{b}^\top \Pi^{-1} \Omega \Pi^{-1} (\nabla_t \bm{\phi}(t) \otimes \bm{\varphi}(\widetilde{x})).
\end{equation*}
Similarly, we can calculate $\frac{\partial T}{\partial x}.$
For notational simplicity, we set
\begin{equation*}
\mathbf{b}_1:=\nabla_t \bm{\phi}(t) \otimes \bm{\varphi}(\widetilde{x}), \mathbf{b}_2:= \bm{\phi}(t) \otimes \nabla_{\widetilde{x}} \bm{\varphi}(\widetilde{x}).
\end{equation*}
With the above notations, we have that 
\begin{equation*}
\det(A A^\top )=\mathcal{E}_1\mathcal{E}_3-\mathcal{E}_2^2,
\end{equation*}
where $\mathcal{E}_k, k=1,2,3,$ are defined as 
\begin{equation*}
\mathcal{E}_1=\frac{\partial T}{\partial t} \frac{\partial T}{\partial t}^\top, \ \mathcal{E}_3= \frac{\partial T}{\partial \widetilde{x}} \frac{\partial T}{\partial \widetilde{x}}^\top , \ \mathcal{E}_2= \frac{\partial T}{\partial \widetilde{x}} \frac{\partial T}{\partial t}^\top.
\end{equation*}
It is easy to see that
\begin{equation*}
\widetilde{h}(t, \widetilde{x}) \asymp \|  \widetilde{\bm{b}}\|_2^2, \ n \widetilde{l}(t, \widetilde{x}) \asymp \|  \widetilde{\bm{b}}\|_2^2, \ s_t^2 \asymp  \| \mathbf{b}_1 \|_2^2.  
\end{equation*}
Moreover, by Assumptions \ref{assum_physical} and \ref{assum_updc}, using (\ref{eq_consistencyconvergency}), we have that for all $1 \leq k \leq p,$
\begin{equation*}
\lambda_{k}(\Pi^{-1}), \ \lambda_{k}(\Omega)  \asymp 1. 
\end{equation*}
With the above estimates, for $\mathcal{E}_1,$ by an elementary but tedious calculation,  we can see that   
\begin{equation}\label{eq_boundcontrol}
\frac{\|\mathbf{b}_1 \|_2^2}{\| \widetilde{\bm{b}} \|_2^2}   \lambda_{\min}(\mathbf{T}) \leq  \| \mathcal{E}_1\| \leq   \lambda_{\max}(\mathbf{T}) \frac{\|\mathbf{b}_1 \|_2^2}{\| \widetilde{\bm{b}} \|_2^2}, \ \mathbf{T}:= \operatorname{Cov}(W^\top \bm{\epsilon}). 
\end{equation}
By (\ref{eq_originaldefinition}), Lemma \ref{lem_covarianceofz} and Assumption \ref{assum_updc}, we find that 
\begin{equation*}
 \lambda_{\min}(\mathbf{T}), \lambda_{\max}(\mathbf{T}) \asymp p.
\end{equation*}
Consequently, we have that 
\begin{equation*}
 \mathcal{E}_1 \asymp p \frac{\|\mathbf{b}_1 \|_2^2}{\| \widetilde{\bm{b}} \|_2^2}. 
\end{equation*}
Similarly, we can show 
\begin{equation*}
 \mathcal{E}_3 \asymp p \frac{\|\mathbf{b}_2 \|_2^2}{\| \widetilde{\bm{b}} \|_2^2}.
\end{equation*}
Therefore, under the assumption of (\ref{eq_assumptionbasisderivative}) and using the choice of (\ref{eq_cdchoice}), we have that for some constant $C>0$
\begin{equation*}
\kappa_0 \leq \int \int \sqrt{\mathcal{E}_1 \mathcal{E}_3}  \dd t\dd \widetilde{x} \leq Cp(n^{\alpha_1}+n^{\alpha_2}).
\end{equation*}
To provide a lower bound, since $AA^\top$ is symmetric and positive semi-definite, by (\ref{eq_kappodefinition}), we have that
\begin{equation*}
\kappa_0 \geq \int \int \lambda_{\min}(AA^\top) \dd t \dd \widetilde{x}. 
\end{equation*}
Together with Lemma \ref{lem_circle}, we have  that
\begin{equation*}
\kappa_0 \geq  \int \int \sqrt{\min\left\{(|\mathcal{E}_2|-\mathcal{E}_3)^2, \ (|\mathcal{E}_2|-\mathcal{E}_1)^2 \right\}} \dd t \dd \widetilde{x}. 
\end{equation*}  
By a discussion similar to (\ref{eq_boundcontrol}), using the definitions of $\mathcal{E}_k, k=1,2,3,$ we see that for some constant $C_1>0,$
\begin{equation*}
\kappa_0 \geq C_1 p(n^{\alpha_1}+n^{\alpha_2}). 
\end{equation*}
 This yields that
\begin{equation*}
 \kappa_0 \asymp  p n^{\alpha_1}+p n^{\alpha_2}. 
\end{equation*}  
We point out that $\zeta_0$ can be analyzed similarly using the second equation on Page 1335 of \cite{MR1311978}. Together with (\ref{eq_expansionformula}), we can complete the proof of (\ref{eq_quantitiesbound}).
\end{proof}

Then we prove Theorem \ref{thm_poweranalysis}. 
\begin{proof}[\bf Proof of Theorem \ref{thm_poweranalysis}]
We again focus our discussion on the case $r=1$ and omit the subscript $j.$ We start our proof of (1). Under the assumption of (\ref{eq_05bound}), we find that Theorem \ref{thm_gaussianapproximationcase} holds for $\mathcal{K}_2(\bm{z}, \bm{w}).$ Together with (\ref{eq_t2expansion}), the assumption of (\ref{eq_assumptionerrorreduce}) and Theorem \ref{thm_approximation}, utilizing the definition of mapping basis functions in (\ref{eq_defnmappedbasis}), we find that  it suffices to investigate the asymptotic distribution of $\bm{w}^\top \mathsf{W} \bm{w}.$  Denote the eigenvalues of $\Omega^{1/2} \mathsf{W} \Omega^{1/2}$ in the decreasing order as $\{\mu_i\}.$ Recall the definition of $\mathsf{W}$ in (\ref{eq_defnw}).  Note that
\begin{equation*}
\lambda_{\min}(\Omega) \lambda^2_{\min}(\Pi^{-1}) \lambda_{\min}(\mathsf{B}) \leq \mu_i \leq \lambda_{\max}(\Omega) \lambda^2_{\max}(\Pi^{-1}) \lambda_{\max}(\mathsf{B}), \ 1 \leq i \leq \operatorname{rank}(\Omega^{1/2} \mathsf{W} \Omega^{1/2}).
\end{equation*}
 Together  with Assumption \ref{assum_updc} and (\ref{eq_beigenvaluebound}), we conclude that 
\begin{equation*}
\mu_i \asymp 1, \ \operatorname{rank}(\Omega^{1/2} \mathsf{W} \Omega^{1/2}) \asymp p. 
\end{equation*} 
Consequently, we have that
\begin{equation*}
\frac{\mu_1}{\mathfrak{m}_2} \rightarrow 0, \ n \rightarrow \infty.  
\end{equation*}
By  (\ref{eq_t1ccc}) and Lindeberg’s central limit theorem (see the proof of \cite[Theorem 4.2]{DZ} or the discussion above \cite[Theorem 4]{XZW}), we obtain that 
\begin{equation*}
\frac{\bm{w}^\top \mathsf{W} \bm{w}-\mathfrak{m}_1}{\mathfrak{m}_2} \simeq \mathcal{N}(0,2). 
\end{equation*}  
This completes the proof of (1). 

For (2), denote
\begin{equation*}
\mathsf{T}_{2,a}:=\int_{[0,1]} \int_{\mathbb{R}} (\widehat{m}_{c,d}(t,x)-m(t,x))^2 \dd t \dd x.
\end{equation*} 
According to part (1), under $\mathbf{H}_a$ in (\ref{eq_alternative}), we have that
\begin{equation*}
\frac{n \mathsf{T}_{2,a}-\mathfrak{m}_1}{\mathfrak{m}_2} \simeq \mathcal{N}(0,2). 
\end{equation*}
By an elementary computation, we have that
\begin{equation*}
n\mathsf{T}_{2,a}=n \mathsf{T}_2-n \int_{[0,1]} \int_{\mathbb{R}}  v_n^2(t,x) \dd t \dd x-2n \int_{[0,1]} \int_{\mathbb{R}} (\widehat{m}_{c,d}(t,x)-m(t,x)) v_n(t,x) \dd t \dd x.
\end{equation*}
By (\ref{eq_betabound}), (\ref{eq_fundementalexpression}) and Cauchy-Schwarz inequality, we see that
\begin{equation*}
\left\| \int_{[0,1]} \int_{\mathbb{R}} (\widehat{m}_{c,d}(t,x)-m(t,x)) v_n(t,x) \dd t \dd x \right\|=\oo\left(  \int_{[0,1]} \int_{\mathbb{R}}  v_n^2(t,x) \dd t \dd x  \right).
\end{equation*}
This concludes the proof of the first statement of part (2). The second statement of part (2) follows from $\mathfrak{m}_2 \asymp \sqrt{p}$ and the assumption (\ref{eq_alternativeassumption}).  This completes our proof. 
\end{proof}

Finally, we proceed to prove Theorem \ref{thm_gaussianapproximationcase}. Its proof relies on the device of $m$-dependent approximation, the technique of suitable truncation and Lemma \ref{lem_mvnapp} which states a Gaussian approximation result on the convex set.  We prepare  some notations in the beginning. For $\bm{y}_i, i \geq 1,$ in (\ref{eq_decompositionkronecker}) and any nonnegative integer $m \geq 0,$ we define the so-called $m$-approximation of $\bm{y}_i$ as
\begin{equation*}
\bm{y}_i^\Mt=\bm{x}_i^\Mt \otimes \ab(t_i),
\end{equation*}
where $\bm{x}_i^\Mt$ is denoted as  
\begin{equation}\label{eq_dependencesequence}
\bm{x}_i^\Mt=\mathbb{E}(\bm{x}_i|\sigma(\eta_{i-m}, \cdots, \eta_i)),  \ m \geq 0,
\end{equation} 
where $\sigma(\eta_{i-m}, \cdots, \eta_i)$ is the natural sigma-algebra generated by the sequence of random variables.  
Corresponding to (\ref{eq_bmzgreatform}), we denote
\begin{equation*}
\bm{z}^\Mt=\frac{1}{\sqrt{n}}\sum_{i=2}^n \bm{y}_i^\Mt. 
\end{equation*}
Next, for a given truncation level $\hd>0,$ we define the truncated version of $\bm{y}_i^\Mt$ and $\bm{z}^\Mt$ following 
\begin{equation}\label{eq_truncation}
\bar{\bm{x}}_i^\Mt=\bm{x}_i^\Mt \mathbf{1}(\bm{x}_i^\Mt \leq \hd), \ \bar{\bm{y}}^\Mt_i=\bar{\bm{x}}^\Mt_i \otimes \ab(t_i), \ \bar{\bm{z}}^\Mt=\frac{1}{\sqrt{n}}\sum_{i=2}^n \bar{\bm{y}}^\Mt_i, 
\end{equation}
where the operation of truncation is applied entrywisely. Moreover, define $\{\bar{\bm{g}}^{\Mt}_i\}$ as the sequence of Gaussian random vectors which preserve the covariance structure of $\{\bar{\bm{x}}_i^\Mt\}$ whose Gaussian part is the same as $\bm{w}$ as in (\ref{eq_wgaussiandefinition}) and let
\begin{equation*}
\bar{\wb}^\Mt=\frac{1}{\sqrt{n}}\sum_{i=2}^n \bar{\gb}_i^\Mt.
\end{equation*}


%
%
%

\begin{proof}[\bf Proof of Theorem \ref{thm_gaussianapproximationcase}] Due to similarity, we focus our proof on the affine form and only briefly discuss the quadratic form in the end. The starting point is the following triangle inequality 
\begin{align}\label{eq_decomposition}
\mathcal{K}_1(\zb, \wb)& =\sup_{x \in \mathbb{R}} \left| \mathbb{P}(\bm{z}^\top \bm{l} \leq x)-\mathbb{P}(\bm{w}^\top \bm{l} \leq x) \right| \nonumber \\
& \leq \mathcal{K}_1(\zb, \zb^\Mt)+\mathcal{K}_1(\zb^\Mt, \bar{\zb}^\Mt)+\mathcal{K}_1(\bar{\zb}^\Mt, \bar{\wb}^\Mt)+\mathcal{K}_1(\bar{\wb}^\Mt, \wb).
\end{align}
It suffices to control every term of the right-hand side of (\ref{eq_decomposition}). 
In what follows, we provide the detailed  arguments for the controls following a five-step strategy.  

\vspace{3pt}
\noindent{\bf Step one. \underline{Show the closeness of $\operatorname{Cov}(\bm{z})$ and $\operatorname{Cov}(\bar{\zb}^\Mt).$}} Let $\zb=(z_1, \cdots, z_p)$, $\bar{\zb}^\Mt=(\bar{z}_1, \cdots, \bar{z}_p)$ and $\zb^\Mt=(z_1^\Mt, \cdots, z_p^\Mt).$ Note that
\begin{equation}\label{eq_tridecomposition}
\left\|\operatorname{Cov}(\bm{z})-\operatorname{Cov}(\bar{\zb}^\Mt) \right\|_{\op} \leq \left\| \operatorname{Cov}(\zb)-\operatorname{Cov}(\zb^\Mt) \right\|_{\op}+ \left\| \operatorname{Cov}(\bar{\zb}^\Mt)-\operatorname{Cov}(\zb^\Mt) \right\|_{\op},
\end{equation}
where we recall again that $\|\cdot \|_{\op}$ is the operator norm of the given positive definite matrix. 

First, we control the first term of the right-hand side of (\ref{eq_tridecomposition}). Observe that for all $1 \leq i \leq p,$
\begin{equation}\label{eq_variancedecomposition}
\operatorname{Var}(z_i)-\operatorname{Var}(z^\Mt_i)=\mathbb{E}(z_i)^2-\mathbb{E}(z^\Mt_i)^2+\mathbb{E}(z_i-z^\Mt_i) \mathbb{E}(z_i+z^\Mt_i). 
\end{equation} 
Without loss of generality, we focus on the case $i=1$ and still keep the subscript $i$ without causing any further confusion. By (\ref{eq_z1form}), Lemmas \ref{lem_locallystationaryform} and  \ref{lem_mdependent}, 
we readily obtain that for some constant $C>0,$
\begin{equation*}
\mathbb{E}(|z_i-z_i^\Mt|^q)^{2/q} \leq C \xi_c^2 \Theta^2_{m,q} \leq C \xi_c^2 m^{-2\tau+2},  
\end{equation*} 
where $\Theta_{m,q}=\sum_{k=m}^{\infty} \delta_x(k,q)$ (recall (\ref{eq_deltaxdefinition})). Consequently, as $q>2,$ by (1) of Lemma \ref{lem_collectionprobineq}, we readily obtain that for some constant $C_1>0$
\begin{equation}\label{eq_control}
\mathbb{E}|z_i-z^\Mt_i| \leq C_1 \xi_c \Theta_{m,q}, \  \mathbb{E}|z_i-z^\Mt_i|^2 \leq C_1 \xi_c^2 \Theta^2_{m,q}.
\end{equation} 
Note that (\ref{eq_variancedecomposition}) implies that 
\begin{equation}\label{eq_decompositionv}
\left| \operatorname{Var}(z_i)-\operatorname{Var}(z_i^\Mt)\right| \leq \sqrt{\mathbb{E}|z_i-z_i^\Mt|^2 \mathbb{E}|z_i+z_i^\Mt|^2}+\mathbb{E}|z_i-z_i^\Mt| \mathbb{E}|z_i+z_i^\Mt|,
\end{equation}
where we used (2) of Lemma \ref{lem_collectionprobineq}. Moreover, by Lemmas \ref{lem_locallystationaryform} and \ref{lem_concentration}, using the definition (\ref{eq_z1form}), we find that for some constant $C_2>0$
\begin{equation}\label{eq_bound}
\mathbb{E}|z_i|< C_2 \xi_c, \ \mathbb{E}|z_i|^2< C_2 \xi_c^2.
\end{equation} 
Together with (\ref{eq_control}), we see that
\begin{equation*}
\mathbb{E}|z_i+z_i^\Mt| \leq C_2 \xi_c+ C_1 \xi_c \Theta_{m,q}.
\end{equation*}
Furthermore, using (\ref{eq_bound}), we see that
\begin{align*}
 \mathbb{E}|z_i+z_i^\Mt|^2 & \leq 2 \mathbb{E}|z_i|^2+2 \mathbb{E}|z_i^\Mt|^2 \\
 & \leq 2C_2 \xi_c^2+2\mathbb{E}|z_i|^2+2\mathbb{E}|z_i-z_i^\Mt| \mathbb{E} |z_i+z_i^\Mt| \\
 & \leq 4 C_2 \xi_c^2+2 C_1 \xi_c\Theta_{m,q}(C_2\xi_c+C_1 \Theta_{m,q}).
\end{align*}
Together (\ref{eq_control}) and (\ref{eq_decompositionv}), we conclude that for some constant $C>0,$
\begin{equation*}
\left| \operatorname{Var}(z_i)-\operatorname{Var}(z_i^\Mt)\right| \leq C \xi^2_c \Theta_{m,q}. 
\end{equation*}
Similarly, we can show that for all $1 \leq i,j \leq p$
\begin{equation*}
\left| \operatorname{Cov}(z_i, z_j)-\operatorname{Cov}(z_i^\Mt, z_j^\Mt) \right| \leq C \xi^2_c \Theta_{m,q}. 
\end{equation*}
Consequently, by Lemma \ref{lem_circle}, we obtain that
\begin{equation}\label{eq_partonecontrol}
\left\| \operatorname{Cov}(\zb)-\operatorname{Cov}(\zb^\Mt)  \right\|_{\op} \leq C p \xi^2_c \Theta_{m,q} \leq C_1 p \xi^2_c m^{-\tau+1}, 
\end{equation}
where in the second inequality we used (\ref{eq_transferbound}). 

Second, we control the second term on the right-hand side of (\ref{eq_tridecomposition}). Recall (\ref{eq_truncation}). We point out that $\bar{z}^\Mt_i \neq z_i^\Mt \mathbf{1}(z_i \leq \hd)$ in general. For notational simplicity, we denote $\bm{x}_i^\Mt=(x_{i1}^\Mt, \cdots, x_{i,d}^\Mt).$ Using the construction (\ref{eq_dependencesequence}), with an argument similar to Lemma \ref{lem_locallystationaryform}, we can show that $\{\bm{x}_i^\Mt\}$ is a locally stationary time series whose physical dependence measure also satisfies $\delta(j,q) \leq Cj^{-\tau},$ for some constant $C>0$. Analogous to the discussion for the first term, we focus on the case $i=1$ and keep the subscript $i.$ Observe that
\begin{equation}\label{eq_firstorder}
|\mathbb{E}z_i^\Mt-\mathbb{E}\bar{z}_i^\Mt|=\left|\frac{1}{\sqrt{n}} \sum_{k=2}^n \phi_1(t_k) \mathbb{E}(x^\Mt_{k,i} \mathbf{1}(|x_{k,i}^\Mt|>\hd)) \right|.
\end{equation} 
Moreover, using Chebyshev's inequality (c.f. (3) of Lemma \ref{lem_collectionprobineq}), a discussion similar to (\ref{eq_bound}),  and the fact that 
\begin{equation*}
\mathbf{1}(|x_{k,i}^\Mt|>\hd) \leq \frac{|x_{k,i}^\Mt|^{q-1}}{\hd^{q-1}},
\end{equation*}
we find that for some constant $C>0,$
\begin{equation*}
|\mathbb{E}(x^\Mt_{k,i} \mathbf{1}(|x_{k,i}^\Mt|>\hd))| \leq C \xi_c \hd^{-(q-1)}.  
\end{equation*}
Consequently, we see that 
\begin{equation}\label{eq_expectationcontrolone}
|\mathbb{E}z_i^\Mt-\mathbb{E}\bar{z}_i^\Mt| \leq C \sqrt{n} \xi_c^2  \hd^{-(q-1)}.
\end{equation}
Similarly, we can show that 
\begin{equation}\label{eq_secondorder}
\left| \mathbb{E}(z_1^\Mt)^2-\mathbb{E} (\bar{z}_1^\Mt)^2 \right| \leq C n \xi_c^4  \hd^{-(q-2)}.
\end{equation}
By (\ref{eq_firstorder}), (\ref{eq_secondorder}) and the fact $|\bar{x}_{k,i}^\Mt| \leq \hd,$ using a decomposition similar to (\ref{eq_variancedecomposition}), we conclude that for some
constant $C>0,$
\begin{equation*}
\left| \operatorname{Var}(\bar{z}^\Mt_i)-\operatorname{Var}(z_i^\Mt)\right| \leq C n \xi_c^4 \hd^{-(q-2)}. 
\end{equation*}
Similarly, we can show that for all $1 \leq i,j \leq p$
\begin{equation*}
\left| \operatorname{Cov}(\bar{z}^\Mt_i, \bar{z}^\Mt_j)-\operatorname{Cov}(z_i^\Mt, z_j^\Mt) \right| \leq  C n \xi_c^4 \hd^{-(q-2)}. 
\end{equation*}
Consequently, by Lemma \ref{lem_circle}, we obtain that
\begin{equation}\label{eq_parttwocontrol}
\left\| \operatorname{Cov}(\bar{\zb}^\Mt)-\operatorname{Cov}(\zb^\Mt)  \right\|_{\op} \leq C pn \xi_c^4 \hd^{-(q-2)}. 
\end{equation}

In summary, by (\ref{eq_tridecomposition}), (\ref{eq_partonecontrol}) and (\ref{eq_parttwocontrol}), we find that for some constant $C>0$
\begin{equation}\label{eq_covarianceclose}
\left\|\operatorname{Cov}(\bm{z})-\operatorname{Cov}(\bar{\zb}^\Mt) \right\|_{\op}  \leq C \left( p \xi^2_c m^{-\tau+1}+ pn \xi_c^4 \hd^{-(q-2)} \right). 
\end{equation}
This completes Step one. Furthermore, since the assumption (\ref{eq_mhchoice}) ensures that the right-hand side of (\ref{eq_covarianceclose}) is of order $\mathrm{o}(1),$ we can conclude that the covariance matrices are close. 

\vspace{3pt}
\noindent {\bf Step two. \underline{Control the third term  $\mathcal{K}_1(\bar{\zb}^\Mt, \bar{\wb}^\Mt).$}} In this step, we apply Lemma \ref{lem_mvnapp} to control the associated term $\bar{\bm{z}}^{\mathtt{M}}$. For $x \in \mathbb{R},$ denote 
\begin{equation}\label{eq_defnset}
\mathsf{A}_x:=\left\{ \bm{q} \in \mathbb{R}^p: \bm{q}^\top \bm{l} \leq x   \right\}. 
\end{equation}
It is easy to see that $\mathsf{A}_x$ is a convex set in $\mathbb{R}^p$. Denote $\mathcal{A}$ as the collection of all the convex sets in $\mathbb{R}^p.$ Then we have that 
\begin{equation*}
\mathcal{K}_1(\bar{\zb}^\Mt, \bar{\wb}^\Mt)= \sup_{\mathsf{A}_x \in \mathcal{A}} \left| \mathbb{P}(\bar{\zb}^\Mt \in \mathsf{A}_x)-\mathbb{P}(\bar{\wb}^\Mt \in \mathsf{A}_x ) \right|.
\end{equation*} 
Therefore, it suffices to bound the right-hand side of the above equation using Lemma \ref{lem_mvnapp}. 
We verify conditions of Lemma \ref{lem_mvnapp}. Denote 
\begin{equation*}
N_i=N_{ij}=N_{ijk}=\{i-m, i-m+1, \cdots, i\}.
\end{equation*}
By the constructions (\ref{eq_dependencesequence}) and (\ref{eq_truncation}), using the property of conditional expectation (c.f. \cite[Example 4.1.7]{probbook}), it is easy to see that the conditions of Lemma \ref{lem_mvnapp} are satisfied such that 
\begin{equation*}
n_1=n_2=n_3=m, \ \beta=\sqrt{p} \mathrm{h}. 
\end{equation*}
By Lemma \ref{lem_mvnapp}, we have that for some constant $C>0$
 \begin{equation*}
\mathcal{K}_1(\bar{\zb}^\Mt, \bar{\wb}^\Mt)  \leq C p^{7/4} n^{-1/2} \| \Sigma^{-1/2} \|^3 \mathrm{h}^3m\left(m+\frac{m}{p}\right),
\end{equation*}  
where $\Sigma=\operatorname{Cov}(\bar{\zb}^\Mt).$  Together with Assumption \ref{assum_updc}, Lemma \ref{lem_covarianceofz}, (\ref{eq_covarianceclose}) and the assumption of (\ref{eq_mhchoice}),  we conclude that for some constant $C>0,$ 
 \begin{equation}\label{eq_steptworesults}
\mathcal{K}_1(\bar{\zb}^\Mt, \bar{\wb}^\Mt)  \leq C p^{7/4} n^{-1/2}  \mathrm{h}^3 m^2.  
\end{equation}
Moreover, since the assumption of  (\ref{eq_mhchoice}) ensures that $p^{7/4} n^{-1/2}  \mathrm{h}^3 m^2=\mathrm{o}(1),$ we conclude that $\bar{\zb}^\Mt$ is asymptotically Gaussian. This completes Step two. 

\vspace{3pt}
\noindent {\bf Step three. \underline{ Control the fourth term  $\mathcal{K}_1(\bar{\wb}^\Mt, \wb).$}}  Decompose that 
\begin{equation}\label{eq_probabilitydecomposition}
\mathcal{K}_1(\bar{\wb}^\Mt, \bm{w})=\sup_x \left| \mathbb{P}(\wb^\top \bm{l} \leq x )-\mathbb{P} (\wb^\top \bm{l} \leq x+\mathtt{D}(\bar{\wb}^\Mt, \wb)) \right|,
\end{equation}
where $\mathtt{D}(\bar{\wb}^\Mt, \wb))$ is defined as 
\begin{equation}\label{eq_differencedefinition}
\mathtt{D}(\bar{\wb}^\Mt, \wb))=(\wb-\bar{\wb}^\Mt)^\top \bm{l}. 
\end{equation}
Let $\Sigma_0=\operatorname{Cov}(\bm{z})$ and $\Sigma=\operatorname{Cov}(\bar{\bm{z}}^\Mt).$ Moreover, let $\bm{f}$ be some $p$-dimensional standard Gaussian random vector.  By construction, we can further write 
\begin{equation*}
\mathtt{D}(\bar{\wb}^\Mt, \wb)=((\Sigma_0^{1/2}-\Sigma^{1/2})\bm{f})^\top \bm{l}. 
\end{equation*}
Using Cauchy-Schwarz inequality, we find that for some constant $C>0$
\begin{equation*}
\left\| \mathtt{D}(\bar{\wb}^\Mt, \wb) \right\| \leq C \| \Sigma_0-\Sigma \|_{\op} \left(\mathbb{E} \| \bm{f} \|^2_2  \right)^{1/2} \| \bm{l} \|_2,
\end{equation*}
where $\| \bm{l} \|_2$ is the $L_2$ norm of the vector $\bm{l}.$ By Bernstein’s concentration inequality (c.f. (4) of Lemma \ref{lem_collectionprobineq}), we find that  for some constant $C>0,$
\begin{equation}\label{oldb1}
(\mathbb{E}\| \bm{f} \|_2^2)^{1/2} \leq C \sqrt{p}. 
\end{equation}
Moreover, by the definition (\ref{eq_defnxic}) and Assumption \ref{assum_updc}, we find that for some constant $C>0$
\begin{equation}\label{oldb2}
\| \bm{l} \|_2 \leq C\zeta. 
\end{equation}
Together with (\ref{eq_covarianceclose}), we conclude that 
\begin{equation}\label{eq_bounddbarww}
\|\mathtt{D}(\bar{\wb}^\Mt, \wb)\| \leq C \sqrt{p} \zeta \left( p \xi_c^2 m^{-\tau+1}+ pn \xi_c^4 \hd^{-(q-2)} \right).
\end{equation}
As $\bm{w}$ is a Gaussian random vector, we have 
\begin{equation*}
\bm{w}^\top \bm{l} \sim \mathcal{N}(0, \bm{l}^\top \Sigma_0 \bm{l}).
\end{equation*}
By Assumption \ref{assum_updc}, Lemmas \ref{lem_covarianceofz} and \ref{lem_basisl2norm}, it is easy to see that for the considered basis functions, there exist some constants $C_1, C_2>0$ such that
\begin{equation}\label{eq_controlbounderror}
C_1 \leq \bm{l}^\top \Sigma_0 \bm{l} \leq C_2 \zeta^2.
\end{equation} 
Without loss of generality, we assume that $\mathtt{D}(\bm{w}, \bar{\bm{w}}^\Mt)$ is positive so that 
\begin{equation}\label{eq_normalformcontrol}
\mathbb{P}(\wb^\top \bm{l} \leq x )-\mathbb{P} (\wb^\top \bm{l} \leq x+\mathtt{D}(\bar{\wb}^\Mt, \wb))=\frac{1}{\sqrt{2 \pi}}\int_{x'}^{x'+\mathtt{D}'(\bm{w}, \bar{\bm{w}}^\Mt)} e^{-\frac{z^2}{2}} \dd z,
\end{equation}
where $x'=x/\sqrt{\bm{l}^\top \Sigma_0 \bm{l}}, \  \mathtt{D}'(\bm{w}, \bar{\bm{w}}^\Mt)=\mathtt{D}(\bm{w}, \bar{\bm{w}}^\Mt)/\sqrt{\bm{l}^\top \Sigma_0 \bm{l}}.$ Based on (\ref{eq_normalformcontrol}), it is easy to see that for some constants $C, C'>0$
\begin{equation}\label{eq_decompositionform}
\mathbb{P}(\wb^\top \bm{l} \leq x )-\mathbb{P} (\wb^\top \bm{l} \leq x+\mathtt{D}(\bar{\wb}^\Mt, \wb)) \leq C \frac{\mathtt{D}(\bm{w}, \bar{\bm{w}}^\Mt)}{\sqrt{\bm{l}^\top \Sigma_0 \bm{l}}} \leq C' \mathtt{D}(\bm{w}, \bar{\bm{w}}^\Mt),
\end{equation}
where in the last inequality we used (\ref{eq_controlbounderror}). Together with (\ref{eq_bounddbarww}), we conclude that for some constant $C>0$ 
\begin{equation}\label{eq_stepthreeconclusion}
\mathcal{K}_1(\bar{\wb}^\Mt, \bm{w}) \leq C \left( \sqrt{p} \zeta  \left( p \xi_c^2 m^{-\tau+1}+ pn \xi_c^4 \hd^{-(q-2)} \right) \right). 
\end{equation}
This completes the proof of Step three.

\vspace{3pt}
\noindent {\bf Step four. \underline{Control the second term  $\mathcal{K}_1(\zb^\Mt, \bar{\zb}^\Mt).$}} In this step, we apply a discussion similar to Step three except that we utilize the asymptotic normality of $\bar{\zb}^\Mt$ which has been established in Step two under the assumption of (\ref{eq_mhchoice}). Analogously to (\ref{eq_probabilitydecomposition}), we decompose that 
\begin{equation}\label{eq_dd1}
\mathcal{K}_1(\bar{\zb}^\Mt, \bm{z}^\Mt)=\sup_x \left| \mathbb{P}((\bar{\zb}^\Mt)^\top \bm{l} \leq x )-\mathbb{P} ((\bar{\zb}^\Mt)^\top \bm{l} \leq x+\mathtt{D}(\bar{\zb}^\Mt, \zb^\Mt)) \right|,
\end{equation}
where $\mathtt{D}(\bar{\zb}^\Mt, \zb^\Mt)$ is defined similarly as in (\ref{eq_differencedefinition}). Recall (\ref{eq_firstorder}). By a discussion similar to (\ref{eq_expectationcontrolone}), we have that for some constant $C>0,$
\begin{equation*}
\mathbb{E} | \bar{z}_i^\Mt-z^\Mt_i | \leq C \sqrt{n} \xi_c^2 \hd^{-(q-1)}.
\end{equation*}
Similarly, we have that 
\begin{equation*}
(\mathbb{E} \| \bar{\zb}^\Mt-\zb^\Mt \|_2^2)^{1/2} \leq C \sqrt{p n} \xi_c^2 \hd^{-(q-1)}.
\end{equation*}
Together with (\ref{oldb1}) and (\ref{oldb2}), using Cauchy-Schwarz inequality, 
we have that for some constant $C>0$
\begin{equation}\label{eq_dd2}
\|\mathtt{D}(\bar{\zb}^\Mt, \zb^\Mt)\|  \leq C p\sqrt{n} \zeta \xi_c^2 \hd^{-(q-1)}. 
\end{equation}
By (\ref{eq_stepthreeconclusion}), (\ref{eq_covarianceclose}) and an argument similar to (\ref{eq_controlbounderror}) and (\ref{eq_normalformcontrol}), we obtain that for some constant $C>0$
\begin{equation}\label{eq_stepfourconclusion}
\mathcal{K}_1(\bar{\zb}^\Mt, \bm{z}^\Mt) \leq C p\sqrt{n} \zeta \xi_c^2 \hd^{-(q-1)}
\end{equation}
This completes the proof of Step four. 

\vspace{3pt}
\noindent{\bf Step five. \underline{Control the first term  $\mathcal{K}_1(\zb, \zb^\Mt).$}} In the last step, we apply a discussion similar to Step four utilizing the asymptotic normality of $\zb^\Mt$ as established in Step four under the assumption of (\ref{eq_mhchoice}). By (\ref{eq_control}), we have that for some constant $C>0$
\begin{equation*}
\mathbb{E} \| \zb-\zb^\Mt \|_2 \leq C \sqrt{p} \sum_{i=1}^p \Theta_{m,q} \leq C p^{3/2} \xi_c m^{-\tau+1},
\end{equation*}
where in the second inequality we again used (\ref{eq_transferbound}). By a discussion similar to (\ref{eq_dd1}) and (\ref{eq_dd2}), we conclude that for some constant $C>0$
\begin{equation}\label{eq_stepfiveconclusion}
\mathcal{K}_1(\zb^\Mt, \zb) \leq C \zeta p^2 \xi_c m^{-\tau+1}.
\end{equation}
This completes the proof of Step five. 

In summary, by (\ref{eq_steptworesults}), (\ref{eq_stepthreeconclusion}), (\ref{eq_stepfourconclusion}), (\ref{eq_stepfiveconclusion}) and (\ref{eq_decomposition}), we can conclude the proof of $\mathcal{K}_1. $


Finally, we briefly discuss how to handle the quadratic form $\mathcal{K}_2.$ We only focus on explaining the differences in the five-step strategies. More specifically, Step one can be applied without modification. For step two, the analogous set of (\ref{eq_defnset}) is
\begin{equation*}
\mathtt{A}_x=\{ \bm{q} \in \mathbb{R}^p: \bm{q}^\top \mathsf{W} \bm{q} \leq x\}. 
\end{equation*} 
Since $\mathsf{W}$ is positive semi-definite, $\mathtt{A}_x$ is a convex set. Then the rest of Step two also applies. For Step three, the main difference lies in (\ref{eq_decompositionform}). In the quadratic case, the distribution of $\bm{w}^\top \mathsf{W} \bm{w}$ can be written as a summation of Chi-square random variables and (\ref{eq_decompositionform}) should be controlled using Lemma \ref{lem_Gaussianquadratic}. In particular, we have that 
\begin{equation*}
\mathbb{P}(\wb^\top \mathsf{W} \wb \leq x )-\mathbb{P} (\wb^\top \mathsf{W} \bm{w} \leq x+\mathtt{D}(\bar{\wb}^\Mt, \wb)) \leq C \sqrt{\mathtt{D}(\bar{\wb}^\Mt, \wb)},
\end{equation*}
where $\mathtt{D}(\bar{\wb}^\Mt, \wb)$ is defined as
\begin{equation*}
\mathtt{D}(\bar{\wb}^\Mt, \wb):=-\wb^\top \mathsf{W} \wb+(\bar{\wb}^\Mt)^\top \mathsf{W} \bar{\wb}^\Mt,
\end{equation*}
which can be controlled similarly as in (\ref{eq_bounddbarww}). Similar modifications should be made for Steps four and five. This completes the proof.

\end{proof}
\subsection{Consistency of bootstrapping: proof of Theorems \ref{thm_boostrapping} and \ref{thm_boostrapping2}}

In this subsection, we prove Theorems \ref{thm_boostrapping} and \ref{thm_boostrapping2} following \cite{DZ2,MR3174655}. Without loss of generality, we focus on the case $r=1$ as in Sections \ref{sec_consistencyproof} and \ref{sec_gassuianapproximation}.  Set 
\begin{equation}\label{eq_defnitionphi}
\Phi=\frac{1}{\sqrt{n-m-1} \sqrt{m}} \sum_{i=2}^{n-m} \left[ \left( \sum_{j=i}^{i+m} \bm{x}_j \right) \otimes \ab(t_i) \right] R_i,
\end{equation}
where $\{R_i\}$ are the same random variables as in  
(\ref{eq_defnstatisticXi}). We point out that $\Phi$ is a population version of $\Xi$ in (\ref{eq_defnstatisticXi}). Moreover, denote 
\begin{equation}\label{eq_defnupsilon}
\Upsilon_{i,m}=\frac{1}{\sqrt{m}} \mathtt{X}_i \otimes \ab(t_i), \ \mathtt{X}_i=\sum_{j=i}^{i+m} \bm{x}_j.  
\end{equation}
Based on the above notations, we denote that
\begin{equation}\label{eq_bigupsilon}
\Upsilon=\frac{1}{(n-m-1)} \sum_{i=2}^{n-m} \Upsilon_{i,m} \Upsilon_{i,m}^\top.
\end{equation}
Note that $\Upsilon$ is the covariance matrix of $\Phi$ when conditioned on the data. 

The proof routine contains five steps. The first three steps concern $\Phi.$ In particular, step one (c.f. Lemma \ref{lem_stepone}) aims to establish the concentration of $\Upsilon_{i,m} \Upsilon_{i,m}^\top $ and $\Upsilon \Upsilon^\top;$ step two (c.f. Lemma \ref{lem_steptwo}) focuses on constructing a stationary time series whose covariance matrix can well approximate the concentration from step one; step three (c.f. Lemma \ref{lem_stepthree}) utilizes the stationary time series from step two and  shows that its covariance matrix is close to the integrated long-run covariance matrix $\Omega;$ step four is optional, it aims to replace the deterministic quantities with their consistent estimators if there is a need, for example $\mathsf{W}$ for Theorem \ref{thm_boostrapping2}; step five (c.f. Lemma \ref{lem_residualclosepreparation}) will conclude the proof by replacing the error $\{\epsilon_i\}$ with the residuals of the sieve estimators (c.f. (\ref{eq_defnresidual})) and handle $\Xi$ to conclude the proof. 
 
Steps one, two and three focus on analyzing $\Upsilon$ and the long-run covariance matrices and are general, i.e., irrelevant of the inference problems. Combining them we can show that $\Upsilon$ is close to $\Omega.$  Therefore, we separate them and prove them in Lemmas \ref{lem_stepone}--\ref{lem_stepthree} below. Steps four and five are more specific and we will provide the details when we prove Theorems \ref{thm_boostrapping} and \ref{thm_boostrapping2}.

\begin{lemma}\label{lem_stepone}
Suppose Assumptions \ref{assum_models}--\ref{assum_updc} hold and $m=\oo(n)$. Recall (\ref{eq_defnxic}). We have that  
\begin{equation*}
\sup_{2 \leq i \leq n-m} \left\| \Upsilon_{i,m} \Upsilon_{i,m}^\top -\mathbb{E} \left( \Upsilon_{i,m} \Upsilon_{i,m}^\top\right) \right\|_{\op}=\OO_{\mathbb{P}}\left( d \zeta^2 \sqrt{m} \right). 
\end{equation*}
Similarly, we have that 
\begin{equation}\label{eq_overallproof}
\sup_{2 \leq i \leq n-m} \left\| \Upsilon-\mathbb{E} \Upsilon \right\|_{\op}=\OO_{\mathbb{P}}\left( d \zeta^2 \sqrt{\frac{m}{n}} \right). 
\end{equation}
\end{lemma}
\begin{proof}
Using the basic property of Kronecker product, we find that 
\begin{equation*}
 \Upsilon_{i,m} \Upsilon_{i,m}^\top =\frac{1}{m} \left[ \mathtt{X}_i \mathtt{X}_i^\top \right] \otimes \left[ \ab(t_i) \ab(t_i)^\top  \right]. 
\end{equation*} 
Consequently, we have that 
\begin{equation}\label{eq_boundxx}
\sup_{2 \leq i \leq n-m} \left\| \Upsilon_{i,m} \Upsilon_{i,m}^\top -\mathbb{E} \left( \Upsilon_{i,m} \Upsilon_{i,m}^\top\right) \right\|_{\op} \leq \frac{\zeta^2}{m} \sup_{2 \leq i \leq n-m} \left\| \mathtt{X}_i \mathtt{X}_i^\top-\mathbb{E} \mathtt{X}_i \mathtt{X}_i^\top \right\|_{\op}, 
\end{equation}
where we use the property of the spectrum of Kronecker product and the fact that $\ab(t_i) \ab(t_i)^\top$ is a rank-one matrix. As before, we focus on the first entry of $\mathsf{L}=\mathtt{X}_i \mathtt{X}_i^\top.$ Note that 
\begin{equation*}
\mathsf{L}_{11}=\left(\sum_{j=i}^{i+m} \varphi_1(X_{j-1}) \epsilon_j \right)^2. 
\end{equation*}
By a discussion similar to Lemma \ref{lem_locallystationaryform}, $\{\mathsf{L}_{11}\}$ can be regarded as a locally stationary time series whose physical dependence measure $\delta(l, q)$ satisfies that 
\begin{equation*}
\delta(l,q) \leq C \sqrt{m} \left( \sum_{j=l-m}^l \delta_x(j,q) \right),
\end{equation*}
where $\delta_x(j,q)$ is defined in (\ref{eq_deltaxdefinition}). Combing (\ref{eq_transferbound}) and (1) of Lemma \ref{lem_concentration}, we readily obtain that
\begin{equation*}
\| \mathsf{L}_{11}-\mathbb{E} \mathsf{L}_{11} \|=\OO(m^{3/2}). 
\end{equation*}
By (\ref{eq_boundxx}), Lemma \ref{lem_circle} and (1) of Lemma \ref{lem_collectionprobineq}, we complete our proof of the first part. The second part can be proved analogously. 
\end{proof}

Recall $\bm{x}_i$ in (\ref{eq_locallystationaryform}). We denote
\begin{equation*}
\widetilde{\bm{x}}_{i,j}=\Ub(t_i, \mathcal{F}_j), \ i \leq j \leq i+m. 
\end{equation*}
Corresponding to (\ref{eq_defnupsilon}), we define 
\begin{equation*}
\widetilde{\Upsilon}_{i,m}=\frac{1}{\sqrt{m}} \widetilde{\mathtt{X}}_i \otimes \ab(t_i), \ \widetilde{\mathtt{X}}_i=\sum_{j=i}^{i+m} \widetilde{\bm{x}}_{i,j}. 
\end{equation*}

\begin{lemma}\label{lem_steptwo}
Suppose the assumptions of Lemma \ref{lem_stepone} hold. We have that 
\begin{equation*}
\sup_{2 \leq i \leq n-m} \left\| \mathbb{E} \left( \Upsilon_{i,m} \Upsilon_{i,m}^\top \right)-\mathbb{E} \left( \widetilde{\Upsilon}_{i,m} \widetilde{\Upsilon}_{i,m}^\top \right) \right\|_{\op}=\OO\left( d \zeta^2 \left( \frac{m}{n}\right)^{1-\frac{1}{\tau}} \right). 
\end{equation*}
\end{lemma}
\begin{proof}
By an argument similar to (\ref{eq_boundxx}), we find that
\begin{equation}\label{eq_xxbound2}
\sup_{2 \leq i \leq n-m}\left\| \mathbb{E} \left( \Upsilon_{i,m} \Upsilon_{i,m}^\top \right)-\mathbb{E} \left( \widetilde{\Upsilon}_{i,m} \widetilde{\Upsilon}_{i,m}^\top \right)   \right\|_{\op} \leq \frac{\zeta^2}{m} \sup_{2 \leq i \leq n-m} \left\| \mathbb{E} \left(\mathtt{X}_i \mathtt{X}_i^\top \right)-\mathbb{E} \left( \widetilde{\mathtt{X}}_i \widetilde{\mathtt{X}}_i^\top \right) \right\|_{\op}. 
\end{equation}
We again focus on the first entry of $\mathsf{M}:=\mathtt{X}_i \mathtt{X}_i^\top- \widetilde{\mathtt{X}}_i \widetilde{\mathtt{X}}_i^\top.$ Note that 
\begin{align}\label{eq_decompositionm11}
\mathsf{M}_{11}& =\left(\sum_{j=i}^{i+m} \varphi_1(X_{j-1}) \epsilon_j \right)^2-\left(\sum_{j=i}^{i+m} \varphi_1(\widetilde{X}_{j-1}) \widetilde{\epsilon}_j \right)^2 \\
& =\left[\sum_{j=i}^{i+m} \left(\varphi_1(X_{j-1}) \epsilon_j-\varphi_1(\widetilde{X}_{j-1}) \widetilde{\epsilon}_j \right)   \right] \left[\sum_{j=i}^{i+m} \left(\varphi_1(X_{j-1}) \epsilon_j+\varphi_1(\widetilde{X}_{j-1}) \widetilde{\epsilon}_j \right)   \right]:=\mathsf{P}_1\mathsf{P}_2. \nonumber 
\end{align}
By Lemma \ref{lem_locallystationaryform} and (1) of Lemma \ref{lem_concentration}, we see that 
\begin{equation*}
\| \mathsf{P}_2 \|=\OO(\sqrt{m}). 
\end{equation*}
For $\mathsf{P}_1,$ by Lemma \ref{lem_locallystationaryform}, we find that 
\begin{equation*}
\| \mathsf{P}_1 \|=\OO\left( \sum_{j=i}^{i+m} \varphi_1(X_{j-1})(\widetilde{\epsilon}_j-\epsilon_j) \right). 
\end{equation*}
Together with Lemma \ref{lem_locallystationaryform}, the stochastic continuity property (\ref{eq_slc}) and (1) of Lemma \ref{lem_concentration}, we have that
\begin{equation*}
\| \mathsf{P}_1 \|=\OO\left(\sqrt{m} \sum_{j=0}^{\infty} \min\{\frac{m}{n}, j^{-\tau}\} \right)=\OO\left( \sqrt{m} \left( \frac{m}{n} \right)^{1-1/\tau}\right).
\end{equation*} 
Combining the above arguments, we obtain that
\begin{equation*}
\| \mathsf{M}_{11} \|=\OO\left( m \left( \frac{m}{n} \right)^{1-1/\tau} \right).
\end{equation*}
Together with (\ref{eq_xxbound2}), Lemma \ref{lem_circle} and (1) of Lemma \ref{lem_collectionprobineq}, we conclude the proof.  
\end{proof}

Corresponding to (\ref{eq_bigupsilon}), denote 
\begin{equation*}
\widetilde{\Upsilon}:=\frac{1}{n-m-1} \sum_{i=2}^{n-m} \widetilde{\Upsilon}_{i,m} \widetilde{\Upsilon}_{i,m}^\top.
\end{equation*}
Recall (\ref{eq_defnpsim}). 
\begin{lemma}\label{lem_stepthree}
Suppose the assumptions of Lemma \ref{lem_stepone} hold. Then we have 
\begin{equation}\label{eq_firstproof}
\left\| \mathbb{E} \widetilde{\Upsilon}-\Omega  \right\|_{\op}=\OO\left(\frac{1}{(n-m-1)^2}+\frac{d \zeta^2}{m} \right).
\end{equation}
Consequently, we have  
\begin{equation}\label{eq_finalcovclose}
\left\| \Upsilon-\Omega \right\|_{\op}=\OO(\Psi(m)). 
\end{equation}
\end{lemma}
\begin{proof}
For (\ref{eq_firstproof}), first of all, using the definition of $\Omega(t)$ in (\ref{eq_longrunwitht}), by (2) of Lemma \ref{lem_concentration} and a discussion similar to (\ref{eq_boundxx}), we have that (also see Lemma 4 of \cite{MR3174655})
\begin{equation*}
\sup_{2 \leq i \leq n-m} \left\| \mathbb{E} \widetilde{\Upsilon}_{i,m} \widetilde{\Upsilon}_{i,m}^\top-\Omega(t_i) \right\|_{\op}=\OO\left( \frac{d \zeta^2}{m} \right).
\end{equation*}
Together with Lemma \ref{lem_intergralappoximation}, we can complete our proof of (\ref{eq_firstproof}). 

The proof of (\ref{eq_finalcovclose}) follows from (\ref{eq_firstproof}), Lemmas \ref{lem_stepone} and \ref{lem_steptwo}. 
\end{proof}

The following lemma indicates that the residual (\ref{eq_defnresidual}) is close to $\{\epsilon_i\}$ such that $\Xi$ is close to $\Phi.$ Recall $\{\widehat{\bm{x}}_i\}$ in (\ref{eq_xihatdefinition}). Denote 
\begin{equation*}
\widehat{\Upsilon}_{i,m}=\frac{1}{\sqrt{m}} \widehat{\mathtt{X}}_i \otimes \ab(t_i), \ \widehat{\mathtt{X}}_i=\sum_{j=i}^{i+m} \widehat{\bm{x}}_j.  
\end{equation*}
Accordingly, we can define $\widehat{\Upsilon}$ as in (\ref{eq_bigupsilon}) using $\widehat{\Upsilon}_{i,m}.$ Recall (\ref{eq_defnresidual}).
\begin{lemma}\label{lem_residualclosepreparation}
Suppose the assumptions of Lemma \ref{lem_stepone} and Theorem \ref{thm_consistency} hold. We have that 
\begin{equation*}
\sup_{2 \leq i \leq n-m}\| \Upsilon_{i,m} \Upsilon_{i,m}^\top-  \widehat{\Upsilon}_{i,m} \widehat{\Upsilon}_{i,m}^\top\|_{\op}=\OO_{\mathbb{P}}\left( d \zeta^2 \left[\xi_c \zeta \sqrt{\frac{p}{n}}+c^{-\mathsf{m}_1}+d^{-\mathsf{m}_2} \right] \right).
\end{equation*}
As a result, we have that 
\begin{equation*}
\left\| \Upsilon-\widehat{\Upsilon} \right\|_{\op}=\OO_{\mathbb{P}}\left( \frac{d \zeta^2}{\sqrt{n}} \left[\xi_c \zeta \sqrt{\frac{p}{n}}+c^{-\mathsf{m}_1}+d^{-\mathsf{m}_2} \right] \right).
\end{equation*}
\end{lemma}
\begin{proof}
By Theorem \ref{thm_consistency}, we find that
\begin{equation}\label{eq_residualbound} 
\sup_{2 \leq i \leq n-m}\| \epsilon_i-\widehat{\epsilon}_i\|=\OO\left( \xi_c \zeta \sqrt{\frac{p}{n}}+c^{-\mathsf{m}_1}+d^{-\mathsf{m}_2} \right).
\end{equation}
Using a discussion similar to (\ref{eq_xxbound2}) and (\ref{eq_decompositionm11}) such that $\mathsf{P}_2$ in  (\ref{eq_decompositionm11}) is controlled using (\ref{eq_residualbound}), we can conclude the proof of the first part of the results. The second part of the results follows from a discussion similar to (\ref{eq_overallproof}).  
\end{proof}

Armed with Lemmas \ref{lem_stepone}--\ref{lem_residualclosepreparation}, we proceed to finish the proof of Theorems \ref{thm_boostrapping} and \ref{thm_boostrapping2}. 

\begin{proof}[\bf Proof of Theorem \ref{thm_boostrapping}] 
As before, we focus on the case $r=1$ and omit the subscript $j.$ First, we prove the result holds for 
\begin{equation*}
\widetilde{\mathsf{T}}_1:=\Phi^\top  \widehat{\Pi}^{-1} \bm{b} . 
\end{equation*}
Conditional on the data, by the definition of $\Phi$ in (\ref{eq_defnitionphi}), $\Phi$ is a Gaussian random vector. We now write
\begin{equation}\label{eq_phigaussianrepression}
\Phi=\Lambda^{1/2} \mathbf{G},
\end{equation}
where $\mathbf{G} \sim \mathcal{N}(\mathbf{0}, \mathbf{I}_p)$ follows standard multivariate Gaussian distribution. In fact, conditional on the data, $\Lambda=\Upsilon.$ Consequently, we can write 
\begin{equation*}
\widetilde{\mathsf{T}}_1=\mathbf{G}^\top \mathsf{H}, \  \mathsf{H}:=\Lambda^{1/2} \widehat{\Pi}^{-1} \bm{b}.
\end{equation*}
Conditional on the data, according to (\ref{eq_consistencyconvergency}) and (\ref{eq_finalcovclose}), using the definition of $h(t,x)$ and triangle inequality,  we see that 
\begin{equation*}
\left\| \mathsf{H}^\top \mathsf{H}-h^2(t,x) \right\|=\OO\left( \zeta \left[ \Psi(m)+p\left( \frac{\xi^2_c}{\sqrt{n}}+\frac{\xi^2_c n^{\frac{2}{\tau+1}}}{n}\right)  \right] \right).
\end{equation*}
Consequently, under the assumption of (\ref{eq_boothstrappingextraassumption}), since $\mathbf{G}$ is a standard multivariate Gaussian vector, we have shown that the result holds for $\widetilde{\mathsf{T}}_1.$ 

Then we prove for $\widehat{\mathsf{T}}_1$ based on $\widetilde{\mathsf{T}}_1.$ Conditional on the data, by construction, analogous to (\ref{eq_phigaussianrepression}), we can write
\begin{equation}\label{eq_gaussionrepapprox}
\Xi=\widehat{\Lambda}^{1/2} \mathbf{G}. 
\end{equation} 
Accordingly, we can write
\begin{equation*}
\widehat{\mathsf{T}}_1=\mathbf{G}^\top \widehat{\mathsf{H}},  \ \widehat{\mathsf{H}}= \widehat{\Lambda}^{1/2} \widehat{\Pi}^{-1} \bm{b}. 
\end{equation*}
Note that condition on the data, we have that
\begin{equation}\label{eq_close}
\Lambda=\Upsilon, \ \widehat{\Lambda}=\widehat{\Upsilon}. 
\end{equation}
By (4) of Lemma \ref{lem_collectionprobineq}, we find that
\begin{equation}\label{eq_norm}
\mathbb{E} \| \mathbf{G} \|^2_2=\OO(p).
\end{equation}
By triangle inequality and Lemma \ref{lem_residualclosepreparation}, we see that 
\begin{equation*}
\| \widehat{\mathsf{T}}_1-\widetilde{\mathsf{T}}_1\|=\OO\left( \sqrt{p} \frac{d \zeta^3}{\sqrt{n}} \left[\xi_c \zeta \sqrt{\frac{p}{n}}+c^{-\mathsf{m}_1}+d^{-\mathsf{m}_2} \right] \right). 
\end{equation*}
This concludes our proof under the assumption of (\ref{eq_boothstrappingextraassumption}). 
\end{proof}

\begin{proof}[\bf Proof of Theorem \ref{thm_boostrapping2}] The proof is similar to that of Theorem \ref{thm_boostrapping} and we only list the key points. First, we prove that the result holds for 
\begin{equation*}
\widetilde{\mathsf{T}}_2:=\Phi^\top  \widehat{\mathsf{W}} \Phi.
\end{equation*} 
Using (\ref{eq_phigaussianrepression}), we can rewrite $\widetilde{\mathsf{T}}_2$ as 
\begin{equation*}
\widetilde{\mathsf{T}}_2=\mathbf{G}^\top \mathbf{M} \mathbf{G}, \ \mathbf{M}=\Lambda^{1/2} \widehat{\mathsf{W}} \Lambda^{1/2}.
\end{equation*}
Denote $r=\operatorname{rank}(\mathbf{M}).$ Recall (\ref{eq_defnhatsfw}). By Lemma \ref{lem_stepthree}, Assumption \ref{assum_updc} and (\ref{eq_boundone}), we find that
\begin{equation*}
\lambda_{\min}(\mathbf{M})=\OO(1).
\end{equation*}
Consequently, $r=p.$ Denote the eigenvalues of $\mathbf{M}$ as $\lambda_1 \geq \lambda_2 \geq \cdots \geq  \lambda_r>0.$ It is easy to see that $\lambda_i=\OO(1), \ 1 \leq i \leq r.$ Then by Lindeberg's central limit theorem, when conditional on the data, we have that 
\begin{equation}\label{eq_tt221}
\frac{\mathbf{G}^\top \mathbf{M} \mathbf{G}-\sum_{i=1}^r \lambda_i}{\left( \sum_{i=1}^r \lambda_i^2 \right)^{1/2}} \simeq \mathcal{N}(0,2). 
\end{equation}
 In view of (\ref{eq_limiting}), it suffices to compare $(\sum_{i=1}^r \lambda_i^k)^k$ with $\mathfrak{m}_k, k=1,2.$ Denote the eigenvalues of $\Omega^{1/2} \mathsf{W} \Omega^{1/2}$ as $\mu_1 \geq \mu_2 \geq \cdots \geq \mu_r>0.$  Denote the set $\mathcal{A} \equiv \mathcal{A}_n$ as 
 \begin{equation*}
 \mathcal{A} \equiv \mathcal{A}_n:=\left\{ \left| \sum_{i=1}^r(\lambda_i-\mu_i) \right| \leq \delta_{1n} \sqrt{p}, \ \left| \sum_{i=1}^r (\lambda_i^2-\mu_i^2)  \right| \leq \delta_{2n} \sqrt{p} \right\},
 \end{equation*}
 where $\delta_{kn}, k=1,2,$ are some sequences such that $\delta_{kn}=\oo(1).$ When restricted on the event $\mathcal{A},$ we have that
 \begin{align*}
 \frac{\mathbf{G}^\top\mathbf{M} \mathbf{G}-\mathfrak{m}_1}{\mathfrak{m}_2} & =\frac{\mathbf{G}^\top\mathbf{M} \mathbf{G}-\sum_{i=1}^r \lambda_r+\sum_{i=1}^r \lambda_i-\mathfrak{m}_1}{(\sum_{i=1}^r \lambda_i^2)^{1/2}} \left( \frac{(\sum_{i=1}^r \lambda_i^2)^{1/2}}{\mathfrak{m}_2} \right) \\
 &=\frac{\mathbf{G}^\top \mathbf{M} \mathbf{G}-\sum_{i=1}^r \lambda_i}{(\sum_{i=1}^r \lambda_i^2)^{1/2}}+\oo(1).
 \end{align*}
Together with (\ref{eq_tt221}), we have shown that the results hold for $\widetilde{\mathsf{T}}_2$ when  restricted to the event $\mathcal{A}.$ Recall (\ref{eq_defnw}) and (\ref{eq_defnhatsfw}). By (\ref{eq_consistencyconvergency}) and (\ref{eq_finalcovclose}), we have that
\begin{equation*}
\max_i | \lambda_i-\mu_i |=\OO_{\mathbb{P}}\left(\Psi(m)+p\left( \frac{\xi^2_c}{\sqrt{n}}+\frac{\xi^2_c n^{\frac{2}{\tau+1}}}{n}\right) \right).
\end{equation*}
Consequently, under the assumption of (\ref{eq_boostrappingparamterassumption}), we have that
\begin{equation*}
\mathbb{P}(\mathcal{A})=1-\oo(1). 
\end{equation*}
This shows that the result holds for $\widetilde{\mathsf{T}}_2.$ 

Second, we prove the result for $\widehat{\mathsf{T}}_2$ based on $\widetilde{\mathsf{T}}_2.$  According to (\ref{eq_gaussionrepapprox}), we can write
\begin{equation*}
\widehat{\mathsf{T}}_2=\mathbf{G}^\top \widehat{\mathbf{M}} \mathbf{G}, \ \widehat{\mathbf{M}}=\widehat{\Lambda}^{1/2} \widehat{\mathsf{W}} \widehat{\Lambda}^{1/2}. 
\end{equation*}
By (\ref{eq_close}), (\ref{eq_norm}) and  Lemma \ref{lem_residualclosepreparation}, we see that 
\begin{equation*}
\| \widehat{\mathsf{T}}_2-\widetilde{\mathsf{T}}_2\|=\OO\left( p \frac{d \zeta^2}{\sqrt{n}} \left[\xi_c \zeta \sqrt{\frac{p}{n}}+c^{-\mathsf{m}_1}+d^{-\mathsf{m}_2} \right] \right). 
\end{equation*}
This concludes our proof under the assumption of (\ref{eq_boostrappingparamterassumption}). 

\end{proof}

\section{Some auxiliary lemmas}\label{sec_auxililarylemma}
In this section, we collect some useful auxiliary lemmas. The first lemma, Lemma \ref{lem_mvnapp} will be used in the proof of Gaussian approximation, i.e., Theorem \ref{thm_gaussianapproximationcase}. 
\begin{lemma}[Multivariate Gaussian approximation]\label{lem_mvnapp}
Denote $[n]:=\{1,\cdots, n\}$ and $X_N:=\sum_{i \in N} X_i$ for some index set $N.$ Let $X_i \in \mathbb{R}^d$ and $W=\sum_{i=1}^n X_i.$ Moreover, we assume that 
\begin{equation*}
\mathbb{E} X_i=0, \ \operatorname{Cov}(W)=\Sigma.
\end{equation*}
For a standard $d$-dimensional Gaussian random vector $Z$, denote 
\begin{equation*}
\mathtt{d}_c(\mathcal{L}(W), \mathcal{L}(\Sigma^{1/2}Z)):=\sup_{A \in \mathcal{A}}\left| \mathbb{P}(W \in A)-\mathbb{P}(\Sigma^{1/2}Z \in A) \right|,
\end{equation*}
where $\mathcal{A}$ denotes the collection of all the convex sets in $\mathbb{R}^d.$  Suppose that $W$ can be decomposed as follows
\begin{itemize}
\item $\forall i \in [n], \exists i \in N_i \subset [n] $ such that $W-X_{N_i}$ is independent of $X_i;$ 
\item $\forall i \in [n], j \in N_i, \exists N_i \subset N_{ij} \subset [n]$ such that   $W-X_{N_{ij}}$ is independent of $\{X_i, X_j\};$
\item $\forall i \in [n], j \in  N_i, k \in N_{ij}, \exists N_{ij} \subset N_{ijk} \subset [n]$ such that $W-X_{N_{ijk}}$ is independent of $\{X_i,X_j, X_k\}.$
\end{itemize}
Suppose further that for each $i \in [n], j \in N_i$ and $k \in N_{ij},$
\begin{equation*}
\|X_i\|_2 \leq \beta, \ |N_i| \leq n_1, \ |N_{ij}| \leq n_2, \ |N_{ijk}| \leq n_3. 
\end{equation*}
Then there exists a universal constant $C>0$ such that 
\begin{equation*}
\mathtt{d}_c(\mathcal{L}(W), \mathcal{L}(\Sigma^{1/2}Z)) \leq C d^{1/4} n \| \Sigma^{-1/2} \|^3  \beta^3 n_1 \left( n_2+\frac{n_3}{d} \right). 
\end{equation*}
\end{lemma}
\begin{proof}
See Theorem 2.1 and Remark 2.2 of \cite{MR3571252}.
\end{proof}

In the following two lemmas, we consider that $x_i=G_i(\mathcal{F}_i),$ where $G_i(\cdot)$ is some measurable function and $\mathcal{F}_i=(\cdots, \eta_{i-1}, \eta_i)$ and $\eta_i$ are i.i.d. random variables. Suppose that $\mathbb{E} x_i=0$ and $\max_i \mathbb{E} |x_i|^q<\infty$ for some $q>1.$ For any integer $k>0,$ denote $\theta_{k,q}=\max_{1 \leq i \leq n}\| G_i(\mathcal{F}_i)-G_i(\mathcal{F}_{i,i-k}) \|_q,$ where $\mathcal{F}_{i,i-k}=(\cdots, \eta_{i-k-1}, \eta'_{i-k}, \eta_{i-k+1}, \eta_i)$ and $\{\eta_i'\}$ are i.i.d. copies of $\{\eta_i\}.$ Let $\Theta_{k,q}=\sum_{i=k}^{\infty} \theta_{i,q}$ and $\Lambda_{k,q}=\sum_{i=0}^k \theta_{i,q},$ where we use the convention that $\theta_{i,q}=0, i<0.$ 

\begin{lemma}[Concentration inequalities for locally stationary time series]\label{lem_concentration} Let $S_n=\sum_{i=1}^n x_i$ and $q'=\min(2,q).$ Then: \\
(1). We have that for some constant $C>0$ 
\begin{equation*}
\|S_n\|_q^{q'}  \leq C \sum_{i=-n}^{\infty} (\Lambda_{i+n}-\Lambda_{i,q})^{q'}. 
\end{equation*}
Moreover, we have that
\begin{equation*}
\| \max_{1 \leq i \leq n} |S_i| \|_q \leq C n^{1/{q'}} \Theta_{0,q}. 
\end{equation*}
(2). Suppose $x_i=G(i/n, \mathcal{F}_i)$ and $\Theta_{k,q}=\OO(k^{-\gamma}), \gamma>0,$ then we have that
\begin{equation*}
\sum_{i=1}^n \left|\mathbb{E} x_i^2-\sigma\left(\frac{i}{n}\right)\right|=\OO\left(n^{1-\frac{\gamma}{2+\gamma}}\right), 
\end{equation*}
where $\sigma(\cdot)$ is the long run covariance matrix defined as 
\begin{equation*}
\sigma(t)=\sum_{k=-\infty}^{\infty} \operatorname{Cov}\left( G(t, \mathcal{F}_0), G(t, \mathcal{F}_k) \right). 
\end{equation*}
(3). For some constant $C>0,$ we have that 
\begin{equation*}
|\operatorname{Cov}(x_i, x_j)| \leq  C \theta_{|i-j|,q}.
\end{equation*}
\end{lemma}
\begin{proof}
The first part can be found in Theorem 1 of \cite{WUAOP}; also see Lemma 6 of \cite{MR3174655}; the second part of the proof follows directly from Corollary 2 of \cite{MR2827528}; see the last equation in the end of the proof of Corollary 2 therein; the third part can be found in Lemma 6 of \cite{MR3161462} or Lemma 2.6 of \cite{DZ}. 
\end{proof}

\begin{lemma}[$m$-dependent approximation]\label{lem_mdependent} Let $S_n'=\sum_{i=1}^n x_i'$, where $x_i'=\mathbb{E}(x_i| \mathcal{F}_m(i)), \mathcal{F}_m(i)=\sigma(\eta_{i-m}, \cdots, \eta_i)$ is the sigma-algebra generated by  $(\eta_{i-m}, \cdots, \eta_i).$ Here $m \geq 0.$ Define $R_n=S_n-S_n'$ and $R_n^*=\max_{1 \leq i \leq n} |R_i|.$  Then for some constant $C>0,$ we have
\begin{equation*}
\| R_n \|_q^{q'} \leq C n \Theta_{m,q}^{q'}, 
\end{equation*}
and
\begin{equation*}
\| R_n^* \|_q^{q'} \leq C 
\begin{cases}
n \Theta_{m,q}^2, & q>2;\\
n (\log n)^q   \Theta_{m,q}^2, & 1 <q \leq 2. 
\end{cases}
\end{equation*}
\end{lemma}
\begin{proof}
See Lemma A.1 \cite{MR2485027}. 
\end{proof}
In the next lemma, we provide a control for the summation of $\chi_1^2$ random variables. 

\begin{lemma}\label{lem_Gaussianquadratic}
Let $a_1 \geq a_2 \geq \cdots a_p \geq 0$ such that $\sum_{i=1}^p a_i^2=1.$ For i.i.d. $\chi^2_1$ random variables $\eta_i, 1 \leq i \leq p,$ we have that 
\begin{equation*}
\sup_{x \in \mathbb{R}} \mathbb{P} \left( x \leq \sum_{k=1}^p a_k \eta_k  \leq x+h \right) \leq \sqrt{\frac{4 h}{\pi}}.  
\end{equation*}   
\end{lemma}
\begin{proof}
See Lemma S.2 of \cite{XZW}.
\end{proof}

The next lemma provides a deterministic control for the Riemann summation. 
\begin{lemma}\label{lem_intergralappoximation}
Suppose that $f$ is twice differentiable and $f^{''}$ is bounded and almost everywhere continuous on a compact set $[a,b].$ Let $\Delta: a=s_0 \leq s_1 \leq \cdots \leq s_{n-1} \leq s_n=b$ and $s_{i-1} \leq \xi_i \leq s_i.$ Denote the Riemann sum as
\begin{equation*}
\mathcal{R} \equiv \mathcal{R}(f; \Delta, \xi_i):=\sum_{i=1}^n (s_i-s_{i-1}) f(\xi_i).
\end{equation*}
Then we have that for $s_i=a+i(b-a)/n$
\begin{equation*}
\int_a^b f(x) \dd x =\mathcal{R}+\OO(n^{-2}). 
\end{equation*}
\end{lemma}
\begin{proof}
See Theorem 1.1 of \cite{TASAKI2009477}. 
\end{proof}

Then we collect some elementary probability inequalities. 
\begin{lemma}\label{lem_collectionprobineq}
(1). If $\mathbb{E}|X|^k<\infty,$ then for $0<j<k,$ $\mathbb{E}|X|^j<\infty$ and
\begin{equation*}
\mathbb{E}|X|^j \leq (\mathbb{E}|X|^k)^{j/k}. 
\end{equation*}
(2). (Holder's inequality) For $p,q \in [1, \infty)$ with $1/p+1/q=1,$
\begin{equation*}
\mathbb{E}|XY| \leq \| X\|_p \| Y\|_q. 
\end{equation*}
(3). (Chebyshev’s inequality) If $g(x)$ is a  monotonically increasing nonnegative function for the nonnegative reals, then we have that 
\begin{equation*}
\mathbb{P}(|X|>a) \leq \frac{E(g(|X|))}{g(a)}. 
\end{equation*}
(4). (Bernstein’s concentration inequality) Let $\{x_i\}$ be i.i.d. standard Gaussian random variables. Then for every $0<t<1,$ we have that
\begin{equation*}
\mathbb{P}\left(\left|\frac{1}{n} \sum_{k=1}^n x_k^2-1 \right| \geq t \right) \leq 2 e^{-nt^2/8}. 
\end{equation*}
\end{lemma}
\begin{proof}
(1) follows from Exercise 1.6.11 of \cite{probbook}; (2) is Theorem 1.6.3 of \cite{probbook}; (3) follows from Theorem 1.6.4 of \cite{probbook}; (4) is in Example 2.11 of \cite{MR3967104}.  
\end{proof}

The following lemma provides a deterministic control  for the norm of a symmetric matrix. 
\begin{lemma}[Gershgorin circle theorem]\label{lem_circle}
 Let  $A=(a_{ij})$ be a complex $ n\times n$ matrix. For  $1 \leq i \leq n,$ let  $R_{i}=\sum _{{j\neq {i}}}\left|a_{{ij}}\right| $ be the sum of the absolute values of the non-diagonal entries in the  $i$-th row. Let  $ D(a_{ii},R_{i})\subseteq \mathbb {C} $ be a closed disc centered at $a_{ii}$ with radius  $R_{i}$. Such a disc is called a \emph{Gershgorin disc.} The every eigenvalue of $ A=(a_{ij})$ lies within at least one of the Gershgorin discs  $D(a_{ii},R_{i})$, where $R_i=\sum_{j\ne i}|a_{ij}|$.
\end{lemma}
\begin{proof}
See Theorem 7.2.1 of \cite{MR3024913}. 
\end{proof}
Finally, we collect some approximation formulas for the tail probabilities of the maxima of a two-dimensional Gaussian random field in the setting of simutaneous confidence bands. Assuming that 
$Y_i=f(x_i)+\epsilon_i, \ 1 \leq i \leq n, \ x_i \in \mathbb{R}^2$ and $\epsilon_i$ are i.i.d. $\mathcal{N}(0, \sigma^2)$ random variables. Let $\widehat{f}(x)$ be an unbiased estimator of $f(x)$ such that
\begin{equation*}
\widehat{f}(x)=l(x)^\top Y, \ l(x)=(l_1(x), \cdots, l_n(x))^\top, \ Y=(Y_1, \cdots, Y_n)^\top. 
\end{equation*}  
For some constant $c_{\alpha}>0,$ the simultaneous coverage probability of the confidence bands is 
\begin{equation*}
1-\alpha=\mathbb{P}\left(\left|f(x)-\widehat{f}(x) \right| \leq c_{\alpha} \sqrt{\operatorname{Var}(\widehat{f}(x))}, \ x \in \mathcal{X} \right). 
\end{equation*} 
Moreover, according to equation (1.4) of \cite{MR1311978}, we have that
\begin{equation}\label{eq_alphatheortrep}
\alpha=\p \left( \sup_{x \in \mathcal{X}} \left| l(x)^\top \epsilon /\sqrt{\operatorname{Var}(\widehat{f}(x))} \right| \geq c_{\alpha} \sigma \right). 
\end{equation}
\begin{lemma}\label{lem_volumeoftube} Suppose $\mathcal{X}$ is a rectangle in $\mathbb{R}^2.$ Assume the manifold $\mathcal{M}:=\left\{l(x)/\sqrt{\operatorname{Var}(\widehat{f}(x))}: x \in \mathcal{X} \right\}$ is $C^3$ with a positive radius. Let $\kappa_0$ be the area of $\mathcal{M}$ and $\zeta_0$ be the length of the boundary of $\mathcal{M},$ then  (\ref{eq_expansionformula}) holds true. 
\end{lemma}
\begin{proof}
The proof follows from Proposition 2 of \cite{MR1311978} by replacing $\|l(x) \|$ with $\sqrt{\operatorname{Var}(\widehat{f}(x))}.$ 
\end{proof}


%

\section{A brief summary of sieve spaces approximation theory}\label{sec_sieves}
In this section, we give a brief overview on the sieve approximation theory \cite{CXH} for compact domain and the mapped sieve approximation theory for the unbounded domain \cite{MR1874071,MR2486453,unboundeddomain}.  We start with introducing some commonly used sieve basis functions.

\begin{example}[Trigonometric and mapped trigonometric polynomials]\label{examplebasis_fourier} For $x \in [0,1],$ we consider the following trigonometric polynomials 
\begin{equation*}
\left\{1, \sqrt{2} \cos(2 k \pi x), \sqrt{2} \sin (2k\pi x), \cdots \right\}, \ k \in \mathbb{N}.
\end{equation*}
The above basis functions form an orthonormal basis 
of $L_2([0,1]).$ Note that the classic trigonometric basis function is well suited for approximating periodic functions on $[0,1].$ 

When $x \in \mathbb{R},$ let $u(x): \mathbb{R} \rightarrow [0,1]$ be constructed following Example \ref{example_mappings}. Since $u(x)$ in Example \ref{example_mappings} satisfies that $u'(x) \in L_2(\mathbb{R})$, the following mapped trigonometric polynomials form an orthonomal basis for $L_2(\mathbb{R})$ 
\begin{equation*}
\left\{ \sqrt{u'(x)}, \ \sqrt{2 u'(x) } \cos(2 k \pi u(x)), \ \sqrt{2 u'(x)} \sin(2 k \pi u(x)),\cdots \right\}, \ k \in \mathbb{N}, 
\end{equation*}
where we used inverse function theorem. 
\end{example}

\begin{example}[Jacobi and mapped Jacobi polynomials]\label{exam_jacobibasis} For $I=(-1,1)$ and some constants $\alpha, \beta>-1,$ denote the Jacobi weight function as
\begin{equation*}
\omega^{\alpha,\beta}(y)=(1-y)^{\alpha}(1+y)^{\beta},
\end{equation*}
and associated weighted $\mathcal{L}^2$ space as $\mathcal{L}^2_{\omega^{\alpha,\beta}}(I).$ With the above notations, the Jacobi polynomials are the orthogonal polynomials $\{J_n^{\alpha,\beta}(y)\}$ such that 
\begin{equation*}
\int_I J_n^{\alpha, \beta}(y) J_m^{\alpha,\beta}(y) \omega^{\alpha,\beta}(y) \dd y=\gamma_n^{\alpha,\beta} \delta_{n,m}, 
\end{equation*}
where $\delta_{n,m}$ is the Kronecker function, and 
\begin{equation*}
\gamma_n^{\alpha,\beta}=\frac{2^{\alpha+\beta+1} \Gamma(n+\alpha+1) \Gamma(n+\beta+1)}{(2n+\alpha+\beta+1) \Gamma(n+1) \Gamma(n+\alpha+\beta+1)}.
\end{equation*}
More explicitly, the Jacobi polynomials can be characterized using the  Rodrigues' formula \cite{MR0000077} 
\begin{equation*}
J_n^{\alpha,\beta}(x)=\frac{(-1)^n }{2^n n!}(1-x)^{-\alpha}(1+x)^{-\beta} \frac{\dd^n }{\dd x^n} \left( (1-x)^\alpha(1+x)^\beta(1-x^2)^n \right).
\end{equation*}
Then for $x \in [0,1],$ the following sequence forms an orthonormal basis for $L_2([0,1])$ 
\begin{equation*}
\left\{ \mathsf{J}^{\alpha,\beta}_k(x):= \sqrt{\frac{2 \omega^{\alpha, \beta}(2x-1)}{\gamma_k^{\alpha,\beta}} } J_k^{\alpha,\beta}(2x-1)  \right\}, \ k \in \mathbb{N}. 
\end{equation*}
We mention that when $\alpha=\beta=0,$ the Jacobi polynomial reduces to the Legendre polynomial and when $\alpha=\beta=-0.5,$ the Jacobi polynomial reduces to the Chebyshev polynomial of the first kind. 

When $x \in \mathbb{R},$ let $u(x)$ be the mapping as in Example \ref{examplebasis_fourier}. Then the following sequence provides an orthonormal basis for $L_2(\mathbb{R})$ (see \cite[Section 2.3]{unboundeddomain}) 
\begin{equation*}
\left\{ \sqrt{u'(x)}\mathsf{J}_k^{\alpha,\beta}(u(x)) \right\}.
\end{equation*}


\end{example}

%

\begin{example}[Wavelet and mapped Wavelet basis]\label{examplebasis_wave}
For $N \in \mathbb{N},$ a Daubechies (mother) wavelet of class $D-N$ is a function $\psi \in L_2(\mathbb{R})$ defined by 
\begin{equation*}
\psi(x):=\sqrt{2} \sum_{k=1}^{2N-1} (-1)^k h_{2N-1-k} \varphi(2x-k),
\end{equation*}
where $h_0,h_1,\cdots,h_{2N-1} \in \mathbb{R}$ are the constant (high pass) filter coefficients satisfying the conditions
$\sum_{k=0}^{N-1} h_{2k}=\frac{1}{\sqrt{2}}=\sum_{k=0}^{N-1} h_{2k+1},$
as well as, for $l=0,1,\cdots,N-1$
\begin{equation*}
\sum_{k=2l}^{2N-1+2l} h_k h_{k-2l}=
\begin{cases}
1, & l =0 ,\\
0, & l \neq 0.
\end{cases}
\end{equation*} 
And $\varphi(x)$ is the scaling (father) wavelet function is supported on $[0,2N-1)$ and satisfies the recursion equation
$\varphi(x)=\sqrt{2} \sum_{k=0}^{2N-1} h_k \varphi(2x-k),$ as well as the normalization $\int_{\mathbb{R}} \varphi(x) dx=1$ and
$ \int_{\mathbb{R}} \varphi(2x-k) \varphi(2x-l)dx=0, \ k \neq l.$
Note that the filter coefficients can be efficiently computed as listed in \cite{ID92}.  The order $N$, on the one hand, decides the support of our wavelet; on the other hand, provides the regularity condition in the sense that
\begin{equation*}
\int_{\mathbb{R}} x^j \psi(x)dx=0, \ j=0,\cdots,N, \  \text{where} \ N \geq d. 
\end{equation*}
We will employ Daubechies wavelet with a sufficiently high order when forecasting in our simulations and data analysis. In the present paper, to construct a sequence of orthogonal wavelet, we will follow the dyadic construction of \cite{ID98}. For a given $J_n$ and $J_0,$ we will consider the following periodized wavelets on $[0,1]$
\begin{equation}\label{eq_constructone}
\Big\{ \varphi_{J_0 k}(x), \ 0 \leq k \leq 2^{J_0}-1; 
\psi_{jk}(x), \ J_0 \leq j \leq J_n-1, 0 \leq k \leq 2^{j}-1  \Big\},\ \mbox{ where}
\end{equation} 
\begin{equation*}
\varphi_{J_0 k}(x)=2^{J_0/2} \sum_{l \in \mathbb{Z}} \varphi(2^{J_0}x+2^{J_0}l-k) , \
\psi_{j k}(x)=2^{j/2} \sum_{l \in \mathbb{Z}} \psi(2^{j}x+2^{j}l-k),
\end{equation*}
or, equivalently \cite{MR1085487}
\begin{equation}\label{eq_meyerorthogonal}
\Big\{ \varphi_{J_n k}(x), \  0 \leq k \leq  2^{J_n-1} \Big\}.
\end{equation}

When $x \in \mathbb{R},$ let $u(x)$ be the mapping as in Example \ref{examplebasis_fourier}. Then using (\ref{eq_meyerorthogonal}), the following sequence provides an orthonormal basis for $L_2(\mathbb{R})$ 
\begin{equation*}
\left\{ \sqrt{u'(x)}\varphi_{J_n k}(u(x)) \right\}.
\end{equation*}
Similarly, we can construct the mapped wavelet functions using  (\ref{eq_constructone}). 

\end{example}


Next, we collect some important properties of the above basis functions and some useful results on the approximation errors. 

\begin{lemma}\label{lem_deterministicapproximation}
Suppose Assumption \ref{assum_smoothnessasumption} holds. Then for any fixed $x \in \mathbb{R},$  denote 
\begin{equation*}
m_{j,c}(t,x)=\sum_{j=1}^c a_j \phi_j(t), \ 1 \leq j \leq r, 
\end{equation*}
where we assume that $m_j(t,x)=\sum_{j=1}^{\infty} a_j \phi_j(t), \ a_j \equiv a_j(x).$ Then for the basis functions in Examples \ref{examplebasis_fourier}--\ref{examplebasis_wave},  we have that
\begin{equation*}
\sup_t |m_j(t,x)-m_{j,c}(t,x) |=\OO(c^{-\mathsf{m}_1}).
\end{equation*}
\end{lemma}
\begin{proof}
See Section 2.3.1 of \cite{CXH}. 
\end{proof}

\begin{lemma}\label{lem_deterministicapproximation2}
Suppose Assumption \ref{assum_smoothnessasumption} holds. Then for any fixed $t \in [0,1],$  denote 
\begin{equation*}
m_{j,d}(t,x)=\sum_{j=1}^d b_j \varphi_j(x),
\end{equation*}
where we assume that $m_j(t,x)=\sum_{j=1}^{\infty} b_j \varphi_j(x), \ b_j \equiv b_j(t).$ Then for the mapped basis functions in Examples \ref{examplebasis_fourier}--\ref{examplebasis_wave},  we have that
\begin{equation*}
\sup_{x \in \mathbb{R}} |m_{j}(t,x)-m_{j,d}(t,x) |=\OO(d^{-\mathsf{m}_2}).
\end{equation*}
\end{lemma}
\begin{proof}
Recall (\ref{eq_transformedmjty}). Since 
\begin{equation*}
\sup_{x \in \mathbb{R}} |m_{j}(t,x)-m_{j,d}(t,x) |=\sup_{y \in [0,1]}|\widetilde{m}_{j}(t,x)-\widetilde{m}_{j,d}(t,x)|,
\end{equation*}
the proof follows from Section 2.3.1 of \cite{CXH}. Or see Section 6.2 of \cite{MR1176949}. 
\end{proof}

In the following lemma, we provide some controls for the quantities $\xi_c$ and $\zeta$ defined in (\ref{eq_defnxic}).  
\begin{lemma}\label{lem_basisl2norm}
For the basis functions in Example \ref{examplebasis_fourier}, we have that $\xi_c=\OO(1)$ and $\zeta=\OO(\sqrt{p});$ for the basis functions in Example \ref{exam_jacobibasis}, we have that  $\xi_c=\OO(1)$ and $\zeta=\OO(p);$ finally, for the basis functions in Example \ref{examplebasis_wave}, we have that  $\xi_c=\OO(\sqrt{c})$ and $\zeta=\OO(\sqrt{p}).$ 
\end{lemma}
\begin{proof}
See \cite{DZ} or Section 3 of  \cite{MR3343791}. 
\end{proof}

%

\bibliographystyle{abbrv}
\bibliography{lag}

\end{document}